\documentclass[10pt]{amsart}
\usepackage{xypic,amscd,amssymb,latexsym}
\usepackage{amsmath}	
\usepackage{pstricks}
\usepackage{pst-poly}

\usepackage{stmaryrd}

\begin{document}
\pagenumbering{arabic}

\newtheorem{theorem}{Theorem}[section]
\newtheorem{lemma}[theorem]{Lemma}
\newtheorem{proposition}[theorem]{Proposition}
\newtheorem{corollary}[theorem]{Corollary}
\newtheorem{definition}[theorem]{Definition}
\newtheorem{remark}[theorem]{Remark}

\newcommand{\vs}[0]{\vspace{2mm}}

\newcommand{\upker}[4]{
\left\lceil
\begin{array}{cc}
#1 & #2 \\
#3 & #4 \\
\end{array}
\right\rceil
}

\newcommand{\downker}[4]{
\left\lfloor
\begin{array}{cc}
#1 & #2 \\
#3 & #4 \\
\end{array}
\right\rfloor
}

\address{Department of Mathematics, Yale University, 10 Hillhouse ave., New Haven, CT 06511, USA}

\author{Igor B. Frenkel}

\author{Hyun Kyu Kim}

\numberwithin{equation}{section}

\title{Quantum Teichm\"{u}ller space from quantum plane}

\begin{abstract}
We derive the quantum Teichm\"{u}ller space, previously constructed by Kashaev and by Fock and Chekhov, from tensor products of a single canonical representation of the modular double of the quantum plane. We show that the quantum dilogarithm function appears naturally in the decomposition of the tensor square, the quantum mutation operator arises from the tensor cube, the pentagon identity from the tensor fourth power of the canonical representation, and an operator of order three from isomorphisms between canonical representation and its left and right duals. We also show that the quantum universal Teichm\"{u}ller space is realized in the infinite tensor power of the canonical representation naturally indexed by rational numbers including the infinity. This suggests a relation to the same index set in the classification of projective modules over the quantum torus, the unitary counterpart of the quantum plane, and points to a new quantization of the universal Teichm\"{u}ller space.
\end{abstract}

\maketitle

\tableofcontents

\date{May 2010}

\section{Introduction}

It was known for a long time that the Teichm\"{u}ller space of an oriented surface, with possible punctures and boundaries, possesses a canonical Weil-Petersson symplectic form (see \cite{A}, \cite{Wolpert}). But only more recently Kashaev \cite{Kash98} and, independently, Fock and Chekhov \cite{Fo}, \cite{FC} were able to construct a quantization of the classical Teichm\"{u}ller space. This quantization was achieved thanks to an ingenious systems of coordinates, first studied by Thurston \cite{Th} and Penner \cite{Penner}, that are based on ideal triangulations of an oriented surface equipped with a hyperbolic metric. In the Kashaev variant of the Penner coordinates the canonical Poisson structure has a diagonal form and its quantization is given by the Heisenberg algebra
\begin{align}
\label{1_eq:Heisenberg}
[p_j, x_k] = \frac{\delta_{jk}}{2\pi i}, \quad j,k \in I
\end{align}
indexed by the set of ideal triangles $I$. The Fock-Chekhov quantization can be easily derived from Kashaev's via a Heisenberg subalgebra of \eqref{1_eq:Heisenberg} in a certain basis corresponding to quantization of Thurston coordinates (see \cite{GL}). The key point of the Kashaev and Fock-Chekhov quantization is the verification of the consistency of the quantization under the change of the ideal triangulations. In the Kashaev approach this consistency follows from a certain projective representation of the group $G_I$ with generators $a_j, t_{jk}, p_{jk}$, where $j,k$ belong to a finite set $I$, with the relations for all $j,k,\ell \in I$
\begin{align}
\label{1_eq:rel1}
a_j^3 & = e, \\
\label{1_eq:rel2}
a_j t_{jk} a_k & = a_k t_{kj} a_j, \\
\label{1_eq:rel3}
t_{jk} a_j t_{kj} & = a_j a_k p_{jk}, \\
\label{1_eq:rel4}
t_{k\ell} t_{jk} & = t_{jk} t_{j\ell} t_{k\ell},
\end{align}
and $p_{jk}$ satisfies the relations of an involution of $j,k \in I$ in the permutation group of the set $I$ (here $p_{jk}$ is not to be confused with one of the Heisenberg generator $p_j$ as in \eqref{1_eq:Heisenberg}).
The projective representation of $G_I$ can be viewed as a representation of the central extension $\widehat{G}_I$ of $G_I$ with the corresponding generators $\widehat{a}_{j}$, $\widehat{t}_{jk}$, $\widehat{p}_{jk}$ and the central element $z$ satisfying the same relations \eqref{1_eq:rel1}, \eqref{1_eq:rel2}, \eqref{1_eq:rel4}, and a modification of \eqref{1_eq:rel3}:
\begin{align}
\label{1_eq:rel3_modi}
\widehat{t}_{jk} \widehat{a}_j \widehat{t}_{kj}
= z \widehat{a}_j \widehat{a}_k \widehat{p}_{jk}
\end{align}
with $z$ represented by the identity times $\zeta = e^{-2\pi i (b+b^{-1})^2/24}$.

\vs

The main ingredient of the Kashaev representation of the group $\widehat{G}_I$, as well as of the Fock-Chekhov construction, is the remarkable function introduced previously by Faddeev and Kashaev under the name of quantum dilogarithm \cite{FK}. Though this function, in different guises, was known long time ago \cite{B}, \cite{Shin}, it became especially important since the works of Faddeev and Kashaev \cite{FK}, \cite{F} and its central role in quantization of the Teichm\"{u}ller spaces \cite{Kash98}, \cite{Kash99}, \cite{Fo}, \cite{FC}. In fact the main goal of the Faddeev-Kashaev paper \cite{FK} was to quantize the Rogers pentagon identity for the classical dilogarithm function thus providing a representation of the subgroup of $G_I$ with the generators $t_{jk}$, $j,k \in I$, and the relations \eqref{1_eq:rel4}.

\vs

Though the representation of the group $\widehat{G}_I$ via the quantum dilogarithm does provide a quantization of the Teichm\"{u}ller space the formulas involved in the construction look ad hoc and require further elucidation. In particular, it is well known that the pentagon identity arises from an associativity constraint for a tensor category. If the tensor category is also rigid it yields in addition a duality constraint which satisfies some additional relations. This leads to a natural question: is the representation of the group $\widehat{G}_I$ and therefore the quantum Teichm\"{u}ller space originate from a special rigid tensor category? In this paper we show that it is indeed the case and the appropriate rigid tensor category comes from the representation theory of a rather basic Hopf algebra associated to the quantum plane! The latter is the noncompact version of the modular quantum torus and is generated by the four self-adjoint operators $X,Y,\widetilde{X}, \widetilde{Y}$ satisfying the following relations (cf. \cite{F})
\begin{align}
XY = q^2 YX, \quad
\widetilde{X}\widetilde{Y} = \widetilde{q}^{2} \widetilde{Y}\widetilde{X},
\end{align}
where
\begin{align}
q = e^{\pi ib^2}, \quad \widetilde{q} = e^{\pi i b^{-2}},
\quad b\in (0,\infty), ~ b^2\notin \mathbb{Q}.
\end{align}
This algebra, denoted by $B_{q\widetilde{q}}$ in this paper, has one natural representation $\pi$ in the representation space $\mathcal{H}$ of the Heisenberg algebra
\begin{align}
[p,x] = \frac{1}{2\pi i}
\end{align}
via
\begin{align}
\pi (X) = e^{-2\pi b p}, \quad
\pi(Y) = e^{2\pi b x},\quad
\widetilde{X} = e^{-2\pi b^{-1}p}, \quad
\widetilde{Y} = e^{2\pi b^{-1}x}.
\end{align}
To ensure a nontrivial tensor category we define the comultiplication as follows
\begin{align}
\label{1_eq:coproduct}
\Delta X = X\otimes X, \quad
\Delta Y = Y\otimes X + 1\otimes Y, \quad
\Delta \widetilde{X} = \widetilde{X}\otimes \widetilde{X}, \quad
\Delta \widetilde{Y} = \widetilde{Y}\otimes \widetilde{X} + 1 \otimes \widetilde{Y}.
\end{align}
From \eqref{1_eq:coproduct} and the counit $\epsilon(X)=1=\epsilon(\widetilde{X})$, $\epsilon(Y) = 0 = \epsilon(\widetilde{Y})$, the antipode $S$ is uniquely determined, yielding a Hopf algebra structure on $B_{q\widetilde{q}}$, which can be viewed as the Borel subalgebra of the modular double of $\mathcal{U}_q(\mathfrak{sl}(2,\mathbb{R}))$.

\vs

The key property of $\mathcal{H}$ crucial to our construction is that its tensor square is decomposed into a direct integral of irreducible representations equivalent to $\mathcal{H}$. As a result we obtain an isomorphism
\begin{align}
\label{1_eq:tensor_square_isomorphism}
\mathcal{H} \otimes \mathcal{H} \cong M \otimes \mathcal{H},
\end{align}
where $M\cong Hom_{B_{q\widetilde{q}}} (\mathcal{H}, \mathcal{H} \otimes \mathcal{H})$ is the ``multiplicity'' space. After an identification of $\mathcal{H}$ and $M$ with $L^2(\mathbb{R})$ we can express \eqref{1_eq:tensor_square_isomorphism} by a certain integral transform based on the quantum dilogarithm function. Then the canonical isomorphism
\begin{align}
\label{1_eq:isomorphism_triple}
(\mathcal{H}_1 \otimes \mathcal{H}_2) \otimes \mathcal{H}_3
\cong
\mathcal{H}_1 \otimes (\mathcal{H}_2 \otimes \mathcal{H}_3)
\end{align}
yields an operator
\begin{align}
\label{1_eq:T_operator}
{\bf T}: M_{43}^6 \otimes M_{12}^4 \stackrel{\sim}{\longrightarrow} M_{23}^5 \otimes M_{15}^6,
\end{align}
which is often referred to as `quantum mutation operator', and it can be explicitly identified using the realization $M\cong L^2(\mathbb{R})$. By construction, the operator ${\bf T}$ satisfies the pentagon identity \eqref{1_eq:rel4}.

\vs

Another important property of $\mathcal{H}$ that we use is isomorphism with the dual representations
\begin{align}
\label{1_eq:isomorphism_with_dual}
\mathcal{H}' \cong \mathcal{H} \cong {}' \mathcal{H}
\end{align}
where $\mathcal{H}'$ (resp. ${}'\mathcal{H}$) denotes the dual space with the action of the quantum plane algebra defined to be dual to $\pi \circ S$ (resp. $\pi \circ S^{-1}$). The pairing gives the intertwining operators
\begin{align}
\label{1_eq:pairing_intertwiner}
\mathcal{H}' \otimes \mathcal{H} \to \mathbb{C}, \quad
\mathcal{H} \otimes {}' \mathcal{H} \to \mathbb{C}.
\end{align}
Then the composition of the operators \eqref{1_eq:isomorphism_with_dual} and \eqref{1_eq:pairing_intertwiner} provides the canonical isomorphisms
\begin{align}
\label{1_eq:isomorphism_Hom_Inv}
Hom_{B_{q\widetilde{q}}}(\mathcal{H}_3, \mathcal{H}_1 \otimes \mathcal{H}_2) \cong Inv(\mathcal{H}_1 \otimes \mathcal{H}_2 \otimes \mathcal{H}_3') \cong Inv({}'\mathcal{H}_1 \otimes \mathcal{H}_2 \otimes \mathcal{H}_3) \cong Hom_{B_{q\widetilde{q}}}(\mathcal{H}_1, \mathcal{H}_2 \otimes \mathcal{H}_3)
\end{align}
and yields another operator
\begin{align}
\label{1_eq:A_operator}
{\bf A}: M_{12}^3 \stackrel{\sim}{\longrightarrow} M_{23}^1
\end{align}
which after realizations of $M_{12}^3$ and $M_{23}^1$ becomes an explicit operator in $L^2(\mathbb{R})$. The natural requirement that ${}' \mathcal{H}' \cong \mathcal{H}$ is the identity map implies that ${\bf A}^3 = 1$ representing \eqref{1_eq:rel1}. One can also deduce two other relations \eqref{1_eq:rel2} and \eqref{1_eq:rel3_modi} that involve both operators ${\bf T}$ and ${\bf A}$. Thus tensor products of a single representation of the quantum plane and its equivalent dual yield a faithful representation of the group $\widehat{G}_I$.

\vs

To compare our representation of the group $\widehat{G}_I$ with the original Kashaev projective representation we embed our operator ${\bf A}$ into a one-parameter family of operators ${\bf A}^{(m)}$ that differ from ${\bf A}$ by the factor $e^{\pi (m-1)(b+b^{-1})p}$ so that ${\bf A} \equiv {\bf A}^{(1)}$. We show that for any $m\in \mathbb{R}$, the pair of operators ${\bf T}$, ${\bf A}^{(m)}$ provides a representation of the group $\widehat{G}_I$ with the central element $z$ represented by the identity times $\zeta^{1-m^2}$. Note that the Kashaev original representation of $\widehat{G}_I$ by certain explicit operators $\widetilde{\bf T}_{jk}$, $\widetilde{\bf A}_j$, $j,k \in I$, yields the central extension corresponding to $m=0$. One of the main results of this paper is the equivalence of our representation of the group $\widehat{G}_I$ realized by ${\bf T}_{jk}$, ${\bf A}_j^{(0)}$, $j,k \in I$ and the Kashaev representation. More explicitly, we find a unitary operator $U$ in $L^2(\mathbb{R})$ such that
\begin{align}
(U^{-1} \otimes U^{-1}) {\bf T}_{jk} (U \otimes U) = \widetilde{{\bf T}}_{jk}, \quad U^{-1} {\bf A}_j^{(0)} U = \widetilde{{\bf A}}_j,
\quad i, j \in I ~(i\neq j).
\end{align}
Considering unitary transformation of the one-parameter family $U^{-1} {\bf A}^{(m)}_j U = \widetilde{{\bf A}}^{(m)}_j$ we are able to extend the Kashaev representation for all values $m\in \mathbb{R}$, of the central extension.

\vs

As soon as we have a representation of $\widehat{G}_I$ we can apply it to quantization of Teichm\"{u}ller space for various surfaces as it was done in the original work  \cite{Kash98}, \cite{Fo}, \cite{FC}. In this paper we consider just a simple example of a disk with $n$ distinguished points on the boundary to illustrate an advantage of a representation theoretic approach. Though our quantization of the Teichm\"{u}ller space is equivalent to the one of Kashaev we do not need to use an initial triangulation of the surface; the independence of the quantum Teichm\"{u}ller space on a triangulation is built into our construction.

\vs

Moreover, our construction can be generalized to higher rank quantum algebras that are the modular doubles of the Borel subalgebras of quantum groups $\mathcal{U}_q \mathfrak{g}$ and can be applied to quantization of higher Teichm\"{u}ller spaces, previously constructed in \cite{FG}.

\vs

We conclude the paper with the quantization of the universal Teichm\"{u}ller space which can be viewed as a limit of the example of a disk with $n$ distinguished points on the boundary when $n$ tends to infinity, and we obtain
\begin{align}
Inv\left( \otimes_{r\in \widehat{\mathbb{Q}}} \mathcal{H}_r\right), \quad \mathcal{H}_r \cong \mathcal{H}, \quad \forall r.
\end{align}
It is interesting to note that $\widehat{\mathbb{Q}} = \mathbb{Q} \cup \{\infty\} \cong \Gamma_\infty \backslash PSL_2(\mathbb{Z})$, where $\Gamma_\infty$ is the subgroup of upper triangular matrices, also appears as a natural index set in the classification of the projective modules for the quantum torus, which can be viewed as a unitary counterpart of the quantum plane. This observation leads us to conjecture a new quantization of the universal Teichm\"{u}ller space, discussed at the end of our paper.

\vs

\noindent{\bf Acknowledgments.} H. K. would like to thank Ivan Ip for helpful discussions. The research of I. F. was supported by NSF grant DMS-0457444.

\section{Quantum dilogarithm}

\subsection{Introduction to quantum dilogarithm}
\label{2_subsection:intro_quantum_dilogarithm}

We mainly follow \cite{Rui97} and \cite{Rui05} for technical reference for this section; readers should consult those two papers for any omitted details. Fix $b\in (0,\infty)$ with $b^2\notin \mathbb{Q}$; we will use the following symbols frequently throughout this paper:
\begin{align*}
q = e^{\pi i b^2}, \quad
\widetilde{q} = e^{\pi i b^{-2}}, \quad
a = (b+b^{-1})/2, \quad
\chi = \frac{\pi}{24}(b^2 + b^{-2}), \quad
\zeta = e^{-\pi i a^2/3}.
\end{align*}
For $z\in \mathbb{C}$, the symbols $\Re z$ and $\Im z$ will mean real part and imaginary part of $z$, respectively. Consider the function $G$ defined as
\begin{align}
\label{2_eq:definition_G}
G(z)
= \exp\left(
i \int_0^\infty \frac{dy}{y} \left(
\frac{\sin (2y z)}{2 \sinh (b y) \sinh (b^{-1} y)}
- \frac{z}{y} \right)
\right),
\quad |\Im z| < a.
\end{align}
In the strip $|\Im z| < a$, the function $G$ is a well-defined (it's not difficult to prove the convergence of the integral) analytic function with no zeros in the strip, satisfying the functional equations
\begin{align}
\label{2_eq:G_defining_relations}
\frac{G(z+ib^{\pm 1}/2)}{G(z-ib^{\pm 1}/2)} = 2\cosh(\pi b^{\pm 1}z).
\end{align}
By using \eqref{2_eq:G_defining_relations}, $G$ can be analytically continued to the entire complex plane to a meromorphic function with simple zeros at
$$
z_{kl} \equiv ia + ikb + il b^{-1}, \quad k,l \in \mathbb{N} = \{0,1,2,\ldots\}
$$
and simple poles at $-z_{kl}$. This meromorphic function $G$ is actually the unique ``minimal'' solution of \eqref{2_eq:G_defining_relations} with $G(0)=1$; here ``minimal'' means $\ln G(z)$ is polynomially bounded in $|\Im z| < \max(b,b^{-1})/2$. , i.e.  $\exists c,d>0$ and $N\in \mathbb{N}$ s.t.
$$
|\ln G(z)| < c + d|z|^k, \quad \forall z\in \{|\Im z| \le \max(b,b^{-1})/2\}.
$$
It is easy to see from \eqref{2_eq:definition_G} that $|G(x)|=1$ for all $x\in\mathbb{R}$, and
\begin{align*}
G(z)G(-z)=1.
\end{align*}
Asymptotic behavior of $G$ is
\begin{align}
\label{2_eq:G_asymptotics}
G(z) e^{\pm (i \chi + \frac{\pi i}{2}z^2)} = 1 + O(e^{-\rho|\Re z|}), \quad \Re z \to \pm \infty, \quad \rho<2\pi \min(b,b^{-1}),
\end{align}
where the implied constant can be chosen uniformly for $\Im z$ varying over compact subsets of $\mathbb{R}$.
Note that the defining relations \eqref{2_eq:G_defining_relations} can be written as 
\begin{align}
\label{2_eq:G_defining_relations2}
G(z+ib^{\pm 1}) = (e^{\pi i b^{\pm 2}/2} e^{\pi b^{\pm 1}z} + e^{-\pi i b^{\pm 2}/2} e^{-\pi b^{\pm 1}z}) G(z).
\end{align}

In this paper, we will take advantage of the following two functions:
\begin{align}
\label{2_eq:S_R_def}
S_R(z) & := G(z - ia)
e^{i \chi + \frac{\pi i}{2} (z-ia)^2}, \\
\label{2_eq:S_L_def}
S_L(z) & := G(z - ia) e^{-i\chi -\frac{\pi i}{2} (z-ia)^2}.
\end{align}
We can find out from \eqref{2_eq:G_defining_relations2} that the defining relations for $S_R(z)$ and $S_L(z)$ are
\begin{align}
\label{2_eq:S_R_S_L_relations}
S_R(z+ib^{\pm 1})
= (1 - e^{-2\pi b^{\pm 1}z})S_R(z), \quad
S_L(z+ib^{\pm 1})
= (1 - e^{2\pi b^{\pm 1}z})S_L(z).
\end{align}
Let $p = \frac{1}{2\pi i} \frac{d}{dz}$, so that $e^{-2\pi b^{\pm 1}p}$ is the shift operator (by the amount $ib^{\pm 1}$), i.e. $e^{-2\pi b^{\pm 1}p} f(z) = f(z+ib^{\pm 1})$. Using this, the above relations for $S_R$ and $S_L$ can be written as 
\begin{align*}
(e^{-2\pi b^{\pm 1}p} + e^{-2\pi b^{\pm 1}z}) S_R = S_R, \quad
(e^{-2\pi b^{\pm 1}p} + e^{2\pi b^{\pm 1}z}) S_L = S_L.
\end{align*}
Observe that the Fourier transform maps $z$ to $-p$ and $p$ to $z$. Thus, it sends the defining relations of $S_L$ exactly to those of $S_R$. By uniqueness, the Fourier transform of $S_L$ has to coincide with $S_R$ up to a constant. Ruijsenaars \cite{Rui05} made precise sense of this:

\begin{proposition}\label{2_prop:G_Fourier} One has
\begin{align}
\label{2_eq:G_Fourier}
\int_\mathbb{R} e^{2\pi i z w} S_R(w) d\Re w
= e^{- 2i\chi} e^{-\pi i/4} S_L(z),
\qquad \Im z, \Im w>0, \quad
\Im z + \Im w < a.
\end{align}
\end{proposition}

By taking suitable limits of $\Im z$ and $\Im w$ in \eqref{2_eq:G_Fourier}, we obtain some specializations of \eqref{2_eq:G_Fourier}, which make sense as functionals on $W\subset L^2(\mathbb{R},dx)$, the space of finite $\mathbb{C}$-linear combinations of the functions
\begin{align}
\label{2_eq:def_W}
e^{-Ax^2/2 + Bx} P(x), 
\quad
\mbox{where $P(x)$ is a polynomial in $x$, and $A\in \mathbb{R}_{>0}$, $B\in \mathbb{C}$.}
\end{align}
The elements of $W$ rapidly decrease at $\pm \infty$, and are analytic; moreover, $W$ is closed under the Fourier transform.
When dealing with integrals whose integrands have poles on the real line, we will modify the contour of integration near the poles, so that the contour takes a detour around the pole via a small half circle. To specify whether the small half circle is located above or below the pole, we introduce the notation $\int_\mathbb{R} \sim du^{\cap v}$, which means the modified contour avoids the pole $v$ via a small half circle above $v$ (it could've been also denoted by $\int_{\Omega_v} \sim du$); similarly, the symbol $du^{\cup v}$ will mean that the contour is modified with small half circle below $v$, throughout the paper. When there are multiple poles on the real line, we use combination of these symbols, e.g. $\int_\mathbb{R} \sim du^{\cap v_1 \cup v_2}$.

\begin{corollary}\label{2_cor:Fourier_transform} One has
\begin{align}
\label{2_eq:G_Four_+_+}
\int_\mathbb{R} e^{2\pi i xu} G (u-ia) e^{\frac{\pi i}{2} u^2} e^{\pi a u} du^{\cap 0} & = e^{2i\chi} e^{\pi i /4} G(x-ia) e^{-\frac{\pi i}{2}x^2} e^{-\pi ax}, \\
\label{2_eq:G_Four_-_-}
\int_\mathbb{R} e^{-2\pi i xu} G(x-ia) e^{-\frac{\pi i }{2} x^2} e^{-\pi a x} dx^{\cap 0} & = e^{-2i\chi}e^{-\pi i/4} G(u-ia) e^{\frac{\pi i}{2}u^2} e^{\pi a u}, \\
\label{2_eq:G_Four_+_-}
\int_\mathbb{R} e^{2\pi i u x} G(u-ia) e^{\frac{\pi i}{2}u^2} e^{-\pi a u} du^{\cap 0} & = e^{-i\chi} G(x) e^{-\frac{\pi i}{2} x^2}, \\
\label{2_eq:G_Four_-_0}
\int_\mathbb{R} G(x) e^{-\frac{\pi i}{2}x^2} e^{-2\pi i xu} dx & = e^{i\chi} G(u-ia) e^{\frac{\pi i}{2}u^2} e^{-\pi a u},
\end{align}
as functionals on $W$, in the following sense: when evaluating on elements of $W$, we always avoid the pole $x=0$ of $G(x-ia)$ from above; for example, \eqref{2_eq:G_Four_+_+} is in the sense of
\begin{align}
\label{2_eq:G_Four_+_+_evaluated}
\int_{\mathbb{R}} \left( \int_\mathbb{R} e^{2\pi i xu} G (u-ia) e^{\frac{\pi i}{2} u^2} e^{\pi a u} f(x) dx\right) du^{\cap 0} = \int_\mathbb{R} e^{2i\chi} e^{\pi i /4} G(x-ia) e^{-\frac{\pi i}{2}x^2} e^{-\pi ax} f(x) dx^{\cap 0},
\end{align}
for functions $f(x) \in W$.
\end{corollary}

\begin{proof}
By using \eqref{2_eq:G_asymptotics}, the dominated convergence theorem (applied as $\Im w \searrow 0$) allows us to obtain \eqref{2_eq:G_Fourier} (equality as functions) for the case $\Im w = 0$ (then we need $0<\Im z<a$). Then, we multiply $f(z) = e^{-Az^2 + Bz} P(z)$ to both sides of \eqref{2_eq:G_Fourier} (for $\Im w=0$) and integrate w.r.t. $\Re z$. Using \eqref{2_eq:G_asymptotics} and the decaying property of $f$, we can prove that the integrands of LHS and RHS are both integrable. Hence we can use Fubini's theorem for LHS, to switch the order of integration ($d\Re z$ and $du^{\cap 0}$). Now, modify the contour on the RHS to $(d\Re z)^{\cap 0}$; this is possible because we don't hit the pole of $G$ by doing so. Using \eqref{2_eq:G_asymptotics} and the decaying property of $f$ and $\mathcal{F}^{-1}_{\Re z} f$ (inverse Fourier transform of $f$ w.r.t. $\Re z$), we can show that it's possible to use the dominated convergence theorem as $\Im z \searrow 0$, which yields \eqref{2_eq:G_Four_+_+_evaluated}. Note that the contour for RHS shouldn't be modified to $(d\Re z)^{\cup 0}$ instead, because then the contour will hit the pole $z=0$ of $G(z-ia)$ either when we modify the contour from $d\Re z$ to $(d\Re z)^{\cup 0}$, or when we take the limit $\Im z \searrow 0$.

\vs

Then, \eqref{2_eq:G_Four_-_-} is obtained by Fourier transform applied to \eqref{2_eq:G_Four_+_+}; putting the Fourier transform of $f(x) = e^{-Ax^2+Bx}P(x)$ (where $P$ is a polynomial) into the place of $f$ in \eqref{2_eq:G_Four_+_+_evaluated} yields the result:
\begin{align*}
\int_\mathbb{R} \int_\mathbb{R} e^{-2\pi i xu} G(x-ia) e^{-\frac{\pi i }{2} x^2} e^{-\pi a x} f(u) du dx^{\cap 0} = \int_\mathbb{R} e^{-2i\chi}e^{-\pi i/4} G(u-ia) e^{\frac{\pi i}{2}u^2} e^{\pi a u} f(u) du^{\cap 0}.
\end{align*}
For \eqref{2_eq:G_Four_+_-}: as done above, taking the limit $\Im w\searrow 0$ for \eqref{2_eq:G_Fourier} can be performed without difficulty. We then proceed as we did to obtain \eqref{2_eq:G_Four_+_+_evaluated}, to get
\begin{align}
\label{2_eq:G_Four_+_-_evaluated}
\int_\mathbb{R}\int_\mathbb{R} e^{2\pi i u x} G(u-ia) e^{\frac{\pi i}{2}u^2} e^{-\pi a u} f(x) dx du^{\cap 0} = \int_\mathbb{R} e^{-i\chi} G(x) e^{-\frac{\pi i}{2} x^2} f(x) dx,
\end{align}
for $f(x) = e^{-Ax^2+Bx}P(x)$. Putting the Fourier transform of $f$ into the place of $f$ in \eqref{2_eq:G_Four_+_-_evaluated} yields the result for \eqref{2_eq:G_Four_-_0}:
\begin{align*}
\int_\mathbb{R}\int_\mathbb{R} G(x) e^{-\frac{\pi i}{2}x^2} e^{-2\pi i xu} f(u) du dx = \int_\mathbb{R} e^{i\chi} G(u-ia) e^{\frac{\pi i}{2}u^2} e^{-\pi a u} f(u) du^{\cap 0}.
\end{align*}
\end{proof}

\begin{remark}
For the moment, we made sense of Corollary \ref{2_cor:Fourier_transform} only as functionals on $W$ but not as  distributions on usual Schwartz space, because we need the test functions to have analytic continuation, since we deal with integrals along contours which are not just real lines. However, later in this paper, we will be more general and will treat distributions on a newly defined Schwartz space associated to a $*$-algebra of interest, which replaces the usual Schwartz space to suit better for our purposes. See \S\ref{3_subsection:quantum_plane_modular_double}  and \S\ref{4_subsection:intertwining_operator}.
\end{remark}

Woronowicz \cite{Woro} proves a formula equivalent to another specialization obtained by taking limit as $\Im w \nearrow a$ and $\Im z\searrow 0$. 

\vs

Volkov \cite{V} asserts that the ``tau-binomial theorem'' can be proved similarly as the Fourier transform formula \eqref{2_eq:G_Fourier}; one specialization of the tau-binomial formula is as follows.

\begin{proposition} One has
\label{2_prop:tau_binomial}
\begin{align}
\label{2_eq:G_new_tau_binomial_prime}
{\renewcommand\arraystretch{2}
\begin{array}{l}
\displaystyle
\int_\mathbb{R} G (z-v-ia) G (y-z-ia) e^{-i\pi z( y-v + 2w)} e^{-2\pi az} dz^{\cup y \cap v} \\
\displaystyle
= \frac{G(y-v-ia)G(w-ia)}{G(y-v+w-ia)} e^{-\pi i (y+v)w} e^{\frac{i\pi}{2} (-y^2+v^2)} e^{- \pi a (y+v)},
\end{array}
}
\end{align}
as functionals on $W$; to be precise, \eqref{2_eq:G_new_tau_binomial_prime} means
\begin{align*}
{\renewcommand\arraystretch{2}
\begin{array}{l}
\displaystyle
\int_\mathbb{R} \left( \int_\mathbb{R} G(y-v+w-ia) G (z-v-ia) G (y-z-ia) e^{-i\pi z( y-v + 2w)} e^{-2\pi az} f(w) dw^{\cap (v-y)} \right) dz^{\cup y \cap v} \\
\displaystyle
= \int_\mathbb{R} G(y-v-ia)G(w-ia) e^{-\pi i (y+v)w} e^{\frac{i\pi}{2} (-y^2+v^2)} e^{- \pi a (y+v)} f(w) dw^{\cap 0},
\end{array}
}
\end{align*}
for $f\in W$.
\end{proposition}

\vs

Volkov \cite{V} also asserts that sending $w$ to $p-y$ in \eqref{2_eq:G_new_tau_binomial_prime} yields:
\begin{proposition}
\label{2_prop:G_delta}
One has
\begin{align}
\label{2_eq:G_delta1}
\int_\mathbb{R} G(x-p-ia) G (y-x-ia) e^{-\pi i x(y-p)} e^{2\pi a x} dx^{\cup y \cap p}
& = \delta(y-p) e^{\pi a (p+y)}, \\
\label{2_eq:G_delta2}
\int_\mathbb{R} G (x-p-ia) G(y-x-ia) e^{\pi i x(y-p)} e^{-2\pi a x} dx^{\cup y \cap p}
& = \delta(y-p) e^{-\pi a (p+y)},
\end{align}
as functionals on $W$; \eqref{2_eq:G_delta2} is in the sense of
$$
\int_\mathbb{R} \int_\mathbb{R} G (x-p-ia) G(y-x-ia) e^{\pi i x(y-p)} e^{-2\pi a x} f(y) dy^{\cap x} dx^{\cap p} = e^{-\pi a (p+y)} f(p),
$$
for $f\in W$, and similarly for \eqref{2_eq:G_delta1}.
\end{proposition}
Ruijsenaars \cite{Rui05} proves Proposition \ref{2_prop:G_delta} without using \eqref{2_eq:G_new_tau_binomial_prime}; see Corollary \ref{2_cor:E_orthogonality} in the next subsection. Different people use different versions of the quantum dilogarithm; see Appendix A of \cite{Rui05} for details.


\subsection{Quantum dilogarithm integral transformation}

Define
\begin{align}
\label{2_eq:E_R_def}
& \mathcal{E}_R(z,w) \equiv e^{2\pi iz w} S_R(z-w), \\
\label{2_eq:E_L_def}
& \mathcal{E}_L(z,w) \equiv e^{-2\pi i z w} S_L(z-w),
\end{align}
where $S_R$ and $S_L$ are as defined in \eqref{2_eq:S_R_def} and \eqref{2_eq:S_L_def}.  Then it's easy to check
\begin{align}
\label{2_eq:E_R_relation}
e^{2\pi b^{\pm 1} p_y} \mathcal{E}_R(x,y)
&= (e^{2\pi b^{\pm 1} x} - e^{2\pi b^{\pm 1} y}) \mathcal{E}_R(x,y), \\
\label{2_eq:E_L_relation}
e^{-2\pi b^{\pm 1} p_y} \mathcal{E}_L(y,x)
& = (e^{2\pi b^{\pm 1} x} - e^{2\pi b^{\pm 1} y}) \mathcal{E}_L(y,x),
\end{align}
where $p_y = \frac{1}{2\pi i} \frac{d}{dy}$. As in eq (3.36) of \cite{Rui05}, define the integral transformations $\mathcal{T}_R$ and $\mathcal{T}_L$ by
\begin{align}
\label{2_eq:T_R_T_L_def}
\mathcal{T}_\sigma : W\subset L^2(\mathbb{R}, dy) \to L^2(\mathbb{R}, dx), \quad
\phi(y) \mapsto \int_\mathbb{R} \mathcal{E}_\sigma (x, y) \phi(y) dy^{\cup x}, \quad
\sigma = R,L,
\end{align}
where $W$ is as in \eqref{2_eq:def_W}. The integrals are well-defined due to asymptotic property \eqref{2_eq:G_asymptotics} and decaying property of $\phi\in W$.

\vs

A part of principal result of \cite{Rui05} (Theorem 4.3 there) is:

\begin{theorem}
\label{2_thm:T_R_T_L}
The transforms $\mathcal{T}_R$ and $\mathcal{T}_L$ are unitary operators related by
$$
\mathcal{T}_L = \mathcal{T}_R^*.
$$
\end{theorem}

\begin{corollary}\label{2_cor:E_orthogonality}
One has
\begin{align}
\label{2_eq:E_orthogonality}
\int_\mathbb{R} \mathcal{E}_L (y,x) \mathcal{E}_R (x,w) dx^{\cap w \cup y} = \delta(w-y)
= \int_\mathbb{R} \mathcal{E}_R(y,x) \mathcal{E}_L(x,w) dx^{\cap w\cup y},
\end{align}
as functionals on $W$, in the sense that
\begin{align}
\label{2_eq:E_orthogonality_precise_first}
\int_\mathbb{R} \left( \int_\mathbb{R} \mathcal{E}_L (y,x) \mathcal{E}_R (x,w) \phi(w) dw^{\cup x} \right) dx^{\cup y} = \phi(y)
\end{align}
for $\phi \in W$, and similarly for the second equality of \eqref{2_eq:E_orthogonality}.
\end{corollary}

We will study integral transformations slightly modified from $\mathcal{T}_L, \mathcal{T}_R$.

\section{Modular double of quantum plane Hopf algebra}

\subsection{The quantum plane}
\label{3_subsection:quantum_plane}

The quantum plane, the principal object of our study, is the free algebra over $\mathbb{C}$ generated by two elements $X$ and $Y$ with relation $XY = q^2 YX$, where
\begin{align}
\label{3_eq:def_q}
q = e^{\pi i b^2}, \quad
b\in (0,\infty), ~b^2 \notin \mathbb{Q},
\end{align}
equipped with a $*$-structure given by $X^*= X$ and $Y^*=Y$. As explained in \cite{Sc}, this algebra can be viewed as the coordinate algebra of the `quantum two-dimensional real vector space'. Meanwhile, its unitary counterpart, `the quantum torus', has the $*$-structure given by $X^*=X^{-1}$ and $Y^* = Y^{-1}$, and can be viewed as the coordinate algebra of the `quantum two-dimensional torus'.

\vs

Recall that a {\it Hopf $*$-algebra} is a Hopf algebra $B$ over $\mathbb{C}$ equipped with a $*$-algebra structure which is compatible with the coalgebra structure, i.e. the comultiplication and the counit are $*$-homomorphisms. A $*$-structure on $B$ is a conjugate-linear map $*:B\to B$ such that $(u^*)^*=u$, $(u_1 u_2)^* = u_2^* u_1^*$, $1^*=1$.
\begin{definition}
\label{3_def:quantum_plane}
(quantum plane)
The Hopf $*$-algebra $B_{q}$ is defined as:
\begin{align*}
{\renewcommand\arraystretch{1.3} 
\begin{array}{rl}
\mbox{generators}: & X^{\pm 1},~ Y \\
\mbox{relations}: & XY = q^{2} YX \\
\mbox{star-structure}: & X^* = X, \quad Y^* = Y \\
\mbox{coproduct}: & \Delta X = X \otimes X, \quad
\Delta Y = Y \otimes X + 1\otimes Y \\
\mbox{antipode}: & S(X) = X^{-1}, \quad S(Y) = -YX^{-1} \\
\mbox{counit}: & 
\epsilon(X) = 1, \quad
\epsilon(Y) = 0
\end{array}}
\end{align*}
and $\mathcal{B}_q$ is defined as its subalgebra generated by $X,Y$.
\end{definition}

`$B$' stands for the `Borel subalgebra' of $\mathcal{U}_q(\mathfrak{sl}(2,\mathbb{R}))$.

\begin{proposition}
\label{3_prop:hopf_algebra}
The algebra $B_q$ indeed satisfies the Hopf $*$-algebra axioms.
\end{proposition}

\vs

We now study a class of operator representations of $\mathcal{B}_q$. A class of ``well-behaved'' representations of $\mathcal{B}_q$ by operators on a Hilbert space is classified in \cite{Sc}; they are named ``integrable'' representations. If we further require that the operators representing $X$ and $Y$ be positive definite, then there exists a unique irreducible representation having these properties up to unitary equivalence; this representation is given by the actions
\begin{align}
\label{3_eq:def_repres_B_q}
\pi(X) = e^{-2\pi b p}, \quad \pi (Y) = e^{2\pi bx},
\end{align}
densely defined on $L^2(\mathbb{R})$, where $p = \frac{1}{2\pi i} \frac{d}{dx}$ (so $p,x$ form a Heisenberg pair: $[p,x] = \frac{1}{2\pi i}$), so that $(\pi(X)f)(x) = f(x+ib)$.

\vs

For the treatment about domain of $\pi$, we follow Goncharov's approach \cite{Goncharov}. We first let $\mathcal{B}_q$ act via $\pi$ on the dense subspace $W$ of $L^2(\mathbb{R},dx)$, as defined in \eqref{2_eq:def_W}.
Then $\pi(X), \pi(Y)$ are symmetric unbounded operators on $W$. A quick summary of results of \cite{Sc} is as follows:
\begin{proposition}
\label{3_prop:integrable_representation}
The pair $(\pi(X), \pi(Y))$ of positive definite operators given in \eqref{3_eq:def_repres_B_q} defines an irreducible ``integrable'' operator representation of (the $*$-algebra) $\mathcal{B}_q$, i.e. $\exists k \in \mathbb{Z}$ s.t. for any $t,s\in \mathbb{R}$, the unitary operators $\pi(X)^{it}, \pi(Y)^{is}$ satisfy
$
\pi(X)^{it} \pi(Y)^{is} = e^{i(-2\pi b^2 +2\pi k)ts} \pi(Y)^{is} \pi(X)^{it},
$
and $\pi(X)$, $\pi(Y)$ have self-adjoint extensions in $L^2(\mathbb{R})$.
Furthermore, such $\pi$ is uniquely determined up to unitary equivalence.
\end{proposition}

\subsection{The modular double of the quantum plane}
\label{3_subsection:quantum_plane_modular_double}

Let $q = e^{\pi i b^2}$, $b>0$, $b^2\notin \mathbb{Q}$, as in \eqref{3_eq:def_q}. For the special complex number
$
\widetilde{q} = e^{\pi i b^{-2}},
$
the two quantum plane algebras $B_q$ and $B_{\widetilde{q}}$ can be put together nicely to form `the modular double', the notion first introduced by Faddeev (\cite{F}) :

\begin{definition}\label{3_def:modular_double_quantum_plane}
(modular double of quantum plane)
The Hopf $*$-algebra $B_{q\widetilde{q}}$ is defined as:
\begin{align*}
{\renewcommand\arraystretch{1.3} 
\begin{array}{rl}
\mbox{generators}: &  X^{\pm 1},~ Y,~ \widetilde{X}^{\pm 1},~ \widetilde{Y} \\
\mbox{relations}: & XY = q^{2} YX, \quad \widetilde{X} \widetilde{Y} = \widetilde{q}^{2} \widetilde{Y} \widetilde{X}, \\
& [X,\widetilde{X}]
= [X,\widetilde{Y}]
= [Y,\widetilde{X}]
= [Y,\widetilde{Y}]
= 0 \\
\mbox{star-structure}: & X^* = X, \quad Y^* = Y, \quad \widetilde{X}^* = \widetilde{X}, \quad \widetilde{Y}^* = \widetilde{Y} \\
\mbox{coproduct}: &  \Delta X = X \otimes X, \quad \Delta Y = Y \otimes X + 1\otimes Y, \quad \Delta \widetilde{X} = \widetilde{X} \otimes \widetilde{X}, \quad \Delta \widetilde{Y} = \widetilde{Y} \otimes \widetilde{X} + 1\otimes \widetilde{Y} \\
\mbox{antipode}: & S(X) = X^{-1}, \quad S(Y) = -YX^{-1}, \quad S(\widetilde{X}) = \widetilde{X}^{-1}, \quad S(\widetilde{Y}) = -\widetilde{Y}\widetilde{X}^{-1} \\
\mbox{counit}: & \epsilon(X) = 1, \quad \epsilon(Y) = 0, \quad \epsilon(\widetilde{X}) = 1, \quad \epsilon(\widetilde{Y}) = 0
\end{array}
}
\end{align*} 
and $\mathcal{B}_{q\widetilde{q}}$ is defined as its subalgebra generated by $X,Y,\widetilde{X}, \widetilde{Y}$. Whenever clear from the context, $\mathcal{B}$ will denote $\mathcal{B}_{q\widetilde{q}}$ throughout.
\end{definition}

\begin{remark}
When $\mathcal{B}$ is represented as operators on a Hilbert space, we will require $\widetilde{X} = X^{1/b^2}$ and $\widetilde{Y} = Y^{1/b^2}$ as operators.
\end{remark}

\begin{proposition}
\label{3_prop:hopf_algebra_modular_double}
The algebra $B_{q\widetilde{q}}$ satisfies the Hopf $*$-algebra axioms.
\end{proposition}

We now let $\mathcal{B}$ act on $W$, by letting $\widetilde{X}, \widetilde{Y}$ act in a similar manner as in \eqref{3_eq:def_repres_B_q}, with $b$ replaced by $b^{-1}$. 
Then $W$ is a representation of the $*$-algebra $\mathcal{B}$ via the action $\pi$. Denote by $(\cdot,\cdot)$ the inner product of $L^2(\mathbb{R})$.

\begin{definition}
\label{3_def:Schwartz_space}
The Schwartz space $\mathcal{S}_\mathcal{B}$ for the $*$-algebra $\mathcal{B}$ is a subspace of $L^2(\mathbb{R})$ consisting of vectors $f$ such that the functional $w\mapsto (f, \pi(u) w)$ on $W$ is continuous for the $L^2$-norm, for any fixed $u\in \mathcal{B}$.
\end{definition}

The Schwartz space $\mathcal{S}_{\mathcal{B}}$ for the $*$-algebra $\mathcal{B}$ is the common domain of definition of operators from $\mathcal{B}$ in $L^2(\mathbb{R})$ (via action $\pi$). Given $s\in \mathcal{S}_\mathcal{B}$ and $u\in \mathcal{B}$, since the functional $w\mapsto (f, \pi(u)w)$ defined on a dense subspace $W$ of $L^2(\mathbb{R})$ is continuous, by Riesz representation theorem there exists a unique $g\in L^2(\mathbb{R})$ such that $(s, \pi(u)w) = (g, w)$ for all $w\in W$. Let $\pi(u^*)s := g$; this is how elements of $\mathcal{B}$ act on $\mathcal{S}_\mathcal{B}$ via $\pi$. This Schwartz space $\mathcal{S}_{\mathcal{B}}$ has a natural topology given by seminorms
$$
N_u(f) := ||uf||_{L^2}, \quad \mbox{$u$ runs through a basis in $\mathcal{B}$.}
$$
The key property of the Schwartz space $\mathcal{S}_{\mathcal{B}}$ is the following:

\begin{theorem}
\label{3_thm:W_density}
The space $W$ is dense in the Schwartz space $\mathcal{S}_{\mathcal{B}}$.
\end{theorem}

One can interpret Theorem \ref{3_thm:W_density} by saying that the $*$-algebra $\mathcal{B}$ is essentially self-adjoint in $L^2(\mathbb{R})$ (via $\pi$). See \cite{Goncharov} for proofs and other details.

\vs

We now define ``integrable'' representations for the modular double $\mathcal{B} = \mathcal{B}_{q\widetilde{q}}$:

\begin{definition}
The quadruple $(\pi(X),\pi(Y),\pi(\widetilde{X}),\pi(\widetilde{Y}))$ of positive definite operators on a Hilbert space defines an ``integrable'' representation of (the $*$-algebra) $\mathcal{B}$, if all of the following are satisfied:
\begin{enumerate}
\item[\rm 1)] the pair $(\pi(X), \pi(Y))$ defines an integrable representation of the subalgebra $\mathcal{B}_q$,

\item[\rm 2)] the pair $(\pi(\widetilde{X}), \pi(\widetilde{Y}))$ defines an integrable representation of the subalgebra $\mathcal{B}_{\widetilde{q}}$,

\item[\rm 3)] one has $\pi(X)^{1/b^2} = \pi(\widetilde{X})$, $\pi(Y)^{1/b^2} = \pi(\widetilde{Y})$, and $\pi(u)$ commutes with $\pi(\widetilde{v})$ for $u=X,Y$ and $v=\widetilde{X},\widetilde{Y}$, and

\item[\rm 4)] the $*$-algebra $\mathcal{B}$ is essentially self-adjoint (via $\pi$),
\end{enumerate}
where the subalgebras in $1), 2)$ are defined by $\mathcal{B}_q = \langle X,Y\rangle$ and $\mathcal{B}_{\widetilde{q}} = \langle \widetilde{X}, \widetilde{Y}\rangle$.
\end{definition}

\begin{proposition}
\label{3_prop:repres}
There exists a unique (up to unitary equivalence) irreducible integrable representation of (the $*$-algebra) $\mathcal{B}$ by positive-definite operators $\pi(u)$, $u=X,Y,\widetilde{X},\widetilde{Y}$, defined on $\mathcal{H} \equiv L^2(\mathbb{R},dx)$  by
\begin{align*}
\pi (X) = e^{-2\pi b p}, \quad
\pi (Y) = e^{2\pi b x}, \quad
\pi (\widetilde{X}) = e^{-2\pi b^{-1} p}, \quad
\pi (\widetilde{Y}) = e^{2\pi b^{-1} x},
\end{align*}
where the common domain of operators $\pi(u)$ ($u\in \mathcal{B}$) is $\mathcal{S}_\mathcal{B}$.
\end{proposition}

\begin{remark}
\label{3_rem:tensor_square_relations}
One has 
$(\pi^{\otimes 2}(\Delta X))^{1/b^2} = \pi^{\otimes 2} (\widetilde{X})$, $(\pi^{\otimes 2}(\Delta Y))^{1/b^2} = \pi^{\otimes 2} (\widetilde{Y})$; 
this will be explained at the end of \S\ref{4_subsection:intertwining_operator}.
\end{remark}

\begin{remark}
\label{3_rem:omit_X_inverse}
The reason why we omit $X^{-1}$ (and $\widetilde{X}^{-1}$) when considering representations will be explained in Remark \ref{4_rem:shifting_contour_X}. So far, it wasn't necessary to have $X^{-1}$ (and $\widetilde{X}^{-1}$).
\end{remark}

\section{Intertwining operators and their relations}

\subsection{Intertwining operator for tensor product decomposition}
\label{4_subsection:intertwining_operator}

For two representations $(\pi_j,V_j)$ of $\mathcal{B}$, $j=1,2$, the coproduct of $\mathcal{B}$ allows us to define tensor product representation $\pi_{12}$ on $V_1 \otimes V_2$ via $\pi_{12}(u) \equiv (\pi_1 \otimes \pi_2)(\Delta u)$ for every $u\in \mathcal{B}$. We take $V_j \cong \mathcal{H}$ (as in Proposition \ref{3_prop:repres}), $j=1,2$, and decompose $V_1 \otimes V_2$ into irreducibles. One obtains
\begin{align}
\label{4_eq:CG_direct_integral}
\mathcal{H} \otimes \mathcal{H} \cong \int_\mathbb{R}^{\oplus} \mathcal{H}
\end{align}
in some sense, hence this class of $\mathcal{B}$-representations as in Proposition \ref{3_prop:repres} (with single simple object) is closed under tensor product. The purpose of this section is to establish \eqref{4_eq:CG_direct_integral} in the following form
\begin{align}
\label{4_eq:CG_multiplicity_module}
\mathcal{H} \otimes \mathcal{H} \cong M \otimes  \mathcal{H},
\end{align}
where $M$ is understood as a ``multiplicity'' module. Recall $\mathcal{H} \equiv L^2(\mathbb{R})$ as vector spaces. We will realize $M$ also as $L^2(\mathbb{R})$ by establishing the isomorphism \eqref{4_eq:CG_multiplicity_module}; the explicit formula for the isomorphism (both directions) is given in this subsection.

\begin{definition}\label{4_def:CG}
For real variables $\alpha,x,x_1,x_2$, define a distribution kernel $\downker{\alpha}{x}{x_1}{x_2}$ and its inverse $\upker{\alpha}{x}{x_1}{x_2}$ as follow:
\begin{align}
\label{4_eq:def_downker}
\downker{\alpha}{x}{x_1}{x_2}
& = e^{2\pi i \alpha(x-x_1)} \mathcal{E}_R (x-x_1, x_2-x_1), \\
\label{4_eq:def_upker}
\upker{\alpha}{x}{x_1}{x_2}
& = e^{-2\pi i \alpha(x-x_1)} \mathcal{E}_L (x_2-x_1, x-x_1),
\end{align}
where $\mathcal{E}_R, \mathcal{E}_L$ are as defined in \eqref{2_eq:E_R_def} and \eqref{2_eq:E_L_def}.
\end{definition}

Our space of {\it distributions} will be $\mathcal{S}_{\mathcal{B}}^*$, the topological dual to the Schwartz space $\mathcal{S}_{\mathcal{B}}$ studied in subsection \ref{3_subsection:quantum_plane_modular_double}. 
In actual proofs of identities about distributions, it will be sufficient and convenient to evaluate the distributions on $W$ instead of on $\mathcal{S}_{\mathcal{B}}$, in virtue of Theorem \ref{3_thm:W_density}. We denote by $\langle~,~\rangle$ the pairing between distributions and the test functions. In this paper, the `transpose' (or `dual') of an operator $A$ on $\mathcal{S}_{\mathcal{B}}$ is denoted by $A^t$, so that $\langle A f, g \rangle = \langle f, A^t g\rangle$ for all $f,g\in S_{\mathcal{B}}$.

\vs

The main result of this subsection is the following theorem:
\begin{theorem}\label{4_thm:CG}
Let $(\pi_1, \mathcal{H}_1), (\pi_2, \mathcal{H}_2)$ be isomorphic to the representations of $\mathcal{B}$ as in Proposition \ref{3_prop:repres} (indices are for convenience). Then the $\mathcal{B}$-representation $\pi_{12}$ defined on $\mathcal{H}_1 \otimes \mathcal{H}_2$ decomposes into irreducible representations as follows:
\begin{align}
\label{4_eq:thm_CG_key_statement}
\mathcal{H}_1 \otimes \mathcal{H}_2 
\cong M \otimes \mathcal{H},
\end{align}
where the isomorphism is explicitly given by the unitary maps
\begin{align}
\begin{array}{rrcl}
& L^2(\mathbb{R} \times \mathbb{R}, dx_1dx_2)
& \to &
L^2(\mathbb{R} \times \mathbb{R}, d\alpha dx) \\
F_{1,2}: &
f(x_1,x_2) 
& \mapsto &
\displaystyle 
(F_{1,2} f)(\alpha,x), \\
H_{1,2}: &
\displaystyle
(H_{1,2} \varphi)(x_1,x_2)
& \mapsfrom &
\varphi(\alpha,x),
\end{array}
\end{align}
realized as integral transformations with distribution
kernels defined in Definition \ref{4_def:CG}
\begin{align}
(F_{1,2} f)(\alpha,x) 
& \equiv \int_\mathbb{R} \left( \int_\mathbb{R} \downker{\alpha}{x}{x_1}{x_2} f(x_1,x_2)  dx_2^{\cup x}\right) dx_1, \\
\label{4_eq:H_12_definition}
(H_{1,2} \varphi)(x_1,x_2) 
& \equiv \int_\mathbb{R} \left( \int_\mathbb{R} \upker{\alpha}{x}{x_1}{x_2} \varphi(\alpha,x) d\alpha \right) dx^{\cup x_2},
\end{align}
for $f(x_1,x_2) \in \mathcal{H}_1 \otimes \mathcal{H}_2$ and $\varphi(\alpha,x) \in M \otimes \mathcal{H}$, where the symbols $dx_2^{\cup x}$ and $dx^{\cup x_2}$ are as described in \S\ref{2_subsection:intro_quantum_dilogarithm}. The maps $F_{1,2}$ and $H_{1,2}$ intertwine the action of $\mathcal{B}$, where $M$ is regarded as trivial representation and $(\pi, \mathcal{H})$ as the representation as in Proposition \ref{3_prop:repres}. In particular, the corresponding projections $\Pi_{12}(\alpha)$, $(\Pi_{12}(\alpha) f) (x) = (F_{1,2}f)(\alpha,x)$ mapping $\mathcal{H}_1 \otimes \mathcal{H}_2$ into $\mathcal{H}$ intertwine the respective $\mathcal{B}$ actions according to
\begin{align}
\Pi_{12} (\alpha) \pi_{12} (u) = \pi(u) \Pi_{12}(\alpha), \quad \forall u \in \mathcal{B}.
\end{align}
(similarly for $H_{1,2}$) Moreover, the maps $F_{1,2}, H_{1,2}$ are inverses to each other, i.e.
\begin{align}
\label{4_eq:thm_CG_inverses}
H_{1,2} F_{1,2} f = f, \quad
F_{1,2} H_{1,2} \varphi = \varphi,
\end{align}
or equivalently,
\begin{align}
\label{4_eq:CG_orthogonality_alpha_x}
& \int_{\mathbb{R}^2} \upker{\alpha}{x}{y_1}{y_2} \downker{\alpha}{x}{x_1}{x_2} d\alpha dx^{\cup y_2\cap x_2}
= \delta(x_1-y_1) \delta(x_2 - y_2), \\
\label{4_eq:CG_orthogonality_x1_x2}
& \int_{\mathbb{R}^2} \downker{\beta}{y}{x_1}{x_2} \upker{\alpha}{x}{x_1}{x_2} dx_2^{\cup y\cap x} dx_1
= \delta(\alpha-\beta) \delta(x-y),
\end{align}
as distributions.
\end{theorem}

\begin{proof} First, we can show that the maps $F_{1,2}$ and $H_{1,2}$ are well-defined on $W\otimes W$, using \eqref{2_eq:G_asymptotics} and the decaying property of elements of $W\otimes W$; then we extend to $L^2(\mathbb{R}^2)$ by continuity (using unitarity of $F_{1,2}$ and $H_{1,2}$). Unitarity of $F_{1,2}$ and $H_{1,2}$ follows from Theorem \ref{2_thm:T_R_T_L} (i.e. unitarity of the transformations $\mathcal{T}_R$ and $\mathcal{T}_L$) and the unitarity of Fourier transformation. It now suffices to prove that $H_{1,2}$ intertwines the $\mathcal{B}$-action, and is the inverse mapping of $F_{1,2}$.

\vs

Let's first prove the intertwining property. Observe that
\begin{align}
\label{4_eq:tensor_square_action}
\begin{array}{ll}
\pi_{12}(X) = \pi^{\otimes 2}(\Delta X) = e^{-2\pi b(p_1+p_2)}, 
& \pi_{12}(\widetilde{X}) = e^{-2\pi b^{-1}(p_1+p_2)}, \\
\pi_{12}(Y) = \pi^{\otimes 2}(\Delta Y) = e^{2\pi b (x_1-p_2)} + e^{2\pi b x_2}, \quad
&\pi_{12}(\widetilde{Y}) = e^{2\pi b^{-1} (x_1-p_2)} + e^{2\pi b^{-1} x_2}.
\end{array}
\end{align}
Suppose $u_1,u_2\in \mathcal{B}$ satisfy $(H_{1,2} \pi(u_j) \varphi)(x_1,x_2) = \pi_{12}(u_j) (H_{1,2}\varphi)(x_1,x_2)$ for $j=1,2$ and for all $\varphi \in W\otimes W$. Note that $\pi(u)$ leaves $W$ invariant, for any $u \in \mathcal{B}$; thus, for $\varphi \in W\otimes W$ we have $\pi(u_2) \varphi \in W\otimes W$, and therefore
$$
H_{1,2} \pi(u_1) \pi(u_2) \varphi
= \pi_{12}(u_1) H_{1,2} \pi(u_2) \varphi
= \pi_{12}(u_1) \pi_{12}(u_2) H_{1,2} \varphi.
$$
Hence it's enough to prove the intertwining property for $u=X,Y,\widetilde{X},\widetilde{Y}$.

\vs

Since Fourier transform preserves $W$, we get that for $\varphi(\alpha,x) \in W\otimes W$, the inner integral for the expression \eqref{4_eq:H_12_definition} for $H_{1,2}\varphi$, as a function in $x$, is $\mathcal{E}_L(x_2-x_1, x-x_1)$ times an element of $W$. Put $\pi(X) \varphi(\alpha,x) = \varphi(\alpha,x+ib)$ in place of $\varphi$. Using \eqref{2_eq:G_asymptotics} and the decaying property of elements of $W$, we can move the contour for $x$ by the amount $-ib$; when doing this, $x$ doesn't hit its pole at $x_2$, and the contour for $x$ can now be just straight line (after shifting the contour). Since $\upker{\alpha}{x}{x_1}{x_2}$ is expressed only in terms of $\alpha, (x-x_1), (x_2-x_1)$ (see \eqref{4_eq:def_upker}), it's clear that \begin{align*}
e^{2\pi b^{\pm 1} (p_1+p_2+p)} \upker{\alpha}{x}{x_1}{x_2}
= \upker{\alpha}{x}{x_1}{x_2},
\end{align*}
hence
\begin{align}
\nonumber
(H_{1,2} \pi(X) \varphi)(x_1,x_2)
& = \int_\mathbb{R} \left( \int_\mathbb{R} \upker{\alpha}{x}{x_1}{x_2} \varphi(\alpha,x+ib) d\alpha \right) dx^{\cup x_2} \\
\label{4_eq:shifting_contour_X}
& = \int_\mathbb{R} \left( \int_\mathbb{R} \upker{\alpha}{x-ib}{x_1}{x_2} \varphi(\alpha,x) d\alpha \right) dx \\
\nonumber
& = \int_\mathbb{R} \left( \int_\mathbb{R} \upker{\alpha}{x}{x_1+ib}{x_2+ib} \varphi(\alpha,x) d\alpha \right) dx \\
\nonumber
& = \pi_{12}(X) (H_{1,2} \varphi)(x_1,x_2).
\end{align}
Similar proof goes for intertwining property for the action of $\widetilde{X}$.

\vs

From property \eqref{2_eq:E_L_relation} of $\mathcal{E}_L$,  we can deduce
$$
e^{2\pi b^{\pm 1}x} \upker{\alpha}{x}{x_1}{x_2}
= e^{2\pi b^{\pm 1} x_1} \upker{\alpha}{x}{x_1}{x_2+ib}
 + e^{2\pi b^{\pm 1} x_2} \upker{\alpha}{x}{x_1}{x_2}.
$$
Now, note
\begin{align*}
(H_{1,2} \pi(Y) \varphi)(x_1,x_2)
& = \int_\mathbb{R} \left( \int_\mathbb{R} \upker{\alpha}{x}{x_1}{x_2} e^{2\pi b x} \varphi(\alpha,x) d\alpha \right) dx^{\cup x_2} \\
& = \int_\mathbb{R} \left( \int_\mathbb{R} e^{2\pi b x_1} \upker{\alpha}{x}{x_1}{x_2+ib} \varphi(\alpha,x) d\alpha \right) dx^{\cup x_2} \\
& \quad + \int_\mathbb{R} \left( \int_\mathbb{R} e^{2\pi b x_2} \upker{\alpha}{x}{x_1}{x_2} \varphi(\alpha,x) d\alpha \right) dx^{\cup x_2} \\
& = \pi_{12}(Y) (H_{1,2} \varphi)(x_1,x_2),
\end{align*}
Similar proof goes for intertwining property for the action of $\widetilde{Y}$.

\vs

To prove that $F_{1,2}$ and $H_{1,2}$ are inverses to each other, we use Corollary \ref{2_cor:E_orthogonality} and
\begin{align}
\label{4_eq:usual_delta}
\int_\mathbb{R} \int_\mathbb{R} e^{2\pi i y(w-z)} f(w) dw dy = f(z). \quad (\mbox{e.g. for $f\in W$})
\end{align}
For $f\in W\otimes W$, observe
\begin{align*}
& (H_{1,2} F_{1,2} f) (y_1,y_2) \\
& = \int_{\mathbb{R}^3} \left( \int_\mathbb{R} \upker{\alpha}{x}{y_1}{y_2} \downker{\alpha}{x}{x_1}{x_2} f(x_1,x_2)  dx_2^{\cup x} \right) dx_1 d\alpha dx^{\cup y_2} \\
&  = \int_{\mathbb{R}^4} e^{2\pi i \alpha(y_1-x_1)} \mathcal{E}_L(y_2-y_1,x-y_1) \mathcal{E}_R(x-x_1,x_2-x_1) f(x_1,x_2)  dx_2^{\cup x} dx_1 d\alpha dx^{\cup y_2} \\
&  = \int_{\mathbb{R}^2} \mathcal{E}_L(y_2-y_1,x-y_1) \mathcal{E}_R(x-y_1,x_2-y_1) f(y_1,x_2)  dx_2^{\cup x} dx^{\cup y_2} \quad (\mbox{by \eqref{4_eq:usual_delta}}) \\
&  = \int_{\mathbb{R}^2} \mathcal{E}_L(y_2-y_1,x-y_1) \mathcal{E}_R(x-y_1,X_2) f(y_1,X_2+y_1)  dX_2^{\cup (x-y_1)} dx^{\cup y_2} \quad (\mbox{set } X_2 = x_2-y_1) \\
&  = \int_{\mathbb{R}^2} \mathcal{E}_L(y_2-y_1,X) \mathcal{E}_R(X,X_2) f(y_1,X_2+y_1)  dX_2^{\cup X} dX^{\cup (y_2-y_1)} \quad (\mbox{set } X = x-y_1) \\
&  = f(y_1,y_2) \quad(\mbox{by }\eqref{2_eq:E_orthogonality_precise_first}).
\end{align*}
For $\varphi \in W \otimes W$, observe
\begin{align*}
& (F_{1,2} H_{1,2} \varphi) (\beta,y) \\
& = \int_{\mathbb{R}^4} \downker{\beta}{y}{x_1}{x_2} \upker{\alpha}{x}{x_1}{x_2}  \varphi(\alpha,x)  d\alpha dx^{\cup x_2} dx_2^{\cup y} dx_1  \\
&  = \int_{\mathbb{R}^4} e^{2\pi i \beta(y-x_1)} e^{-2\pi i \alpha(x-x_1)} \mathcal{E}_R(y-x_1,x_2-x_1) \mathcal{E}_L(x_2-x_1,x-x_1)  \varphi(\alpha,x) d\alpha dx^{\cup x_2} dx_2^{\cup y} dx_1  \\
&  = \int_{\mathbb{R}^3} e^{2\pi i \beta(y-x_1)} \mathcal{E}_R(y-x_1,x_2-x_1) \mathcal{E}_L(x_2-x_1,x-x_1) \widehat{\varphi}_\alpha (x-x_1,x) dx^{\cup x_2} dx_2^{\cup y} dx_1 \\
&  = \int_{\mathbb{R}^3} e^{2\pi i \beta(y-x_1)} \mathcal{E}_R(y-x_1,x_2-x_1) \mathcal{E}_L(x_2-x_1,X) \widehat{\varphi}_\alpha (X,X+x_1) dX^{\cup (x_2-x_1)} dx_2^{\cup y} dx_1  \\
&  = \int_{\mathbb{R}^3} e^{2\pi i \beta(y-x_1)} \mathcal{E}_R(y-x_1,X_2) \mathcal{E}_L(X_2,X) \widehat{\varphi}_\alpha (X,X+x_1) dX^{\cup X_2} dX_2^{\cup (y-x_1)} dx_1  \\
&  \stackrel{{\rm Cor.}~\ref{2_cor:E_orthogonality}}{=} \int_{\mathbb{R}} e^{2\pi i \beta(y-x_1)} \widehat{\varphi}_\alpha (y-x_1,y) dx_1 = \int_{\mathbb{R}} e^{2\pi i \beta X_1} \widehat{\varphi}_\alpha (X_1,y) dX_1    \\
& = \int_{\mathbb{R}^2} e^{2\pi i (\beta -\alpha) X_1} \varphi (\alpha,y) d\alpha dX_1 \stackrel{\eqref{4_eq:usual_delta}}{=} \varphi(\beta,y),
\end{align*}
where $\widehat{\varphi}_\alpha$ in the fourth line and on means the Fourier transform of $\varphi$ w.r.t. $\alpha$ variable.
\end{proof}

\begin{remark}
\label{4_rem:shifting_contour_X}
In the above proof, when proving the intertwining property for $X$ and $\widetilde{X}$, we shifted contours of integration (see \eqref{4_eq:shifting_contour_X}); if we try to prove the same for $X^{-1}$ and $\widetilde{X}^{-1}$, we will gain an additional term from residue (since in this case we pass through a pole when shifting the contour). This is why we omitted $X^{-1}$ when considering representations, as mentioned in Remark \ref{3_rem:omit_X_inverse}.
\end{remark}

\begin{remark}
We can define the corresponding Schwartz spaces for both sides $\mathcal{H} \otimes \mathcal{H}$ and $M \otimes \mathcal{H}$ (according to $\mathcal{B}$-action on those spaces). Then, using similar proof as in \S2.4 of \cite{Goncharov}, one can prove that the unitary mapping $F_{1,2}$ (resp. $H_{1,2}$) maps the Schwartz space for $\mathcal{H} \otimes \mathcal{H}$ to that for $M \otimes \mathcal{H}$ (resp. vice versa).
\end{remark}

Remark \ref{3_rem:tensor_square_relations} follows from \eqref{4_eq:tensor_square_action} and the following proposition (see e.g. Lemma 3 of \cite{BT}), which is essentially a special case of tau-binomial formula (see \eqref{2_eq:G_new_tau_binomial_prime}):

\begin{proposition}
If $u=e^{2\pi b A}$, $v=e^{2\pi bB}$, $A,B$ self-adjoint operators s.t. $[A,B]=1/(2\pi i)$, then $(u+v)^{1/b^2} = u^{1/b^2} + v^{1/b^2}$ (here $b \in (0,\infty)$, $b^2\notin \mathbb{Q}$).
\end{proposition}

\subsection{Relation between intertwining operators for triple tensor products}

For three representations $(\pi_j, V_j)$ of $\mathcal{B}$, $j=1,2,3$, the coproduct of $\mathcal{B}$ allows us to define representation on $V_1 \otimes V_2 \otimes V_3$ via $\pi_{123}(u) = (\pi_1 \otimes \pi_2 \otimes \pi_3)(\Delta^{(3)} u)$ for every $u\in \mathcal{B}$, where $\Delta^{(3)} = (\Delta \otimes id) \circ \Delta = (id \otimes \Delta) \circ \Delta$. For the case $V_j \cong \mathcal{H}$ ($j=1,2,3$) as in Proposition \ref{3_prop:repres}, we can decompose this triple tensor product into direct integral of irreducible representations $\mathcal{H}$, using the intertwining operator we obtained in \S\ref{4_subsection:intertwining_operator}. There are two canonical ways to do this, as follow:
\begin{align}
\label{4_eq:first_decomp}
& (\mathcal{H}_1 \otimes \mathcal{H}_2) \otimes \mathcal{H}_3
\cong M_{12}^4 \otimes \mathcal{H}_4 \otimes \mathcal{H}_3
\cong M_{12}^4 \otimes M_{43}^6 \otimes \mathcal{H}_6
\cong M_{43}^6 \otimes M_{12}^4 \otimes \mathcal{H}_6, \\
\label{4_eq:second_decomp}
& \mathcal{H}_1 \otimes (\mathcal{H}_2 \otimes \mathcal{H}_3)
\cong \mathcal{H}_1 \otimes M_{23}^5 \otimes \mathcal{H}_5
\cong M_{23}^5 \otimes \mathcal{H}_1 \otimes \mathcal{H}_5
\cong M_{23}^5 \otimes M_{15}^7 \otimes \mathcal{H}_7,
\end{align} 
where we give indices to the multiplicity modules $M$ by $\mathcal{H}_j \otimes \mathcal{H}_k \cong M_{jk}^\ell \otimes \mathcal{H}_\ell$.  Each isomorphism can be realized as integral transformation with some distribution kernel (see \S\ref{4_subsection:intertwining_operator}), as each space $\mathcal{H}_j$ and $M_{jk}^\ell$ is realized as $L^2(\mathbb{R})$; we give the variable names as follows: $\mathcal{H}_j \equiv L^2(\mathbb{R}, dx_j)$, $\forall j\le7$,
\begin{align*}
M_{12}^4 \equiv L^2(\mathbb{R}, d\alpha), ~
M_{43}^6 \equiv L^2(\mathbb{R}, d\beta), ~
M_{23}^5 \equiv L^2(\mathbb{R}, dA),~
M_{15}^7 \equiv L^2(\mathbb{R}, dB).
\end{align*}
By composing isomorphisms in \eqref{4_eq:first_decomp} and \eqref{4_eq:second_decomp}, we obtain the isomorphism
\begin{align}
\label{4_eq:isomorphism_between_two}
M_{43}^6 \otimes M_{12}^4 \otimes \mathcal{H}_6
& \stackrel{\sim}{\longrightarrow}
~ M_{23}^5 \otimes M_{15}^7 \otimes \mathcal{H}_7,
\end{align}
realized by the unitary map
\begin{align}
\nonumber
L^2(\mathbb{R}^3, d\beta d\alpha dx_6)
& \stackrel{\sim}{\longrightarrow}
~ L^2(\mathbb{R}^3, dA dB dx_7)  \\
\label{4_eq:isomorphism_between_two_explicit}
\varphi(\beta,\alpha,x_6)
& \longmapsto
\int_{\mathbb{R}^3} \chi_{ts}(\beta,\alpha,x_6,A,B,x_7) \varphi(\beta,\alpha,x_6) d\beta d\alpha dx_6,
\end{align}
where
\begin{align*}
& \chi_{ts}(\beta,\alpha,x_6,A,B,x_7)  \\
& =
\int_{\mathbb{R}^5} \upker{\beta}{x_6}{x_4}{x_3}
\upker{\alpha}{x_4}{x_1}{x_2}
\downker{A}{x_5}{x_2}{x_3}
\downker{B}{x_7}{x_1}{x_5}
dx_4^{\cup x_2} dx_3^{\cup x_5 \cap x_6} dx_2 dx_5^{\cup x_7} dx_1,
\end{align*}
as distributions. Readers can easily find out the precise meaning of \eqref{4_eq:isomorphism_between_two_explicit}, following definitions of $F_{1,2}$ and $H_{1,2}$.

\begin{proposition}
One has $\chi_{ts}(\beta,\alpha,x_6,A,B,x_7) = \delta(x_6-x_7) T(\beta,\alpha,A,B)$ as distributions, where
\begin{align}
\label{4_eq:T_formula}
T(\beta,\alpha,A,B)
& = \frac{1}{2} e^{\pi i (\beta B + \alpha A - \beta \alpha - AB + \alpha^2/2 - B^2/2)} \mathcal{G}(\beta+\alpha+A-2B), \quad \mbox{and} \\
\label{4_eq:curly_G_formula}
\mathcal{G}(y) & = \int_\mathbb{R} G (x-ia) e^{-\pi i xy} e^{-\pi a x}dx^{\cap 0},
\quad y\in \mathbb{R}.
\end{align}
\end{proposition}

\begin{proof}
We will now compute $\chi_{ts}(\beta,\alpha,x_6,A,B,x_7)$ as distributions, to obtain the asserted result. One can also perform the computation in a more detailed way, in a similar manner as done in the proof of Theorem \ref{4_thm:CG} (when proving $F_{1,2}$ is inverse to $H_{1,2}$).

\vs

By putting in the definitions \eqref{4_eq:def_downker} and \eqref{4_eq:def_upker}, we get
\begin{align*}
\chi_{ts}
& = \int_{\mathbb{R}^5} G(x_2 - x_4-ia) G(x_3 - x_6-ia) G(x_5 - x_3-ia) \\
& \qquad \cdot G(x_7 - x_5-ia) \exp(*T_1) dx_4^{\cup x_2} dx_3^{\cup x_5\cap x_6} dx_2 dx_5^{\cup x_7} dx_1,
\end{align*}
where $\chi_{ts}$ is shorthand for $\chi_{ts}(\beta,\alpha,x_6,A,B,x_7)$, and 
\begin{align*}
(*T_1) / (\pi i)
& =
2 x_1 (\alpha - B + x_4 + x_2 - x_5 - x_7)
- 2x_2 A
+ 2x_4 (-\alpha + \beta)
+ 2x_5 A
- 2x_6 \beta
+ 2x_7 B
 \\
& \quad
+ (
- x_2 x_4
+ 2x_4 x_6
 - 2x_2 x_5
+ x_5 x_7
)
+ x_3( - 2x_2 + 2 x_4+ x_5-x_6) \\
& \quad
+ (
 + 3x_2^2/2 - 5x_4^2 /2
+ x_5^2 - x_6^2 /2
+ x_7^2 /2)
-ia \left( 
- x_2 - 2 x_3 + x_4 + x_6 + x_7 
\right).
\end{align*}
We integrate with respect to the variable $x_3$ first, using tau-binomial theorem \eqref{2_eq:G_new_tau_binomial_prime} 
\begin{align}
\nonumber
& \int_\mathbb{R} G (x_3-x_6-ia) G (x_5-x_3-ia) e^{-i\pi x_3( x_5-x_6 + 2w)} e^{-2\pi ax_3} dx_3^{\cup x_5 \cap x_6} \\
\nonumber
& =
\frac{G(x_5-x_6-ia)G(w-ia)}{G(x_5-x_6+w-ia)} e^{-\pi i (x_5+x_6)w} e^{\frac{i\pi}{2} (-x_5^2+x_6^2)} e^{- \pi a (x_5+x_6)},
\end{align}
for $w=x_6-x_5+x_2-x_4$. Then, 
\begin{align*}
\chi_{ts}
& = \int_{\mathbb{R}^4} G (x_5-x_6-ia) G (x_6-x_5+x_2-x_4-ia) G(x_7 - x_5-ia) \\
& \qquad \cdot \exp(*T_2) dx_4^{\cup (x_2-x_5+x_6)} dx_2 dx_5^{\cup x_7} dx_1, \\
(*T_2) / (\pi i)
& = ( 
2 x_1 (\alpha - B + x_4 + x_2 - x_5 - x_7)
- 2x_2 A
+ 2x_4 (-\alpha + \beta)
+ 2x_5 A
- 2x_6 \beta
+ 2x_7 B
) \\
& \quad
+ (
- x_2 x_4 + 3x_4 x_6  - 3x_2 x_5 + x_5 x_7+ x_4 x_5 - x_2 x_6 
) \\
& \quad
+ \left(
 + 3x_2^2/2 - 5x_4^2 /2
+ 3x_5^2/2 - x_6^2
+ x_7^2 /2
\right)
- ia \left( 
- x_2 + x_4 -x_5  + x_7 
\right).
\end{align*}
Integrating w.r.t. $x_1$ yields the factor $\delta(\alpha - B + x_4 + x_2 - x_5 - x_7)$, because
$
\int_\mathbb{R} e^{2\pi i xy} dx = \delta(y)
$
 as distributions, for real variable $y$. Then, integrating w.r.t $x_2$ has the effect of replacing all $x_2$ by $-\alpha+B-x_4+x_5+x_7$. Thus
\begin{align*}
\chi_{ts}
& = \int_{\mathbb{R}^2} G(x_5-x_6-ia) G(-\alpha + B - 2x_4 + x_6 + x_7-ia) G(x_7 - x_5-ia)  \\
& \qquad \cdot
\exp(*T_3) dx_4^{\cup (-\alpha+B+x_6+x_7)/2}
dx_5^{\cup x_7}, \\
(*T_3) / (\pi i)
& = ( 
2\alpha A- 2AB + 3\alpha^2/2 + 3B^2/2 - 3\alpha B ) \\
& \quad
+ (
x_4(+ 2 A + 2\alpha + 2\beta - 4B)
+ x_6 (- 2 \beta + \alpha - B)
+ x_7 ( - 3 \alpha - 2 A + 5 B )
) \\
& \quad
+ (
 4 x_4 x_6 - 4x_4 x_7 - x_6 x_7  - x_6^2 + 2x_7^2)
+ x_5(x_7 - x_6) 
- i a \left( 
\alpha - B + 2x_4 - 2x_5
\right).
\end{align*}
From \eqref{2_eq:G_delta2} we have
\begin{align*}
\int_\mathbb{R} G(x_5 - x_6-ia) G(x_7 - x_5-ia)
e^{\pi i x_5 (x_7 - x_6)} e^{-2\pi a x_5} dx_5^{\cup x_7}
= \delta( x_7- x_6) e^{-\pi a (x_6 + x_7)}.
\end{align*}
Hence
\begin{align*}
\chi_{ts}
& = \delta(x_7 - x_6) \int_\mathbb{R} G(-\alpha + B - 2x_4 + x_6 + x_7-ia) \exp(*T_4)
dx_4^{\cup (-\alpha+B+x_6+x_7)/2}, \\
(*T_4) / (\pi i)
& = ( 
2\alpha A- 2AB + 3\alpha^2/2 + 3B^2/2 - 3\alpha B ) \\
& \quad
+ (
x_4(+ 2 A + 2\alpha + 2\beta - 4B)
+ x_6 (- 2 \beta + \alpha - B)
+ x_7 ( - 3 \alpha - 2 A + 5 B )
) \\
& \quad
+ (
4 x_4 x_6 - 4x_4 x_7 - x_6 x_7  - x_6^2 + 2x_7^2)
-ia \left( 
\alpha - B + 2x_4 - x_6 - x_7
\right)
\end{align*}
Now that we have $\delta(x_7-x_6)$, we can replace $x_6$'s by $x_7$'s outside $\delta(x_7-x_6)$, which makes sense for distributions. Therefore
\begin{align*}
\chi_{ts}
& = \delta(x_7 - x_6) \int_\mathbb{R} G (-\alpha + B - 2x_4 + 2x_7-ia) \exp(*T_5)  dx_4^{\cup (-\alpha/2+B/2+x_7)}, \\
(*T_5) / (\pi i)
& = ( 
2\alpha A- 2AB + 3\alpha^2/2 + 3B^2/2 - 3\alpha B ) \\
& \quad
+ 2 (x_4 - x_7)( \alpha + \beta + A - 2B)
-ia \left( 
+\alpha - B + 2(x_4 -x_7)
\right).
\end{align*}
Change of variables: set $X = -\alpha + B - 2(x_4-x_7)$. Then
\begin{align*}
\chi_{ts}
= \frac{1}{2} \delta(x_7 - x_6) e^{\pi i ( 
\beta B + \alpha A  -  \beta\alpha - AB + \alpha^2/2 - B^2/2 )} \int_\mathbb{R} G(X-ia) 
e^{- \pi i X(\alpha + \beta + A - 2B)}e^{ - \pi a X} dX^{\cap 0},
\end{align*}
as desired.
\end{proof}

\begin{remark}
If the distribution kernels of the integral transformations $F_{1,2}$ and $H_{1,2}$ are modified in a certain way, $\chi_{ts}(\beta,\alpha,x_6,A,B,x_7)$ can be obtained as a closed form (i.e. without integral in its expression); the modification (new kernels are denoted with subscript $*$)
\begin{align*}
\downker{\alpha}{x}{x_1}{x_2}_* & = \frac{\overline{\zeta_b} e^{-\pi i(x-x_1)^2} }{G_b(a + i\alpha)} \downker{\alpha}{x}{x_1}{x_2}, \quad\mbox{and} \\
\upker{\alpha}{x}{x_1}{x_2}_* & = \zeta_b e^{\pi i (x-x_1)^2} G_b(a+i\alpha) \upker{\alpha}{x}{x_1}{x_2},
\end{align*}
where $G_b(a+ix) = G(-x)e^{-\frac{\pi i}{2}(a^2+x^2)}$ and $\zeta_b= \zeta^{-1}e^{\pi i/12}$, results in
$$
\chi_{ts}(\beta,\alpha,x_6,A,B,x_7)_*
= \delta(x_6-x_7) \frac{G_b(a+i\alpha) G_b(a+i\beta)}{G_b(a+iA) G_b(a+iB)} \frac{\overline{\zeta_b} e^{\pi i (\alpha-B)(\alpha+2A-2\beta+B)}}{G_b(2a - iB + i\beta)}.
$$
This modification of $F_{1,2}$ and $H_{1,2}$ is meaningful in the sense that now taking the limit as $b\to 0$ (in some sense) yields classical results; details about this limit will appear in \cite{Ivan}. If we take out $G_b(a+i\alpha)$ in the definition of the new kernels, then the resulting $\chi_{ts}$ is just $\delta(x_6-x_7) \overline{\zeta_b} e^{\pi i (\alpha-B)(\alpha+2A-2\beta+B)} / G_b(2a - iB + i\beta)$. This whole remark is due to Ivan Ip.
\end{remark}

Thus it now makes sense to identify the indices $6$ and $7$ in \eqref{4_eq:first_decomp}, \eqref{4_eq:second_decomp}, and in \eqref{4_eq:isomorphism_between_two}. We just proved that the mapping
\begin{align}
\label{4_eq:basic_T}
{\bf T} : M_{43}^6 \otimes M_{12}^4
& \stackrel{\sim}{\longrightarrow} M_{23}^5 \otimes M_{15}^6 \\
\nonumber
f(\beta,\alpha)
&\longmapsto \int_{\mathbb{R}^2} T(\beta,\alpha,A,B) f(\beta,\alpha) d\beta d\alpha
\end{align}
(where $T$ is as in \eqref{4_eq:T_formula}) makes the following diagram to commute:
\begin{align}
\label{4_eq:T_commutative_diagram}
\xymatrixcolsep{4pc}\xymatrix{
(\mathcal{H}_1 \otimes \mathcal{H}_2) \otimes \mathcal{H}_3
\ar[r]^{id \otimes id \otimes id} \ar[d]_{ (id \otimes F_{12,3})(F_{1,2} \otimes id) } &
\mathcal{H}_1 \otimes (\mathcal{H}_2 \otimes \mathcal{H}_3)
\ar[d]^{ (id \otimes F_{1,23})(P_{(12)} \otimes id)(id \otimes F_{2,3})} \\
M_{43}^6 \otimes M_{12}^4 \otimes \mathcal{H}_6
\ar[r]^{{\bf T}\otimes id} &
M_{23}^5 \otimes M_{15}^6 \otimes \mathcal{H}_6,
}
\end{align}
where the maps $F$ are as in Theorem \ref{4_thm:CG}, and $F_{12,3}, F_{1,23}$ can be understood as $F_{4,3}, F_{1,5}$ respectively, and $P_{(12)}$ is permutation of the two factors. All four arrows in the above diagram intertwine the $\mathcal{B}$-actions.

\vs

It is often convenient to encode facts in representation theory using various geometric pictures. We'll use the ones considered by Kashaev for quantization of the Teichm\"{u}ller spaces \cite{Kash98}. The full relation of representation theory of quantum plane and Kashaev's quantization will be revealed in the subsequent sections.

\vs

The decomposition $\mathcal{H}_1 \otimes \mathcal{H}_2 \cong M_{12}^3 \otimes \mathcal{H}_3$ can be encoded geometrically as a triangle with the edges named $1,2,3$, with a distinguished corner (indicated by a dot), as in Figure \ref{4_fig:one_triangle}.
\begin{figure}[htbp!]
\centering
\begin{pspicture}[showgrid=false](0,0)(3.3,2.8)
\rput[bl](0,-0.5){
\PstTriangle[unit=1.5,PolyName=P]
\pcline(P1)(P2)\ncput*{1}
\pcline(P2)(P3)\ncput*{2}
\pcline(P3)(P1)\ncput*{3}
\rput[l]{30}(P2){\hspace{1,7mm}$\bullet$}
}
\end{pspicture}
\caption{A triangle representing $M_{12}^3$, where $\mathcal{H}_1 \otimes \mathcal{H}_2 \cong M_{12}^3 \otimes \mathcal{H}_3$}
\label{4_fig:one_triangle}
\end{figure}
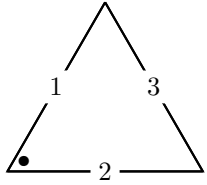
This triangle will represent the multiplicity module $M_{12}^3$, and any orientation-preserving deformation of this picture would mean the same.
Then, the transformation ${\bf T}$ as in \eqref{4_eq:basic_T} can also be encoded geometrically, as in Figure \ref{4_fig:T}.
\begin{figure}[htbp!]
\centering
\begin{pspicture}[showgrid=false](0,0)(8.4,3.5)
\rput[bl](0;0){
\PstSquare[unit=1.8,PolyName=P]
\pcline(P1)(P2)\ncput*{1}
\pcline(P2)(P3)\ncput*{2}
\pcline(P3)(P4)\ncput*{3}
\pcline(P4)(P1)\ncput*{6}
\pcline(P1)(P3)\ncput*{4}
\rput[l]{-45}(P2){\hspace{1,4mm}$\bullet$}
\rput[l]{22}(P3){\hspace{2,4mm}$\bullet$}
}
\rput[l](2.5,1.3){$j$}
\rput[l](1.2,2.5){$k$}
\rput[bl](4.5;0){
\PstSquare[unit=1.8,PolyName=P]
\pcline(P1)(P2)\ncput*{1}
\pcline(P2)(P3)\ncput*{2}
\pcline(P3)(P4)\ncput*{3}
\pcline(P4)(P1)\ncput*{6}
\pcline(P2)(P4)\ncput*{5}
\rput[l]{-23}(P2){\hspace{2,4mm}$\bullet$}
\rput[l]{41}(P3){\hspace{1,4mm}$\bullet$}
}
\rput[l](5.8,1.3){$j$}
\rput[l](6.8,2.5){$k$}
\rput[l](3.7,2){\pcline{->}(0,0)(1;0)\Aput{${\bf T}_{jk}$}}
\end{pspicture}
\caption{The move representing ${\bf T} : M_{43}^6 \otimes M_{12}^4 \to M_{23}^5 \otimes M_{15}^6$}
\label{4_fig:T}
\end{figure}
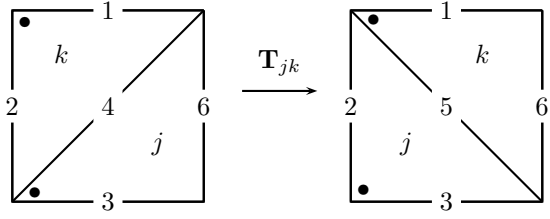

\vs

When we have tensor product of four copies of representations $\mathcal{H}$, we can consider the following commutative diagram (where all maps are identities):
\begin{align}
\label{4_eq:pentagon_H}
\begin{array}{c}
((\mathcal{H}_1 \otimes \mathcal{H}_2) \otimes \mathcal{H}_3) \otimes \mathcal{H}_4
\to
(\mathcal{H}_1 \otimes (\mathcal{H}_2 \otimes \mathcal{H}_3)) \otimes \mathcal{H}_4 
\to
\mathcal{H}_1 \otimes ((\mathcal{H}_2 \otimes \mathcal{H}_3) \otimes \mathcal{H}_4) \\
\searrow \qquad \qquad \qquad
\qquad \qquad \qquad
\qquad \qquad \qquad
\qquad \qquad
\swarrow \\
(\mathcal{H}_1 \otimes \mathcal{H}_2) \otimes (\mathcal{H}_3 \otimes \mathcal{H}_4)
\to
\mathcal{H}_1 \otimes (\mathcal{H}_2 \otimes (\mathcal{H}_3 \otimes \mathcal{H}_4))
\end{array}
\end{align}
If we translate the commutativity of the above diagram into the language of multiplicity modules $M$ and the operators ${\bf T}$, we will get the ``pentagon equation'' for ${\bf T}$,
which can be illustrated as in Figure \ref{4_fig:pentagon_eq}.
\begin{figure}[htbp!]
\centering
\begin{pspicture}[showgrid=false](0,-3.5)(14,4)
\rput[bl](0;0){
\PstPentagon[unit=1.8,PolyName=P]
\pcline(P2)(P3)\ncput*{1}
\pcline(P3)(P4)\ncput*{2}
\pcline(P4)(P5)\ncput*{3}
\pcline(P5)(P1)\ncput*{4}
\pcline(P1)(P2)\ncput*{n}
\pcline(P2)(P4)\ncput*{5}
\rput[l]{-10}(P3){\hspace{1,4mm}$\bullet$}
\pcline(P2)(P5)\ncput*{6}
\rput[l]{36}(P4){\hspace{1,4mm}$\bullet$}
\rput[l]{87}(P5){\hspace{3,2mm}$\bullet$}
}
\rput[bl](5;0){
\PstPentagon[unit=1.8,PolyName=P]
\pcline(P2)(P3)\ncput*{1}
\pcline(P3)(P4)\ncput*{2}
\pcline(P4)(P5)\ncput*{3}
\pcline(P5)(P1)\ncput*{4}
\pcline(P1)(P2)\ncput*{n}
\pcline(P3)(P5)\ncput*{9}
\rput[l]{62}(P4){\hspace{1,4mm}$\bullet$}
\pcline(P2)(P5)\ncput*{6}
\rput[l]{0}(P3){\hspace{2,0mm}$\bullet$}
\rput[l]{87}(P5){\hspace{3,2mm}$\bullet$}
}
\rput[bl](10;0){
\PstPentagon[unit=1.8,PolyName=P]
\pcline(P2)(P3)\ncput*{1}
\pcline(P3)(P4)\ncput*{2}
\pcline(P4)(P5)\ncput*{3}
\pcline(P5)(P1)\ncput*{4}
\pcline(P1)(P2)\ncput*{n}
\pcline(P3)(P5)\ncput*{9}
\rput[l]{762}(P4){\hspace{1,4mm}$\bullet$}
\pcline(P3)(P1)\ncput*{8}
\rput[l]{16}(P3){\hspace{3,0mm}$\bullet$}
\rput[l]{111}(P5){\hspace{1,8mm}$\bullet$}
}
\rput[tl](2.5;-1.5){
\PstPentagon[unit=1.8,PolyName=P]
\pcline(P2)(P3)\ncput*{1}
\pcline(P3)(P4)\ncput*{2}
\pcline(P4)(P5)\ncput*{3}
\pcline(P5)(P1)\ncput*{4}
\pcline(P1)(P2)\ncput*{n}
\pcline(P2)(P4)\ncput*{5}
\pcline(P4)(P1)\ncput*{7}
\rput[l]{-10}(P3){\hspace{1,4mm}$\bullet$}
\rput[l]{53}(P4){\hspace{2,4mm}$\bullet$}
\rput[l]{120}(P5){\hspace{1,4mm}$\bullet$}
}
\rput[tl](7.5;-1){
\PstPentagon[unit=1.8,PolyName=P]
\pcline(P2)(P3)\ncput*{1}
\pcline(P3)(P4)\ncput*{2}
\pcline(P4)(P5)\ncput*{3}
\pcline(P5)(P1)\ncput*{4}
\pcline(P1)(P2)\ncput*{n}
\pcline(P4)(P1)\ncput*{7}
\pcline(P3)(P1)\ncput*{8}
\rput[l]{16}(P3){\hspace{3,0mm}$\bullet$}
\rput[l]{72}(P4){\hspace{1,8mm}$\bullet$}
\rput[l]{120}(P5){\hspace{1,4mm}$\bullet$}
}
\rput[l](4,2){\pcline{->}(0,0)(1;0)\Aput{${\bf T}_{k\ell}$}}
\rput[l](9,2){\pcline{->}(0,0)(1;0)\Aput{${\bf T}_{j\ell}$}}
\rput[l](2.2,0){\pcline{->}(0,0)(0.7,-0.7)\Bput{${\bf T}_{jk}$}}
\rput[l](6.5,-1.5){\pcline{->}(0,0)(1;0)\Aput{${\bf T}_{k\ell}$}}
\rput[l](12,0){\pcline{->}(0,0)(-0.7,-0.7)\Aput{${\bf T}_{jk}$}}
\rput[l](3,2.2){$j$}
\rput[l](1.8,1.2){$k$}
\rput[l](0.8,2){$\ell$}
\rput[l](7.9,2.1){$j$}
\rput[l](6.1,1){$k$}
\rput[l](6.5,2.4){$\ell$}
\rput[l](12.5,1.7){$j$}
\rput[l](11.2,1){$k$}
\rput[l](11.9,2.9){$\ell$}
\rput[l](5,-2.7){$j$}
\rput[l](4.7,-1.2){$k$}
\rput[l](3.3,-1.5){$\ell$}
\rput[l](10,-2.7){$j$}
\rput[l](8.7,-2){$k$}
\rput[l](9.3,-0.7){$\ell$}
\end{pspicture}
\caption{The pentagon equation for ${\bf T}$}
\label{4_fig:pentagon_eq}
\end{figure}
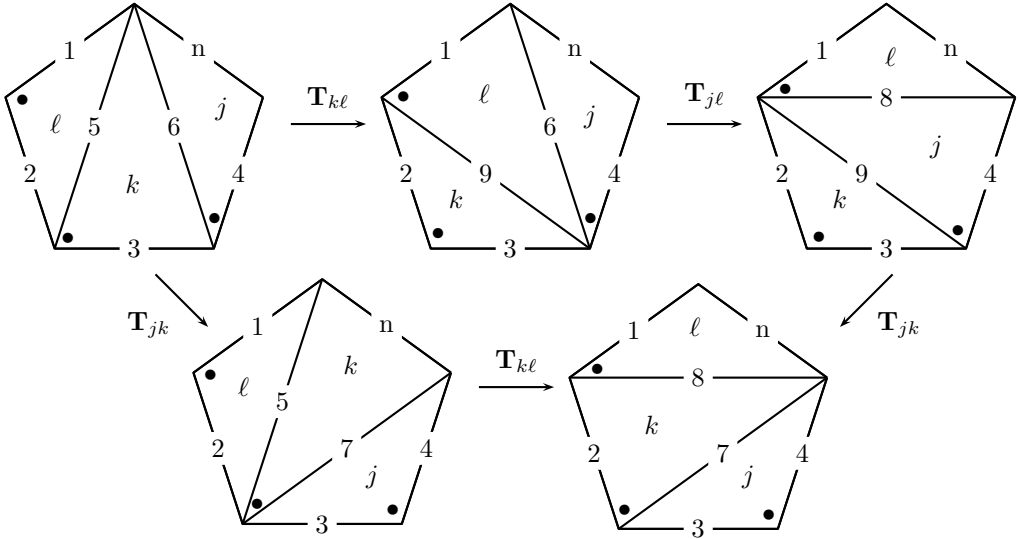
When writing in terms of multiplicity modules, we use $1,2,3$ instead of $j,k,\ell$ (then $1,2,3$ mean first, second, third factors in the tensor product of three multiplicity modules, respectively):
\begin{align}
\label{4_eq:pentagon_M}
\begin{array}{c}
M_{64}^n \otimes M_{53}^6 \otimes M_{12}^5 
\quad \stackrel{{\bf T}_{23}}{\longrightarrow} \quad
M_{64}^n \otimes M_{23}^9 \otimes M_{19}^6 
\quad \stackrel{{\bf T}_{13}}{\longrightarrow} \quad
M_{94}^8 \otimes M_{23}^9 \otimes M_{18}^n  \\
{}_{{\bf T}_{12}}
\searrow \qquad \qquad \qquad
\qquad \qquad \qquad
\qquad \qquad \qquad
\qquad
\swarrow
{}_{{\bf T}_{12}}  \\
M_{34}^7 \otimes M_{57}^n \otimes M_{12}^5
\quad \stackrel{{\bf T}_{23}}{\longrightarrow} \quad
M_{34}^7 \otimes M_{27}^8 \otimes M_{18}^n
\end{array}
\end{align}

\vs

We now formulate the pentagon equation as follows:
\begin{proposition}
\label{4_prop:pentagon}
(the pentagon equation) One has ${\bf T}_{23} {\bf T}_{12} = {\bf T}_{12} {\bf T}_{13} {\bf T}_{23}$.
\end{proposition}

\begin{proof} Observe that all five arrows in \eqref{4_eq:pentagon_H} are identity maps. Using the commutative diagram \eqref{4_eq:T_commutative_diagram} suitably five times for \eqref{4_eq:pentagon_H}, the commutativity of the diagram \eqref{4_eq:pentagon_M} immediately follows.
\end{proof}

\subsection{Dual representation}

For a representation $(\rho,V)$ of $\mathcal{B}$,
the antipode $S$ allows us to turn the dual space $V'$ (we will discuss in the next paragraph which version of a dual space should be used) into a representation of
$\mathcal{B}$ (with action denoted by $\rho'$) via
\begin{align}
\label{4_eq:dual_action}
\langle \rho' (u) \xi, v \rangle
= \langle \xi, \rho (S(u)) v\rangle,
\quad
\forall u \in \mathcal{B},
~
\forall v \in V,
~
\forall \xi \in V',
\end{align}
where $\langle~,~\rangle$ is the natural pairing between $V'$ and $V$.
The reason why \eqref{4_eq:dual_action} yields a well-defined action $\rho'$ is that $S$ is an anti-automorphism (i.e. reverses the product order). Thus, we can replace $S$ by $S^{-1}$ in \eqref{4_eq:dual_action} to get a new action on the dual space, which we denote by ${}'\rho$ (prime on the upper left). 
Then, the following two mappings given by natural pairing commute with the action of
$\mathcal{B}$:
$$
V' \otimes V \to \mathbb{C}, \quad {}' V \otimes V \to \mathbb{C},
$$
where $\mathbb{C}$ on the right-hand-sides is the trivial representation.

\vs

We now study the dual of the representation $\mathcal{H} \equiv L^2(\mathbb{R})$ (as in Proposition \ref{3_prop:repres}). We realize the space $\mathcal{H}'$ as the Hilbert space $L^2(\mathbb{R})$, similarly as we did for $\mathcal{H}$. Using the natural pairing between $\mathcal{H}'$ and $\mathcal{H}$ given by
\begin{align}
\label{4_eq:pairing_with_dual}
\langle g, f \rangle
= \int_\mathbb{R} f(x) g(x) dx,
\end{align}
we formally have from \eqref{4_eq:dual_action} that
\begin{align*}
\langle \pi' (u) g, f \rangle
= \langle g, \pi (S(u)) f\rangle
= \langle (\pi (S(u)))^t  g,  f\rangle,
\quad
\forall u \in \mathcal{B},
~
\forall f \in \mathcal{H},
~
\forall g \in \mathcal{H}',
\end{align*}
where $(\pi (S(u)))^t$ is the transpose of the operator $\pi (S(u))$. This makes sense for $f,g\in W$, for example (see \eqref{2_eq:def_W} for definition of $W$). From this we can define the $\pi'$ action to be
\begin{align}
\label{4_eq:pi_dual_action}
\pi'(X) = e^{-2\pi b p}, ~
\pi'(Y) = - e^{-2\pi b p} e^{2\pi b x}, ~
\pi'(\widetilde{X}) = e^{-2\pi b^{-1} p}, ~
\pi'(\widetilde{Y}) = - e^{-2\pi b^{-1} p} e^{2\pi b^{-1} x},
\end{align}
on a dense subspace of $L^2(\mathbb{R})$.

\begin{remark}
Recall that the subalgebra $\mathcal{B} = \langle X,Y,\widetilde{X},\widetilde{Y} \rangle$ of $B_{q\widetilde{q}}$ doesn't have $X^{-1},\widetilde{X}^{-1}$ as its elements, while $S(u)$ for $u\in \mathcal{B}$ involves $X^{-1}$ or $\widetilde{X}^{-1}$. We can take \eqref{4_eq:dual_action} only in a formal sense, and just use \eqref{4_eq:pi_dual_action} as our definition of the dual representation $\pi'$.
\end{remark}

As done in \S\ref{3_subsection:quantum_plane_modular_double}, we first let the action $\pi'$ as in \eqref{4_eq:pi_dual_action} be defined on the dense subspace $W$, and we define the Schwartz space $\mathcal{S}_{\mathcal{B}'}$ to be the space of all $g\in L^2(\mathbb{R})$ so that
$$
w\mapsto ( g, \pi'(u)w)
$$
on $W$ is continuous in $L^2$-norm, for any chosen $u$ in $\mathcal{B}$, where $(\cdot,\cdot)$ is the inner product of $L^2(\mathbb{R})$. The Schwartz space $\mathcal{S}_{\mathcal{B}'}$ is the common domain of the operators $\pi'(u)$ for $u\in \mathcal{B}$. The space $\mathcal{S}_{\mathcal{B}'}$ is equipped with a topology defined in a similar manner as that for $\mathcal{S}_\mathcal{B}$ (see \S\ref{3_subsection:quantum_plane_modular_double}), and a result analogous to Theorem \ref{3_thm:W_density} can also be proved.

\vs

By replacing $S$ by $S^{-1}$, we obtain the `left dual' ${}'\mathcal{H}$ realized as $L^2(\mathbb{R})$, where the actions are given by
\begin{align*}
{}'\pi(X) = e^{-2\pi b p}, ~
{}'\pi(Y) = - e^{2\pi b x} e^{-2\pi b p}, ~
{}'\pi(\widetilde{X}) = e^{-2\pi b^{-1} p}, ~
{}'\pi(\widetilde{Y}) = - e^{2\pi b^{-1} x} e^{-2\pi b^{-1} p}.
\end{align*}
Again, ${}'\pi$ can first be defined on $W \subset L^2(\mathbb{R})$, and then the Schwartz space $\mathcal{S}_{{}'\mathcal{B}}$ can be defined and given the topology in the similar way as for $\mathcal{S}_\mathcal{B'}$; also, a result analogous to Theorem \ref{3_thm:W_density} can be proved.

\vs

Throughout this subsection, the spaces $\mathcal{H}$, $\mathcal{H}'$, ${}'\mathcal{H}$ will always be realized as $L^2(\mathbb{R}, dx)$, $L^2(\mathbb{R}, dy)$, $L^2(\mathbb{R}, dz)$ with those variable names, respectively, if not designated otherwise (also with spaces indexed by subscripts correspondingly; e.g. $\mathcal{H}_3'$ by $L^2(\mathbb{R}, dy_3)$).  Actually all three $\mathcal{B}$-representations $\mathcal{H}$, $\mathcal{H}'$, ${}'\mathcal{H}$ are isomorphic:

\begin{definition}
For real variables $x,y$, define the following distribution kernels:
\begin{align}
\label{4_eq:dual_kernel_definition}
k(x,y) = e^{\pi i (x-y)^2 + 2\pi a(x-y)}, \quad
K(x,y) = e^{-\pi i (x-y)^2 - 2\pi a(x-y)}.
\end{align}
\end{definition}

\begin{proposition}
\label{4_prop:duals_isomorphisms}
The following maps provide isomorphisms of $\mathcal{B}$-represenations:
\begin{align*}
\begin{array}{cc}
\begin{array}{rrcl}
& \mathcal{H} \equiv L^2(\mathbb{R},dx)
& \leftrightarrow &
\mathcal{H}' \equiv L^2(\mathbb{R}, dy) \\
C': & f(x) 
& \mapsto & 
\int_\mathbb{R} k(x,y) f(x)dx, \\
D': &
\int_\mathbb{R} K(x,y) \varphi(y) dy 
& \mapsfrom &
 \varphi(y),
\end{array}
&
\begin{array}{rrcl}
& \mathcal{H} \equiv L^2(\mathbb{R},dx)
& \leftrightarrow &
{}' \mathcal{H} \equiv L^2(\mathbb{R}, dz) \\
{}'C: & f(x) 
& \mapsto & 
\int_\mathbb{R} k(z,x) f(x) dx, \\
{}'D: &
\int_\mathbb{R} K(z,x) \varphi(z)dz
& \mapsfrom &
 \varphi(z),
\end{array}
\end{array}
\end{align*}
where each map is defined on a dense subspace of $L^2(\mathbb{R})$, e.g. on $W$. Furthermore,
\begin{align}
\label{4_eq:dual_intertwiners_inverses}
(D') (C') = (C') (D') = ({}'D) ({}'C) = ({}'C) ({}'D) = id.
\end{align}
\end{proposition}

\begin{proof}
First, the four maps are well-defined on $W$, because of the decaying property of elements of $W$. From \eqref{4_eq:dual_kernel_definition} we can observe
\begin{align}
\label{4_eq:dual_intertwining_eq}
\begin{array}{ll}
e^{2\pi b^{\pm 1} (p_x+p_y)} k(x,y) = k(x,y), 
& e^{2\pi b^{\pm 1}p_y} k(x,y) = -e^{2\pi b^{\pm 1}(y-x)}k(x,y), \\
e^{2\pi b^{\pm 1} (p_x+p_y)} K(x,y) = K(x,y),
& e^{2\pi b^{\pm 1} p_y}K(x,y) = -e^{2\pi b^{\pm 1}(x-y)}K(x,y),
\end{array}
\end{align}
which yields
\begin{align*}
(C'  \pi(u)w) (y) = \int_\mathbb{R} k(x,y) (\pi_x(u)w(x)) dx
& = \int_\mathbb{R} ((\pi_x(u))^t k(x,y)) w(x) dx \\
& \stackrel{\eqref{4_eq:dual_intertwining_eq}}{=} 
\int_\mathbb{R} (\pi'_y(u) k(x,y)) w(x) dx
= \pi'_y(u) (C'w)(y)
\end{align*}
for $u=X,Y,\widetilde{X},\widetilde{Y}$ and for every $w\in W$, therefore for all $u\in \mathcal{B}$. Hence the intertwining property. It is routine to show
\begin{align}
\label{4_eq:dual_intertwiners_inverses_integral}
\int_\mathbb{R} k(x,y) K(X,y) dy = \delta(x-X), \quad
\int_\mathbb{R} K(x,y) k(x,Y) dx = \delta(y-Y)
\end{align}
as distributions, hence \eqref{4_eq:dual_intertwiners_inverses}.

\vs

To complete the proof, we should prove that each mappping is well-defined on the corresponding Schwartz space, maps the Schwartz space to the other Schwartz space (corresponding to the codomain of the mapping), and intertwines the $\mathcal{B}$-action. It suffices prove the first assertion, i.e. that each mapping is well-defined on the corresponding Schwartz space, because then similar proof as in \S2.4 of \cite{Goncharov} yields the latter two results. It's easy to see that each of the four mappings differ from a unitary mapping by $e^{-2\pi ap}$ (times a constant of modulus $1$); for example, $C' = e^{\pi i a^2} C'_0 e^{-2\pi a p}$, where $C'_0$ is a unitary mapping defined by
$$
C'_0 : f(x) \mapsto \int_\mathbb{R} e^{\pi i (x-y)^2} f(x) dx.
$$
Now, we modify the Schwartz space $\mathcal{S}_\mathcal{B}$ by replacing $\mathcal{B}$ by $\mathcal{B} \cup \{ (X\widetilde{X})^{1/2} \}$ in Definition \ref{3_def:Schwartz_space}, in the sense that we require for any $f$ in the new Schwartz space that the functional $w\mapsto (f, e^{-2\pi ap} w)$ on $W$ be continuous in addition to the functionals $w\mapsto (f, \pi(u) w)$ for $u\in \mathcal{B}$. Then the operator $e^{-2\pi ap}$ is well-defined on the new Schwartz space. We also modify $\mathcal{S}_{\mathcal{B}'}$ and $\mathcal{S}_{{}'\mathcal{B}}$ in a similar manner. Then the four mappings are well-defined on corresponding Schwartz spaces, hence a similar proof as in \S2.4 of \cite{Goncharov} yields the desired results, as mentioned.
\end{proof}

\begin{proposition}
\label{4_prop:dual_intertwiners_property}
The operators $C'$, $D'$, ${}'C$, ${}'D$ defined in Proposition \ref{4_prop:duals_isomorphisms} satisfy
\begin{align}
\label{4_eq:dual_intertwiners_doubles}
& (C') (C') = (C') ({}' D) = ({}' D) (C') = e^{-4\pi a p},
\quad
(D') (D') = (D') ({}' C) = ({}' C) (D') = e^{4\pi a p}, \\
\label{4_eq:dual_intertwiners_transpose}
& 
\langle (C'g)(y), (D'f)(x)\rangle = \langle e^{-4\pi a p_y} g(y), f(x)\rangle,
\quad
\langle ({}'C g)(z), ({}'D f)(x) \rangle
= \langle g(z), e^{-4\pi a p_x}f(x)\rangle, \\
\label{4_eq:dual_intertwiners_transpose_triv}
& \langle (C'g)(y), ({}'D f)(x)\rangle
= \langle g(y), f(x)\rangle.
\end{align}
Observe that \eqref{4_eq:dual_intertwiners_doubles} provides isomorphisms $\mathcal{H} \cong \mathcal{H}''$ and $\mathcal{H} \cong {}''\mathcal{H}$. Proof of \eqref{4_eq:dual_intertwiners_doubles}--\eqref{4_eq:dual_intertwiners_transpose_triv} is straightforward from $k(y,x) = e^{4\pi a p_y} k(x,y)$, $K(y,x) = e^{4\pi a p_y} K(x,y)$, and \eqref{4_eq:dual_intertwiners_inverses_integral}. $\qed$
\end{proposition}

\vs

We will now construct an operator ${\bf A} : M_{12}^3 \to M_{23}^1$ in a canonical way, where
\begin{align}
\label{4_eq:two_M_definitions}
\mathcal{H}_1 \otimes \mathcal{H}_2 \cong M_{12}^3 \otimes \mathcal{H}_3, \quad
\mathcal{H}_2 \otimes \mathcal{H}_3 \cong M_{23}^1 \otimes \mathcal{H}_1,
\end{align}
and the isomorphisms in \eqref{4_eq:two_M_definitions} are in the sense of Theorem \ref{4_thm:CG}. Using \eqref{4_eq:two_M_definitions}, we have a canonical identification of $M_{12}^3$ with $Hom_{\mathcal{B}} (\mathcal{H}_3, \mathcal{H}_1 \otimes \mathcal{H}_2)$; an element $\phi(\alpha) \in W\subset M_{12}^3 \equiv L^2(\mathbb{R}, d\alpha)$ gives rise to the following element of $Hom_{\mathcal{B}}(\mathcal{H}_3, \mathcal{H}_1 \otimes \mathcal{H}_2)$:
\begin{align}
\label{4_eq:identification_I_formula}
f(x_3) \mapsto \int_{\mathbb{R}^2} \phi(\alpha) f(x_3) \upker{\alpha}{x_3}{x_1}{x_2} d\alpha dx_3^{\cup x_2}.
\end{align}
We denote this isomorphism by
\begin{align}
\label{4_eq:identification_I}
I_{123} : M_{12}^3 \to Hom_{\mathcal{B}}(\mathcal{H}_3, \mathcal{H}_1 \otimes \mathcal{H}_2).
\end{align}
(similarly for $I_{231}$) We consider another canonical identification
\begin{align}
\label{4_eq:identification_J}
J_{123} : Hom_\mathbb{C}(\mathcal{H}_3, \mathcal{H}_1 \otimes \mathcal{H}_2)
\to \mathcal{H}_1 \otimes \mathcal{H}_2 \otimes \mathcal{H}_3',
\end{align}
using the pairing map $\mathcal{H}_3' \otimes \mathcal{H}_3\to \mathbb{C}$. Restriction of $J_{123}$ to the $\mathcal{B}$-invariant subspace yields a mapping from  $Hom_{\mathcal{B}} (\mathcal{H}_3, \mathcal{H}_1 \otimes \mathcal{H}_2)$ to $Inv(\mathcal{H}_1 \otimes \mathcal{H}_2 \otimes \mathcal{H}_3')$, denoted by $J_{123}$ again.

\vs

We construct another canonical map 
\begin{align}
\label{4_eq:cyclic_map_A}
A^{Inv}_{123}: Inv(\mathcal{H}_1 \otimes \mathcal{H}_2 \otimes \mathcal{H}_3') \to Inv(\mathcal{H}_2 \otimes \mathcal{H}_3 \otimes \mathcal{H}_1')
\end{align}
as the restriction to the $\mathcal{B}$-invariant subspace of the mapping
$$
A_{123}: \mathcal{H}_1  \otimes \mathcal{H}_2 \otimes \mathcal{H}_3'  \to
\mathcal{H}_2 \otimes \mathcal{H}_3 \otimes \mathcal{H}_1',
$$
given by the following composition:
\begin{align}
\label{4_eq:composition_three}
\xymatrixcolsep{3pc}
\xymatrix{
\mathcal{H}_1 \otimes \mathcal{H}_2 \otimes \mathcal{H}_3' 
\ar[r]^{ {}'C \otimes id \otimes D'} &
{}' \mathcal{H}_1 \otimes \mathcal{H}_2 \otimes \mathcal{H}_3
\ar[r]^{ E_{123} } &
\mathcal{H}_2 \otimes \mathcal{H}_3 \otimes \mathcal{H}_1',
}
\end{align}
where the first arrow just comes from Proposition \ref{4_prop:duals_isomorphisms}, and the more important second arrow $E_{123}$ comes from the two canonical $\mathcal{B}$-isomorphisms (defined for any three $\mathcal{B}$-representations $V_1,V_2,V_3$)
\begin{align}
\label{4_eq:canonical_isomorphisms_Hom_and_primes}
{}' V_1\otimes V_2 \otimes V_3
\cong
Hom_\mathbb{C}(V_1,V_2 \otimes V_3)
\cong
V_2 \otimes V_3 \otimes V_1',
\end{align}
where the first isomorphism is due to the pairing map $V_1 \otimes {}'V_1 \to \mathbb{C}$, and the second one is due to the pairing map $V_1' \otimes V_1 \to \mathbb{C}$; thus $E_{123}$ is given by
$$
E_{123}: \varphi(z_1,x_2,x_3) \mapsto \varphi(y_1,x_2,x_3).$$

\vs

Now, define the map
$
A^{Hom}_{123} : Hom_{\mathcal{B}}(\mathcal{H}_3, \mathcal{H}_1 \otimes \mathcal{H}_2) \to Hom_{\mathcal{B}}(\mathcal{H}_1, \mathcal{H}_2 \otimes \mathcal{H}_3)
$ to be the unique map which makes the following diagram commute:
\begin{align}
\label{4_eq:A_Hom_definition}
\xymatrix{
Hom_{\mathcal{B}}(\mathcal{H}_3, \mathcal{H}_1 \otimes \mathcal{H}_2) 
\ar[d]_{J_{123}} \ar[r]^{A^{Hom}_{123}} &
Hom_{\mathcal{B}}(\mathcal{H}_1, \mathcal{H}_2 \otimes \mathcal{H}_3)
\ar[d]_{J_{231}} \\
Inv(\mathcal{H}_1 \otimes \mathcal{H}_2 \otimes \mathcal{H}_3')
\ar[r]^{A^{Inv}_{123}} &
Inv(\mathcal{H}_2 \otimes \mathcal{H}_3 \otimes \mathcal{H}_1').
}
\end{align}
Finally, our map ${\bf A}$ is defined as follows:
 
\begin{definition}
\label{4_def:A_operator}
The operator ${\bf A}: M_{12}^3 \to M_{23}^1$ is defined to be the unique mapping which makes the following diagram commute:
\begin{align}
\label{4_eq:def_A_operator_diagram}
\xymatrix{
M_{12}^3 \ar[r]^{{\bf A}} \ar[d]_{I_{123}} &
M_{23}^1 \ar[d]^{I_{231}} \\
Hom_{\mathcal{B}}(\mathcal{H}_3, \mathcal{H}_1 \otimes \mathcal{H}_2) 
\ar[r]^{A^{Hom}_{123}} &
Hom_{\mathcal{B}}(\mathcal{H}_1, \mathcal{H}_2 \otimes \mathcal{H}_3).
}
\end{align}
\end{definition}

\vs

Before computing the formula for ${\bf A}$, we can prove the following result:

\begin{proposition}
\label{4_prop:A_orderthree}
One has ${\bf A}^3 = id$ on $M_{12}^3$.
\end{proposition}

\begin{proof}
Consider the following diagram:
\begin{align}
\label{4_eq:large_triangular_diagram}
\xymatrix{
\mathcal{H}_1 \otimes \mathcal{H}_2 \otimes \mathcal{H}_3'
\ar[r]^{A_{123}} \ar[dr] &
\mathcal{H}_2 \otimes \mathcal{H}_3 \otimes \mathcal{H}_1'
\ar[r]^{A_{231}} \ar[dr] &
\mathcal{H}_3 \otimes \mathcal{H}_1 \otimes \mathcal{H}_2'
\ar[r]^{A_{312}} \ar[dr] &
\mathcal{H}_1 \otimes \mathcal{H}_2 \otimes \mathcal{H}_3' \\
&
{}' \mathcal{H}_1 \otimes \mathcal{H}_2 \otimes \mathcal{H}_3
\ar[u]^{E_{123}}  &
{}' \mathcal{H}_2 \otimes \mathcal{H}_3 \otimes \mathcal{H}_1
\ar[u]^{E_{231}}  &
{}' \mathcal{H}_3 \otimes \mathcal{H}_1 \otimes \mathcal{H}_2,
\ar[u]^{E_{312}}
}
\end{align}
where the three diagonal arrows are all ${}' C \otimes id \otimes D'$ as in \eqref{4_eq:composition_three}. Recall that $A_{ijk}$ are defined so that the diagram \eqref{4_eq:large_triangular_diagram} commutes. 

\vs

Observe that the composition of the six non-horizontal arrows of the diagram \eqref{4_eq:large_triangular_diagram} is
\begin{align*}
\mathcal{H}_1 \otimes \mathcal{H}_2 \otimes \mathcal{H}_3'
& \to 
\mathcal{H}_1 \otimes \mathcal{H}_2 \otimes \mathcal{H}_3'
\\
\varphi(x_1,x_2,y_3) & \mapsto
((D'_2) ({}'C_3) (D'_1) ({}'C_2) (D'_3) ({}'C_1) \varphi)(x_1,x_2,y_3) \\
& = ( ({}'C_3)(D'_3)  (D'_2)({}'C_2)  (D'_1)({}'C_1) \varphi)(x_1,x_2,y_3) \\
& = e^{4\pi a (p_{x_1}+p_{x_2}+p_{y_3})} \varphi(x_1,x_2,y_3),
\quad (\mbox{by } \eqref{4_eq:dual_intertwiners_doubles})
\end{align*}
where the subscripts of ${}'C$ and $D'$ indicate the variable names (hence operators having distinct subscripts commute)
and the domain $\mathcal{H}_1 \otimes \mathcal{H}_2 \otimes \mathcal{H}_3'$ and the codomain $\mathcal{H}_1 \otimes \mathcal{H}_2 \otimes \mathcal{H}_3'$ are both realized as $L^2(\mathbb{R}^3, dx_1 dx_2 dy_3)$. Notice that if $\varphi(x_1,x_2,y_3) \in Inv(\mathcal{H}_1 \otimes \mathcal{H}_2 \otimes \mathcal{H}_3')$, then from the invariance under the action of $X, \widetilde{X}$ we have $e^{2\pi b^{\pm 1}(p_{x_1} + p_{x_2} + p_{y_3})} \varphi(x_1,x_2,y_3) = \varphi (x_1,x_2,y_3)$, which yields $e^{4\pi a (p_{x_1} + p_{x_2} + p_{y_3})} \varphi(x_1,x_2,y_3) = \varphi(x_1,x_2,y_3)$.

\vs

Thus, the composition of the six non-horizontal arrows of the diagram \eqref{4_eq:large_triangular_diagram} applied to invariant elements is the identity map (if the first and last spaces are considered identical), and hence so is that of the three horizontal arrows. Following the definition of ${\bf A}$ as in \eqref{4_eq:def_A_operator_diagram} (for three possible cyclic permutations of the indices) and also \eqref{4_eq:A_Hom_definition}, we get that ${\bf A} \circ {\bf A} \circ {\bf A}: M_{12}^3 \to M_{23}^1 \to M_{31}^2 \to M_{12}^3$ is the identity map (if the first and last $M_{12}^3$ are considered identical). 
\end{proof}

We can compute the explicit formula for ${\bf A}$, using its definition \eqref{4_eq:def_A_operator_diagram}:

\begin{proposition}
\label{4_prop:A_formula}
The operator ${\bf A}:L^2(\mathbb{R}, d\alpha)\to L^2(\mathbb{R}, d\beta)$ is given by:
\begin{align}
\label{4_eq:A_formula}
{\bf A}: \phi(\alpha) \mapsto \frac{\zeta e^{-\pi i /12}}{\sqrt{3}} 
\int_{\mathbb{R}} e^{\pi i(2\alpha^2/3 + 2\alpha\beta/3 - \beta^2/3)} (e^{2\pi a p_\alpha} \phi(\alpha)) d\alpha, \quad \mbox{for } \phi \in W.
\end{align}
\end{proposition}

\begin{proof}
The diagram \eqref{4_eq:def_A_operator_diagram} and \eqref{4_eq:A_Hom_definition} together define ${\bf A}$ by
\begin{align}
\label{4_eq:A_definition_equation}
J_{231} I_{231} {\bf A} \phi
= A^{Inv}_{123} J_{123} I_{123} \phi, \quad \phi \in W\subset M_{12}^3.
\end{align}
Following the definitions of $I_{123}$ (as in \eqref{4_eq:identification_I} and \eqref{4_eq:identification_I_formula}), $J_{123}$ (as in \eqref{4_eq:identification_J}), and $A^{Inv}_{123}$ (as in \eqref{4_eq:composition_three}), and considering $M_{12}^3 \equiv L^2(\mathbb{R}, d\alpha)$ and $M_{23}^1 \equiv L^2(\mathbb{R}, d\beta)$, the equation \eqref{4_eq:A_definition_equation} means
\begin{align*}
\int_\mathbb{R} ({\bf A} \phi)(\beta) \upker{\beta}{y_1}{x_2}{x_3} d\beta
= \int_{\mathbb{R}^3} \phi(\alpha) \upker{\alpha}{y_3}{x_1}{x_2} K(x_3,y_3) k(y_1,x_1) d\alpha dy_3^{\cup x_2} dx_1.
\end{align*}
Multiply $\downker{\sigma}{Y_1}{x_2}{x_3}$ to both sides and integrate w.r.t. $x_2,x_3$ along a suitable contour: from \eqref{4_eq:CG_orthogonality_x1_x2}, we get 
\begin{align}
\label{4_eq:to_integrate_A}
({\bf A} \phi)(\sigma) \delta(y_1-Y_1)
= \int_{\mathbb{R}^5} \phi(\alpha) \upker{\alpha}{y_3}{x_1}{x_2} K(x_3,y_3) k(y_1,x_1) \downker{\sigma}{Y_1}{x_2}{x_3} d\alpha dy_3^{\cup x_2} dx_1 dx_3^{\cup Y_1} dx_2,
\end{align}
as distributions. We put $e^{-2\pi a p_\alpha} \phi(\alpha)$ into the place of $\phi(\alpha)$, and let $R(\sigma,y_1,Y_1)$ be the RHS of \eqref{4_eq:to_integrate_A}. By putting in all the definitions (see \eqref{4_eq:def_downker}, \eqref{4_eq:def_upker}, and \eqref{4_eq:dual_kernel_definition}) and simplifying, we get
$$
R(\sigma,y_1,Y_1) = \int_{\mathbb{R}^5} \phi(\alpha) G(x_2-y_3-ia) G(Y_1 - x_3-ia) \exp(*A_1) d\alpha dy_3^{\cup x_2} dx_1 dx_3^{\cup Y_1} dx_2,
$$
(we moved $e^{-2\pi a p_\alpha}$ by transposing w.r.t. $\alpha$ variable; shift the contour of $\alpha$ by $- ia$) where
\begin{align*}
(*A_1) / (\pi i)
& =
( -x_1^2 +3x_2^2/2 - x_3^2/2 + y_1^2 +Y_1^2/2 - 3y_3^2/2)
+ 2x_1( \alpha  + x_2 - y_1  + y_3  ) \\
& \quad
-2\alpha y_3+ 2\sigma Y_1- 2\sigma x_2
- 2x_2 x_3 - 2x_2 Y_1 - x_2 y_3+ x_3Y_1  + 2 x_3 y_3 \\
& \quad
- ia (-x_2-3x_3+2y_1+Y_1+y_3).
\end{align*}
We first integrate w.r.t. $x_1$, using
\begin{align}
\label{4_eq:finite_integral}
\int_\mathbb{R} e^{\pi i (r x^2 + sx)} dx
= \int_\mathbb{R} e^{\pi i (r (x+\frac{s}{2r})^2 - \frac{s^2}{4r})} dx
= \frac{e^{\pi i/4} e^{- \pi i \frac{s^2}{4r}}}{\sqrt{r}}, \quad
r>0,
\end{align}
which holds as distributions. By taking complex conjugate, we get a similar formula for the case $r<0$, which we will also refer to as just \eqref{4_eq:finite_integral} throughout this paper. Performing \eqref{4_eq:finite_integral} for the variable $x_1$ yields
\begin{align*}
R(\sigma,y_1,Y_1) & = e^{-\pi i/4} \int_{\mathbb{R}^4} \phi(\alpha) G(x_2-y_3-ia) G(Y_1 - x_3-ia) \exp(*A_2) d\alpha dy_3^{\cup x_2} dx_3^{\cup Y_1} dx_2, \\
(*A_2) / (\pi i)
& =
( \alpha^2  +5x_2^2/2 - x_3^2/2 + 2y_1^2 +Y_1^2/2 - y_3^2/2)
+ y_3(x_2  + 2 x_3 - 2y_1) \\
& \quad
+ 2\alpha x_2 - 2\alpha y_1 + 2\sigma Y_1- 2\sigma x_2
- 2x_2 x_3  - 2x_2 y_1 - 2x_2 Y_1+ x_3Y_1 \\
& \quad
- ia (-x_2-3x_3+2y_1+Y_1+y_3).
\end{align*}
By a simple change of variables for the Fourier transform formula \eqref{2_eq:G_Four_-_-} we get
\begin{align*}
\int_\mathbb{R} G(x_2-y_3-ia)e^{-\frac{\pi i}{2}y_3^2} e^{\pi i  y_3(x_2+2u)} e^{\pi a y_3} dy_3^{\cup x_2}
= e^{-2i\chi}e^{-\pi i/4} G(u-ia) e^{2\pi i x_2 u} e^{\frac{\pi i}{2} (u^2+x_2^2)} e^{\pi a (u+x_2)},
\end{align*}
for $u = x_3-y_1$:
\begin{align*}
R(\sigma,y_1,Y_1) & = e^{-2i\chi} e^{-\pi i/2} \int_{\mathbb{R}^3} \phi(\alpha) G(x_3-y_1-ia) G(Y_1 - x_3-ia) \exp(*A_3) d\alpha dx_3^{\cup Y_1\cap y_1} dx_2, \\
(*A_3) / (\pi i)
& =
( \alpha^2  + 3x_2^2 + 5y_1^2/2 +Y_1^2/2)
+ x_3(Y_1-y_1) \\
& \quad
+ (2\alpha x_2 - 2\alpha y_1 + 2\sigma Y_1- 2\sigma x_2
- 4x_2 y_1 - 2x_2 Y_1)
- ia (-2x_3+y_1+Y_1).
\end{align*}
From \eqref{2_eq:G_delta2} we have
$
\int_\mathbb{R} G(x_3 - y_1-ia) G(Y_1 - x_3-ia)
e^{\pi i x_3 (Y_1 - y_1)} e^{-2\pi a x_3} dx_3^{\cup Y_1 \cap y_1}
= \delta( Y_1- y_1) e^{-\pi a (y_1 + Y_1)},
$
hence
\begin{align*}
R(\sigma,y_1,Y_1) & = e^{-2i\chi} e^{-\pi i/2} \delta(y_1-Y_1) \int_{\mathbb{R}^2} \phi(\alpha) \exp(*A_4) d\alpha dx_2, \\
(*A_4) / (\pi i)
& =
( \alpha^2  + 3x_2^2 + 5y_1^2/2 +Y_1^2/2)
+ (2\alpha x_2 - 2\alpha y_1 + 2\sigma Y_1- 2\sigma x_2
- 4x_2 y_1 - 2x_2 Y_1).
\end{align*}
Because of the $\delta(y_1-Y_1)$ factor, we can replace all $Y_1$'s by $y_1$'s in $(*A_4)$; so $(*A_4)$ can be replaced by $\pi i (\alpha^2  + 3(x_2-y_1)^2 + 2(\alpha-\sigma) (x_2 - y_1))$. Use change of variables $x_2 \mapsto X_2 = x_2 - y_1$, and integrate w.r.t. $X_2$ using \eqref{4_eq:finite_integral}:
\begin{align*}
R(\sigma,y_1,Y_1) & = \frac{e^{-2i\chi} e^{-\pi i/4}}{\sqrt{3}}  \delta(y_1-Y_1) \int_{\mathbb{R}} \phi(\alpha) e^{\pi i (2\alpha^2/3 + 2\alpha \sigma/3 - \sigma^2/3 )} d\alpha.
\end{align*}
Therefore we just proved
\begin{align}
\label{4_eq:A_formula_clean}
({\bf A} e^{-2\pi a p_\alpha} \phi(\alpha))(\sigma)
= \frac{\zeta e^{-\pi i /12}}{\sqrt{3}} \int_\mathbb{R} e^{\pi i (2\alpha^2/3+2\alpha\sigma/3-\sigma^2/3)} \phi(\alpha) d\alpha
\end{align}
(recall $\zeta = e^{-2i\chi} e^{-\pi i /6}$), which amounts to \eqref{4_eq:A_formula}.
\end{proof}

We now prove the following useful property of ${\bf A}$:
\begin{proposition}
\label{4_prop:A_conjugation_action}
If ${\bf A}, p,x$ are thought of as acting on $L^2(\mathbb{R},dx)$, then
\begin{align}
\label{4_eq:A_conjugation_action}
{\bf A} x {\bf A}^{-1} = x+3p-ia, \quad
{\bf A} p {\bf A}^{-1} = -x-2p.
\end{align}
\end{proposition}

\begin{proof}
We use \eqref{4_eq:A_formula_clean}. For convenience, let ${\bf A}^{(0)} := \zeta^{-1} {\bf A} e^{-2\pi a p_\alpha}$ (this definition is in accordance with the one that will appear later), so that ${\bf A}^{(0)} : L^2(\mathbb{R}, d\alpha) \to L^2(\mathbb{R}, d\beta)$ is an (unitary) integral transformation with distribution kernel $c_0 J(\alpha,\beta)$, where $c_0 = e^{-\pi i/12}/\sqrt{3}$ and $J(\alpha,\beta) = e^{\pi i (2\alpha^2/3 + 2\alpha\beta/3 - \beta^2/3)}$. Note for any real number $m$ that
\begin{align*}
e^{2\pi i m p_\alpha} J(\alpha,\beta)
= J(\alpha+m,\beta)
= e^{\pi i (2m^2/3)}e^{\pi i (4m \alpha/3 + 2m\beta/3)} J(\alpha,\beta)
\\
e^{2\pi i m p_\beta} J(\alpha,\beta)
= J(\alpha,\beta+m)
= e^{\pi i (-m^2/3)} e^{\pi i ( 2m \alpha/3 - 2m \beta/3)} J(\alpha,\beta),
\end{align*}
hence
$$
e^{2\pi i m (-2\alpha/3 + p_\alpha)} J(\alpha,\beta)
= e^{2\pi i m (\beta/3)} J(\alpha,\beta), \quad
e^{2\pi i m (\alpha/3)} J(\alpha,\beta)
= e^{2\pi i m(\beta/3 + p_\beta)} J(\alpha,\beta).
$$
For $\phi \in W$'s, we thus get (by transposing)
$$
{\bf A}^{(0)} e^{2\pi i m (-2\alpha/3 - p_\alpha)} \phi
= e^{2\pi i m (\beta/3)} {\bf A}^{(0)} \phi, \quad
{\bf A}^{(0)} e^{2\pi i m (\alpha/3)} \phi
= e^{2\pi i m(\beta/3 + p_\beta)} {\bf A}^{(0)} \phi,
$$
yielding
\begin{align}
\label{4_eq:A_0_conjugation}
{\bf A}^{(0)} (-2x/3 - p) ({\bf A}^{(0)})^{-1} = x/3, \quad
{\bf A}^{(0)} (x/3) ({\bf A}^{(0)})^{-1} = x/3+p,
\end{align}
if ${\bf A}^{(0)},x,p$ are thought of as acting on $L^2(\mathbb{R}, dx)$. It's easy to see
$$
e^{-2\pi a p} x e^{2\pi a p} = x+ia, \quad
e^{-2\pi a p} p e^{2\pi a p} = p,
$$
which together with \eqref{4_eq:A_0_conjugation} yields \eqref{4_eq:A_conjugation_action}.
\end{proof}

\begin{remark}
Proposition \ref{4_prop:A_conjugation_action} could've been proved without Proposition \ref{4_prop:A_formula}, using the equation \eqref{4_eq:A_definition_equation} $J_{231} I_{231} {\bf A} \phi = A^{Inv}_{123} J_{123} I_{123} \phi$ and the following (for any real numbers $m$):
\begin{align*}
e^{2\pi i m p_{x_1}} \upker{\alpha}{y_3}{x_1}{x_2}
& =e^{2\pi i m (-p_{y_3}-p_{x_2})} \upker{\alpha}{y_3}{x_1}{x_2}
= e^{-2\pi i m^2} e^{2\pi i m\alpha} e^{2\pi i m(x_2+y_3-2x_1)} \upker{\alpha}{y_3}{x_1}{x_2}, \\
e^{2\pi i m p_\alpha} \upker{\alpha}{y_3}{x_1}{x_2}
& = e^{-2\pi i m(y_3-x_1)} \upker{\alpha}{y_3}{x_1}{x_2}, \\
e^{2\pi i m p_{x_1}} k(y_1,x_1)
& = e^{\pi i m^2 - 2\pi a m} e^{2\pi i m(x_1-y_1)} k(y_1,x_1), \\
e^{2\pi i m p_{y_3}} K(x_3,y_3)
& = e^{-\pi i m^2 + 2\pi a m} e^{2\pi i m(x_3-y_3)} K(x_3,y_3).
\end{align*}
Actually, \eqref{4_eq:A_conjugation_action} uniquely determines the operator ${\bf A}$ (constant is fixed by ${\bf A}^3=id$).
\end{remark}

The ${\bf A}$ operator can be encoded geometrically as in Figure \ref{4_fig:A}.
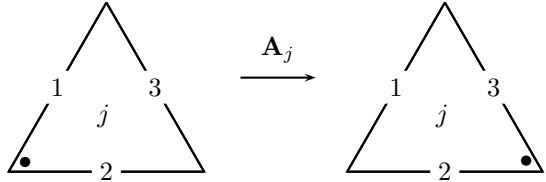
\begin{figure}[htbp!]
\centering
\begin{pspicture}[showgrid=false](0,0.8)(8,3)
\rput[bl](0,0){
\PstTriangle[unit=1.5,PolyName=P]
\pcline(P1)(P2)\ncput*{1}
\pcline(P2)(P3)\ncput*{2}
\pcline(P3)(P1)\ncput*{3}
\rput[l]{30}(P2){\hspace{1,7mm}$\bullet$}
}
\rput[l](1.5,1.5){$j$}
\rput[bl](4.5,0){
\PstTriangle[unit=1.5,PolyName=P]
\pcline(P1)(P2)\ncput*{1}
\pcline(P2)(P3)\ncput*{2}
\pcline(P3)(P1)\ncput*{3}
\rput[l]{145}(P3){\hspace{1,7mm}$\bullet$}
}
\rput[l](6,1.5){$j$}
\rput[l](3.4,2){\pcline{->}(0,0)(1;0)\Aput{${\bf A}_{j}$}}
\end{pspicture}
\caption{The move representing ${\bf A}:M_{12}^3 \to M_{23}^1$}
\label{4_fig:A}
\end{figure}
We will now study some more properties of the ${\bf A}$ operator in the next subsection.

\subsection{Relations involving ${\bf T}$ and ${\bf A}$}

Recall that we proved that ${\bf T}$ satisfies pentagon equation (Proposition \ref{4_prop:pentagon}), and that ${\bf A}^3 = id$ (Proposition \ref{4_prop:A_orderthree}). In this subsection, we prove two relations involving both ${\bf T}$ and ${\bf A}$.

\vs

We can check if the diagram
\begin{align}
\label{4_eq:diagram_ATA_ATA}
\xymatrix{
M_{12}^n \otimes M_{n3}^\ell
\ar[r]^{{\bf T}_{21}} \ar[d]_{{\bf A}_1^{-1} {\bf A}_2} &
M_{1m}^\ell \otimes M_{23}^m \ar[d]_{{\bf A}_1^{-1} {\bf A}_2} \\
M_{n1}^2 \otimes M_{3\ell}^n \ar[r]^{{\bf T}_{12}} &
M_{\ell1}^m \otimes M_{3m}^2
}
\end{align}
commutes.  The above diagram \eqref{4_eq:diagram_ATA_ATA} can be geometrically encoded as in Figure \ref{4_fig:ATA_ATA}.
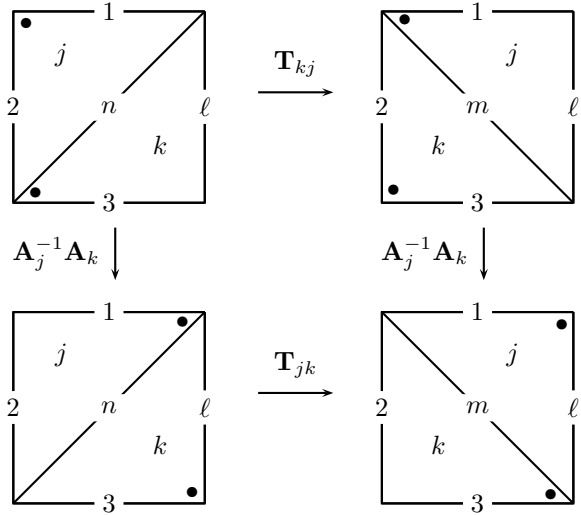
\begin{figure}[htbp!]
\centering
\begin{pspicture}[showgrid=false](0,-3.4)(8.4,3.1)
\rput[bl](0;0){
\PstSquare[unit=1.8,PolyName=P]
\pcline(P1)(P2)\ncput*{1}
\pcline(P2)(P3)\ncput*{2}
\pcline(P3)(P4)\ncput*{3}
\pcline(P4)(P1)\ncput*{$\ell$}
\pcline(P1)(P3)\ncput*{$n$}
\rput[l]{-45}(P2){\hspace{1,4mm}$\bullet$}
\rput[l]{22}(P3){\hspace{2,4mm}$\bullet$}
}
\rput[l](2.5,1.3){$k$}
\rput[l](1.2,2.5){$j$}
\rput[bl](0,-4){
\PstSquare[unit=1.8,PolyName=P]
\pcline(P1)(P2)\ncput*{1}
\pcline(P2)(P3)\ncput*{2}
\pcline(P3)(P4)\ncput*{3}
\pcline(P4)(P1)\ncput*{$\ell$}
\pcline(P1)(P3)\ncput*{$n$}
\rput[l]{202}(P1){\hspace{2,4mm}$\bullet$}
\rput[l]{135}(P4){\hspace{1,4mm}$\bullet$}
}
\rput[l](2.5,-2.7){$k$}
\rput[l](1.2,-1.5){$j$}
\rput[bl](4.9,0){
\PstSquare[unit=1.8,PolyName=P]
\pcline(P1)(P2)\ncput*{1}
\pcline(P2)(P3)\ncput*{2}
\pcline(P3)(P4)\ncput*{3}
\pcline(P4)(P1)\ncput*{$\ell$}
\pcline(P2)(P4)\ncput*{$m$}
\rput[l]{-22}(P2){\hspace{2,4mm}$\bullet$}
\rput[l]{45}(P3){\hspace{1,4mm}$\bullet$}
}
\rput[l](6.2,1.3){$k$}
\rput[l](7.2,2.5){$j$}
\rput[bl](4.9,-4){
\PstSquare[unit=1.8,PolyName=P]
\pcline(P1)(P2)\ncput*{1}
\pcline(P2)(P3)\ncput*{2}
\pcline(P3)(P4)\ncput*{3}
\pcline(P4)(P1)\ncput*{$\ell$}
\pcline(P2)(P4)\ncput*{$m$}
\rput[l]{-135}(P1){\hspace{1,4mm}$\bullet$}
\rput[l]{155}(P4){\hspace{2,4mm}$\bullet$}
}
\rput[l](6.2,-2.7){$k$}
\rput[l](7.2,-1.5){$j$}
\rput[l](3.9,2){\pcline{->}(0,0)(1;0)\Aput{${\bf T}_{kj}$}}
\rput[l](2,0.2){\pcline{->}(0,0)(0,-0.7)\Bput{${\bf A}_j^{-1} {\bf A}_k$}}
\rput[l](3.9,-2){\pcline{->}(0,0)(1;0)\Aput{${\bf T}_{jk}$}}
\rput[l](6.9,0.2){\pcline{->}(0,0)(0,-0.7)\Bput{${\bf A}_j^{-1} {\bf A}_k$}}
\end{pspicture}
\caption{Geometric realization of the diagram \eqref{4_eq:diagram_ATA_ATA}}
\label{4_fig:ATA_ATA}
\end{figure}

\vs

Before formulating and proving this assertion, we introduce another formulation of ${\bf T}$ operator, for convenience. Suppose that the situation is as in \eqref{4_eq:T_commutative_diagram}, and define ${\bf T}^{Hom}$ to be the unique mapping which makes the following diagram to commute:
\begin{align}
\label{4_eq:T_Hom_diagram1}
\xymatrix{
{\begin{array}{l}
Hom_{\mathcal{B}} (\mathcal{H}_6, \mathcal{H}_4 \otimes \mathcal{H}_3) \\
\otimes Hom_{\mathcal{B}} (\mathcal{H}_4, \mathcal{H}_1 \otimes \mathcal{H}_2)
\end{array}}
\ar[r]^{{\bf T}^{Hom}}  &
{\begin{array}{l}
Hom_{\mathcal{B}} (\mathcal{H}_5, \mathcal{H}_2 \otimes \mathcal{H}_3) \\
\otimes Hom_{\mathcal{B}} (\mathcal{H}_6, \mathcal{H}_1 \otimes \mathcal{H}_5)
\end{array}} \\
M_{43}^6 \otimes M_{12}^4 \ar[r]^{{\bf T}} \ar[u]_{I_{436} \otimes I_{124}} &
M_{23}^5 \otimes M_{15}^6, \ar[u]^{I_{235} \otimes I_{156}}
}
\end{align}
where ${\bf T}$ is as in \eqref{4_eq:basic_T} and $I_{jk\ell}$ are as in \eqref{4_eq:identification_I}. Using the similar idea as in \eqref{4_eq:identification_I}, from the two explicit isomorphisms \eqref{4_eq:first_decomp} and \eqref{4_eq:second_decomp} (as studied in that subsection) we get a canonical identification of $M_{43}^6 \otimes M_{12}^4$ with $Hom_{\mathcal{B}}(\mathcal{H}_6, (\mathcal{H}_1\otimes \mathcal{H}_2)\otimes \mathcal{H}_3)$ and that of $M_{23}^5 \otimes M_{15}^6$ with $Hom_{\mathcal{B}}(\mathcal{H}_6, \mathcal{H}_1\otimes (\mathcal{H}_2 \otimes \mathcal{H}_3))$. Using these identifications, the diagram
\begin{align}
\label{4_eq:T_Hom_diagram2}
\xymatrixcolsep{4pc}\xymatrix{
M_{43}^6 \otimes M_{12}^4 \ar[r]^{{\bf T}} \ar[d] &
M_{23}^5 \otimes M_{15}^6 \ar[d] \\
Hom_{\mathcal{B}}(\mathcal{H}_6, (\mathcal{H}_1\otimes \mathcal{H}_2)\otimes \mathcal{H}_3)
\ar[r]^{(id\otimes id\otimes id)_*} &
Hom_{\mathcal{B}}(\mathcal{H}_6, \mathcal{H}_1\otimes (\mathcal{H}_2 \otimes \mathcal{H}_3)),
}
\end{align}
commutes, in view of \eqref{4_eq:T_commutative_diagram}. Suppose $\sum f_1 \otimes f_2$ and $\sum h_1 \otimes h_2$ are elements of  $Hom_{\mathcal{B}} (\mathcal{H}_6, \mathcal{H}_4 \otimes \mathcal{H}_3) \otimes Hom_{\mathcal{B}} (\mathcal{H}_4, \mathcal{H}_1 \otimes \mathcal{H}_2)$ and $Hom_{\mathcal{B}} (\mathcal{H}_5, \mathcal{H}_2 \otimes \mathcal{H}_3) \otimes Hom_{\mathcal{B}} (\mathcal{H}_6, \mathcal{H}_1 \otimes \mathcal{H}_5)$ respectively, where $f_1 \in Hom_{\mathcal{B}} (\mathcal{H}_6, \mathcal{H}_4 \otimes \mathcal{H}_3)$, $f_2 \in Hom_{\mathcal{B}} (\mathcal{H}_4, \mathcal{H}_1 \otimes \mathcal{H}_2)$, $h_1 \in Hom_{\mathcal{B}} (\mathcal{H}_5, \mathcal{H}_2 \otimes \mathcal{H}_3)$, and $h_2\in Hom_{\mathcal{B}} (\mathcal{H}_6, \mathcal{H}_1 \otimes \mathcal{H}_5)$. Combining the two diagrams \eqref{4_eq:T_Hom_diagram1} and \eqref{4_eq:T_Hom_diagram2} (and investigating the vertical arrows of the two diagrams), we can deduce that
\begin{align}
\label{4_eq:T_Hom_formulation}
{\bf T}^{Hom} \left( \sum f_1 \otimes f_2 \right)
= \sum h_1 \otimes h_2
\quad \Leftrightarrow \quad
\sum (f_2 \otimes id) f_1
= \sum (id \otimes h_1) h_2,
\end{align}
where $\sum (f_2 \otimes id) f_1$ and $\sum (id \otimes h_1) h_2$ are elements of $Hom_{\mathcal{B}}(\mathcal{H}_6, \mathcal{H}_1\otimes \mathcal{H}_2 \otimes \mathcal{H}_3)$. We now turn back to our original interest:

\begin{proposition}
\label{4_prop:ATA_ATA}
One has ${\bf A}_2 {\bf T}_{21} {\bf A}_1 = {\bf A}_1 {\bf T}_{12} {\bf A}_2$.
\end{proposition}

\begin{proof} Consider the diagram
\begin{align}
\label{4_eq:ATA_ATA_proof_Hom_diagram}
\xymatrix{
{\begin{array}{l}
Hom_{\mathcal{B}} (\mathcal{H}_n, \mathcal{H}_1 \otimes \mathcal{H}_2) \\
\otimes Hom_{\mathcal{B}} (\mathcal{H}_\ell, \mathcal{H}_n \otimes \mathcal{H}_3)
\end{array}}
\ar[r]^{{\bf T}^{Hom}_{21}}
\ar[d]^{(A^{Hom}_{n12})^{-1} A^{Hom}_{n3\ell}} &
{\begin{array}{l}
Hom_{\mathcal{B}} (\mathcal{H}_\ell, \mathcal{H}_1 \otimes \mathcal{H}_m) \\
\otimes Hom_{\mathcal{B}} (\mathcal{H}_m, \mathcal{H}_2 \otimes \mathcal{H}_3)
\end{array}}
\ar[d]^{(A^{Hom}_{\ell1m})^{-1} A^{Hom}_{23m}} \\
{\begin{array}{l}
Hom_{\mathcal{B}} (\mathcal{H}_2, \mathcal{H}_n \otimes \mathcal{H}_1) \\
\otimes Hom_{\mathcal{B}} (\mathcal{H}_n, \mathcal{H}_3 \otimes \mathcal{H}_\ell)
\end{array}}
\ar[r]^{{\bf T}^{Hom}_{12}}  &
{\begin{array}{l}
Hom_{\mathcal{B}} (\mathcal{H}_m, \mathcal{H}_\ell \otimes \mathcal{H}_1) \\
\otimes Hom_{\mathcal{B}} (\mathcal{H}_2, \mathcal{H}_3 \otimes \mathcal{H}_m)
\end{array}}
}
\end{align}
In view of diagrams \eqref{4_eq:def_A_operator_diagram} and \eqref{4_eq:T_Hom_diagram1}, it suffices to prove the commutativity of the diagram \eqref{4_eq:ATA_ATA_proof_Hom_diagram}.

\vs

An element of $Hom_{\mathcal{B}} (\mathcal{H}_n, \mathcal{H}_1 \otimes \mathcal{H}_2) \otimes Hom_{\mathcal{B}} (\mathcal{H}_\ell, \mathcal{H}_n \otimes \mathcal{H}_3)$ can be written in the form $\sum f_1 \otimes f_2$, where $f_1 \in Hom_{\mathcal{B}} (\mathcal{H}_n, \mathcal{H}_1 \otimes \mathcal{H}_2) \otimes$ and $f_2 \in Hom_{\mathcal{B}} (\mathcal{H}_\ell, \mathcal{H}_n \otimes \mathcal{H}_3)$. Let
\begin{align}
\label{4_eq:ATA_ATA_proof_T_relation}
{\bf T}^{Hom}_{21} \left(\sum f_1 \otimes f_2\right)
= \sum h_1 \otimes h_2,
\end{align}
where $h_1 \in Hom_{\mathcal{B}} (\mathcal{H}_\ell, \mathcal{H}_1 \otimes \mathcal{H}_m)$ and $h_2 \in Hom_{\mathcal{B}} (\mathcal{H}_m, \mathcal{H}_2 \otimes \mathcal{H}_3)$, and let
\begin{align}
\label{4_eq:ATA_ATA_proof_A_relations}
F_1 = (A^{Hom}_{n12})^{-1} f_1, \quad
F_2 = A^{Hom}_{n3\ell} f_2, \quad
H_1 = (A^{Hom}_{\ell1m})^{-1} h_1, \quad
H_2 = A^{Hom}_{23m} h_2.
\end{align}
Using the map $J_{ijk} : Hom_{\mathcal{B}}(\mathcal{H}_k, \mathcal{H}_i \otimes \mathcal{H}_j) \to Inv(\mathcal{H}_i \otimes \mathcal{H}_j \otimes \mathcal{H}_k')$ and another canonical map ${}'J_{ijk} : Hom_{\mathcal{B}}(\mathcal{H}_k, \mathcal{H}_i \otimes \mathcal{H}_j) \to Inv({}'\mathcal{H}_k \otimes \mathcal{H}_i \otimes \mathcal{H}_j)$ as in \eqref{4_eq:canonical_isomorphisms_Hom_and_primes}, let
\begin{align*}
& \widetilde{f}_1 = {}' J_{12n} f_1, \quad
\widetilde{f}_2 = J_{n3\ell} f_2, \quad
\widetilde{h}_1 = {}' J_{1m\ell} h_1, \quad
\widetilde{h}_2 = J_{23m} h_2, \\
& \widetilde{F}_1 = J_{n12} F_1, \quad
\widetilde{F}_2 = {}' J_{3\ell n} F_2, \quad
\widetilde{H}_1 = J_{\ell1m} H_1, \quad
\widetilde{H}_2 = {}' J_{3m2} H_2.
\end{align*}
What we should prove is that \eqref{4_eq:ATA_ATA_proof_A_relations} and \eqref{4_eq:ATA_ATA_proof_T_relation} imply
\begin{align}
\label{4_eq:ATA_ATA_proof_T_relation_target}
{\bf T}^{Hom}_{12} \left(\sum F_1 \otimes F_2\right)
= \sum H_1 \otimes H_2.
\end{align}

\vs

So, assume \eqref{4_eq:ATA_ATA_proof_A_relations} and \eqref{4_eq:ATA_ATA_proof_T_relation}. In view of \eqref{4_eq:A_Hom_definition}, the four equations in \eqref{4_eq:ATA_ATA_proof_A_relations} mean
\begin{align}
\label{4_eq:ATA_ATA_proof_tilde_relation}
\widetilde{f}_1 = ({}'C_n)(D_2') \widetilde{F}_1, \quad
\widetilde{f}_2 = ({}'D_n)(C_\ell') \widetilde{F}_2, \quad
\widetilde{h}_1 = ({}'C_\ell)(D_m') \widetilde{H}_1, \quad
\widetilde{h}_2 = ({}'D_2)(C_m') \widetilde{H}_2.
\end{align}
In view of \eqref{4_eq:T_Hom_formulation}, the equation \eqref{4_eq:ATA_ATA_proof_T_relation} means $\sum (f_1\otimes id) f_2 = \sum (id \otimes h_2) h_1$, hence
$$
\int \widetilde{f}_2(x_n,x_3,x_\ell) \phi(x_\ell) \widetilde{f}_1(x_n,x_1,x_2) dx_\ell dx_n
= \int \widetilde{h}_2(x_2,x_3,x_m) \phi(x_\ell) \widetilde{h}_1(x_\ell,x_1,x_m) dx_\ell dx_m,
$$
for any $\phi(x_\ell) \in \mathcal{H}_\ell \equiv L^2(\mathbb{R}, x_\ell)$. Therefore, as distributions,
\begin{align*}
\int \widetilde{f}_2(x_n,x_3,x_\ell)\widetilde{f}_1(x_n,x_1,x_2) dx_n
& = \int \widetilde{h}_2(x_2,x_3,x_m) \widetilde{h}_1(x_\ell,x_1,x_m) dx_m \\
\int (({}'D_n)(C_\ell') \widetilde{F}_2)(({}'C_n)(D_2') \widetilde{F}_1) dx_n
& = \int (({}'D_2)(C_m') \widetilde{H}_2) (({}'C_\ell)(D_m') \widetilde{H}_1) dx_m \quad (\mbox{by }\eqref{4_eq:ATA_ATA_proof_tilde_relation}) \\
\int (e^{-4\pi a p_n} (C_\ell') \widetilde{F}_2)((D_2') \widetilde{F}_1) dx_n
& = \int (({}'D_2) e^{-4\pi a p_m} \widetilde{H}_2) (({}'C_\ell) \widetilde{H}_1) dx_m \quad (\mbox{by }\eqref{4_eq:dual_intertwiners_transpose}) \\
\int (e^{-4\pi a p_n} ({}'D_\ell) (C_\ell') \widetilde{F}_2) \widetilde{F}_1 dx_n
& = \int ((C_2') ({}'D_2) e^{-4\pi a p_m} \widetilde{H}_2)  \widetilde{H}_1 dx_m \\
& \qquad\qquad (\mbox{apply $({}'D_\ell)(C_2')$ to both sides, and use \eqref{4_eq:dual_intertwiners_inverses}}) \\
\int (e^{-4\pi a (p_n+p_\ell)} \widetilde{F}_2) \widetilde{F}_1 dx_n
& = \int (e^{-4\pi a (p_m+p_2)} \widetilde{H}_2)  \widetilde{H}_1 dx_m \quad (\mbox{by }\eqref{4_eq:dual_intertwiners_doubles}) \\
\int (e^{-4\pi a (p_n+p_3+p_\ell)} \widetilde{F}_2) \widetilde{F}_1 dx_n
& = \int (e^{-4\pi a (p_2+p_3+p_m)} \widetilde{H}_2)  \widetilde{H}_1 dx_m \quad (\mbox{by } \mbox{apply $e^{-4\pi a p_3}$}).
\end{align*}
Since $\widetilde{F}_2 \in Inv({}' \mathcal{H}_n \otimes \mathcal{H}_3 \otimes \mathcal{H}_\ell)$ and $\widetilde{H}_2 \in Inv({}' \mathcal{H}_2 \otimes \mathcal{H}_3 \otimes \mathcal{H}_m)$, invariance property under the action of $X$ and $\widetilde{X}$ yields $e^{2\pi b^{\pm 1} (p_n+p_3+p_\ell)} \widetilde{F}_2  = \widetilde{F}_2$ and $e^{2\pi b^{\pm 1} (p_2+p_3+p_m)} \widetilde{H}_2 = \widetilde{H}_2$, yielding $e^{-4\pi a (p_n+p_3+p_\ell)} \widetilde{F}_2  = \widetilde{F}_2$ and $e^{-4\pi a (p_2+p_3+p_m)} \widetilde{H}_2 = \widetilde{H}_2$. Thus, we get
\begin{align*}
\int \widetilde{F}_2(x_n,x_3,x_\ell) \widetilde{F}_1(x_n,x_1,x_2) dx_n
= \int \widetilde{H}_2(x_2,x_3,x_m)  \widetilde{H}_1(x_\ell,x_1,x_m) dx_m
\end{align*}
as distributions, which is equivalent to $\sum(F_2 \otimes id) F_1 = \sum (id \otimes H_1) H_2$, which is again equivalent to \eqref{4_eq:ATA_ATA_proof_T_relation_target} in view of \eqref{4_eq:T_Hom_formulation}. Thus, the diagram \eqref{4_eq:ATA_ATA_proof_Hom_diagram} commutes, as desired.
\end{proof}

\vs

Lastly, we can check if the diagram
\begin{align}
\label{4_eq:diagram_TAT_AAP}
\xymatrix{
M_{12}^n \otimes M_{n3}^\ell
\ar[r]^{{\bf T}_{21}} \ar[d]_{{\bf A}_1 {\bf A}_2 P_{(12)}} &
M_{1m}^\ell \otimes M_{23}^m \ar[d]^{{\bf A}_1} \\
M_{3\ell}^n \otimes M_{2n}^1 &
M_{m\ell}^1 \otimes M_{23}^m \ar[l]^{{\bf T}_{12}}
}
\end{align}
commutes,  where $P_{(12)}$ is permutation of the two factors. 
Geometrically, this is encoded as in Figure \ref{4_fig:TAT_AAP}.
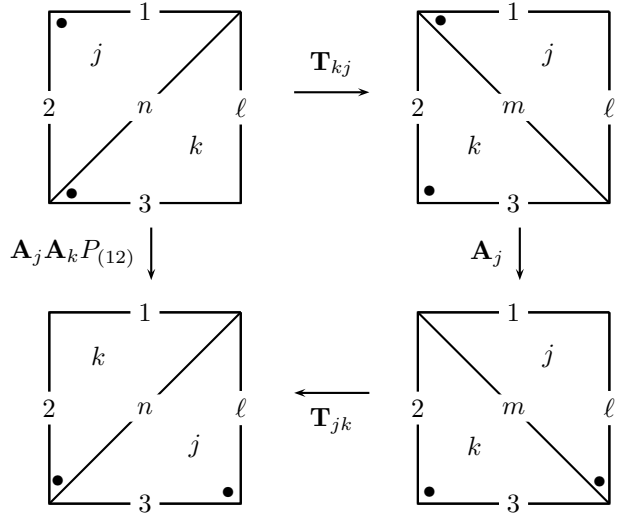
\begin{figure}[htbp!]
\centering
\begin{pspicture}[showgrid=false](0,-3.5)(8.4,3.1)
\rput[bl](0;0){
\PstSquare[unit=1.8,PolyName=P]
\pcline(P1)(P2)\ncput*{1}
\pcline(P2)(P3)\ncput*{2}
\pcline(P3)(P4)\ncput*{3}
\pcline(P4)(P1)\ncput*{$\ell$}
\pcline(P1)(P3)\ncput*{$n$}
\rput[l]{-45}(P2){\hspace{1,4mm}$\bullet$}
\rput[l]{22}(P3){\hspace{2,4mm}$\bullet$}
}
\rput[l](2.5,1.3){$k$}
\rput[l](1.2,2.5){$j$}
\rput[bl](0,-4){
\PstSquare[unit=1.8,PolyName=P]
\pcline(P1)(P2)\ncput*{1}
\pcline(P2)(P3)\ncput*{2}
\pcline(P3)(P4)\ncput*{3}
\pcline(P4)(P1)\ncput*{$\ell$}
\pcline(P1)(P3)\ncput*{$n$}
\rput[l]{68}(P3){\hspace{2,4mm}$\bullet$}
\rput[l]{135}(P4){\hspace{1,4mm}$\bullet$}
}
\rput[l](2.5,-2.7){$j$}
\rput[l](1.2,-1.5){$k$}
\rput[bl](4.9,0){
\PstSquare[unit=1.8,PolyName=P]
\pcline(P1)(P2)\ncput*{1}
\pcline(P2)(P3)\ncput*{2}
\pcline(P3)(P4)\ncput*{3}
\pcline(P4)(P1)\ncput*{$\ell$}
\pcline(P2)(P4)\ncput*{$m$}
\rput[l]{-22}(P2){\hspace{2,4mm}$\bullet$}
\rput[l]{45}(P3){\hspace{1,4mm}$\bullet$}
}
\rput[l](6.2,1.3){$k$}
\rput[l](7.2,2.5){$j$}
\rput[bl](4.9,-4){
\PstSquare[unit=1.8,PolyName=P]
\pcline(P1)(P2)\ncput*{1}
\pcline(P2)(P3)\ncput*{2}
\pcline(P3)(P4)\ncput*{3}
\pcline(P4)(P1)\ncput*{$\ell$}
\pcline(P2)(P4)\ncput*{$m$}
\rput[l]{112}(P4){\hspace{2,4mm}$\bullet$}
\rput[l]{45}(P3){\hspace{1,4mm}$\bullet$}
}
\rput[l](6.2,-2.7){$k$}
\rput[l](7.2,-1.5){$j$}
\rput[l](3.9,2){\pcline{->}(0,0)(1;0)\Aput{${\bf T}_{kj}$}}
\rput[l](2,0.2){\pcline{->}(0,0)(0,-0.7)\Bput{${\bf A}_j {\bf A}_k P_{(12)}$}}
\rput[l](3.9,-2){\pcline{->}(1;0)(0,0)\Aput{${\bf T}_{jk}$}}
\rput[r](6.9,0.2){\pcline{->}(0,0)(0,-0.7)\Bput{${\bf A}_j$}}
\end{pspicture}
\caption{Geometric realization of the diagram \eqref{4_eq:diagram_TAT_AAP}}
\label{4_fig:TAT_AAP}
\end{figure}

\vs

\begin{proposition}\label{4_prop:TAT_AAP}
One has $
{\bf T}_{12} {\bf A}_1 {\bf T}_{21}
= {\bf A}_1 {\bf A}_2 P_{(12)}.
$
\end{proposition}

\begin{proof} Consider the diagram
\begin{align}
\label{4_eq:TAT_AAP_proof_Hom_diagram}
\xymatrix{
{\begin{array}{l}
Hom_{\mathcal{B}} (\mathcal{H}_n, \mathcal{H}_1 \otimes \mathcal{H}_2) \\
\otimes Hom_{\mathcal{B}} (\mathcal{H}_\ell, \mathcal{H}_n \otimes \mathcal{H}_3)
\end{array}}
\ar[r]^{{\bf T}^{Hom}_{21}}
\ar[d]_{(A^{Hom}_{n3\ell} \otimes A^{Hom}_{12n})P_{(12)} } &
{\begin{array}{l}
Hom_{\mathcal{B}} (\mathcal{H}_\ell, \mathcal{H}_1 \otimes \mathcal{H}_m) \\
\otimes Hom_{\mathcal{B}} (\mathcal{H}_m, \mathcal{H}_2 \otimes \mathcal{H}_3)
\end{array}}
\ar[d]^{A^{Hom}_{1m\ell} \otimes id} \\
{\begin{array}{l}
Hom_{\mathcal{B}} (\mathcal{H}_n, \mathcal{H}_3 \otimes \mathcal{H}_\ell) \\
\otimes Hom_{\mathcal{B}} (\mathcal{H}_1, \mathcal{H}_2 \otimes \mathcal{H}_n)
\end{array}}
 &
{\begin{array}{l}
Hom_{\mathcal{B}} (\mathcal{H}_1, \mathcal{H}_m \otimes \mathcal{H}_\ell) \\
\otimes Hom_{\mathcal{B}} (\mathcal{H}_m, \mathcal{H}_2 \otimes \mathcal{H}_3)
\end{array}}
\ar[l]_{{\bf T}^{Hom}_{12}}
}
\end{align}
In view of diagrams \eqref{4_eq:def_A_operator_diagram} and \eqref{4_eq:T_Hom_diagram1}, it suffices to prove the commutativity of the diagram \eqref{4_eq:TAT_AAP_proof_Hom_diagram}.

\vs

As we did in the proof of Proposition \ref{4_prop:ATA_ATA}, let $\sum f_1 \otimes f_2 \in Hom_{\mathcal{B}} (\mathcal{H}_n, \mathcal{H}_1 \otimes \mathcal{H}_2) \otimes Hom_{\mathcal{B}} (\mathcal{H}_\ell, \mathcal{H}_n \otimes \mathcal{H}_3)$, where $f_1 \in Hom_{\mathcal{B}} (\mathcal{H}_n, \mathcal{H}_1 \otimes \mathcal{H}_2) \otimes$ and $f_2 \in Hom_{\mathcal{B}} (\mathcal{H}_\ell, \mathcal{H}_n \otimes \mathcal{H}_3)$. Denote
\begin{align}
\label{4_eq:TAT_AAP_proof_T_relation}
{\bf T}^{Hom}_{21} \left(\sum f_1 \otimes f_2\right)
= \sum h_1 \otimes h_2,
\end{align}
where $h_1 \in Hom_{\mathcal{B}} (\mathcal{H}_\ell, \mathcal{H}_1 \otimes \mathcal{H}_m)$ and $h_2 \in Hom_{\mathcal{B}} (\mathcal{H}_m, \mathcal{H}_2 \otimes \mathcal{H}_3)$, and let
\begin{align}
\label{4_eq:TAT_AAP_proof_A_relations}
F_1 = A^{Hom}_{12n} f_1, \quad
F_2 = A^{Hom}_{n3\ell} f_2, \quad
H_1 = A^{Hom}_{1m\ell} h_1.
\end{align}
We use the maps $J_{ijk} : Hom_{\mathcal{B}}(\mathcal{H}_k, \mathcal{H}_i \otimes \mathcal{H}_j) \to Inv(\mathcal{H}_i \otimes \mathcal{H}_j \otimes \mathcal{H}_k')$ and ${}'J_{ijk} : Hom_{\mathcal{B}}(\mathcal{H}_k, \mathcal{H}_i \otimes \mathcal{H}_j) \to Inv({}'\mathcal{H}_k \otimes \mathcal{H}_i \otimes \mathcal{H}_j)$, similarly as in the proof of Proposition \ref{4_prop:ATA_ATA}, to define the following (be careful that a same symbol may have a different definition than as in the proof of Proposition \ref{4_prop:ATA_ATA}):
\begin{align*}
& \widetilde{f}_1 = J_{12n} f_1, \quad
\widetilde{f}_2 = J_{n3\ell} f_2, \quad
\widetilde{h}_1 = J_{1m\ell} h_1, \quad
\widetilde{h}_2 = J_{23m} h_2, \\
& \widetilde{F}_1 = {}' J_{2n1} F_1, \quad
\widetilde{F}_2 = {}' J_{3\ell n} F_2, \quad
\widetilde{H}_1 = {}' J_{m\ell1} H_1.
\end{align*}
What we should prove is that \eqref{4_eq:TAT_AAP_proof_A_relations} and \eqref{4_eq:TAT_AAP_proof_T_relation} imply
\begin{align}
\label{4_eq:TAT_AAP_proof_T_relation_target}
{\bf T}^{Hom}_{12} \left(\sum H_1 \otimes h_2\right)
= \sum F_2 \otimes F_1.
\end{align}

\vs

So, assume \eqref{4_eq:TAT_AAP_proof_A_relations} and \eqref{4_eq:TAT_AAP_proof_T_relation}. In view of \eqref{4_eq:A_Hom_definition}, the three equations in \eqref{4_eq:TAT_AAP_proof_A_relations} mean
\begin{align}
\label{4_eq:TAT_AAP_proof_tilde_relation}
\widetilde{f}_1 = ({}'D_1)(C_n') \widetilde{F}_1, \quad
\widetilde{f}_2 = ({}'D_n)(C_\ell') \widetilde{F}_2, \quad
\widetilde{h}_1 = ({}'D_1)(C_\ell') \widetilde{H}_1.
\end{align}
In view of \eqref{4_eq:T_Hom_formulation}, the equation \eqref{4_eq:TAT_AAP_proof_T_relation} means $\sum (f_1\otimes id) f_2 = \sum (id \otimes h_2) h_1$, hence
$$
\int \widetilde{f}_1(x_1,x_2,x_n) \widetilde{f}_2(x_n,x_3,x_\ell) \phi(x_\ell) dx_\ell dx_n
= \int \widetilde{h}_2(x_2,x_3,x_m) \widetilde{h}_1(x_1,x_m,x_\ell) \phi(x_\ell) dx_\ell dx_m,
$$
for any $\phi(x_\ell) \in \mathcal{H}_\ell \equiv L^2(\mathbb{R}, x_\ell)$. Therefore, as distributions,
\begin{align*}
\int \widetilde{f}_1(x_1,x_2,x_n) \widetilde{f}_2(x_n,x_3,x_\ell)  dx_n
& = \int \widetilde{h}_2(x_2,x_3,x_m) \widetilde{h}_1(x_1,x_m,x_\ell) dx_m \\
\int (({}'D_1)(C_n') \widetilde{F}_1) (({}'D_n)(C_\ell') \widetilde{F}_2) dx_n
& = 
\int \widetilde{h}_2 (({}'D_1)(C_\ell') \widetilde{H}_1) dx_m \quad (\mbox{by }\eqref{4_eq:TAT_AAP_proof_tilde_relation}) \\
\int (({}'D_1) \widetilde{F}_1) ((C_\ell') \widetilde{F}_2) dx_n
& = 
\int \widetilde{h}_2 (({}'D_1)(C_\ell') \widetilde{H}_1) dx_m \quad (\mbox{by }\eqref{4_eq:dual_intertwiners_transpose_triv}) \\
\int \widetilde{F}_1 \widetilde{F}_2 dx_n
& = 
\int \widetilde{h}_2 \widetilde{H}_1 dx_m. \quad (\mbox{apply $({}'U_1)(D_\ell')$ and use \eqref{4_eq:dual_intertwiners_inverses}})
\end{align*}
Thus we obtained
$$
\int \widetilde{F}_1 (x_1,x_2,x_n)
\widetilde{F}_2 (x_n,x_3,n_\ell) dx_n
= 
\int \widetilde{h}_2 (x_2,x_3,x_m) 
\widetilde{H}_1 (x_1,x_m,x_\ell) dx_m
$$
as distributions, which is equivalent to $\sum (id \otimes F_2)F_1 = \sum (h_2 \otimes id)H_1$, which is again equivalent to \eqref{4_eq:TAT_AAP_proof_T_relation_target} in view of \eqref{4_eq:T_Hom_formulation}. Thus, the diagram \eqref{4_eq:TAT_AAP_proof_Hom_diagram} commutes, as desired.
\end{proof}

The summary of Propositions \ref{4_prop:pentagon}, \ref{4_prop:A_orderthree}, \ref{4_prop:ATA_ATA}, and \ref{4_prop:TAT_AAP} is:

\begin{theorem}
\label{4_thm:summary_T_A}
The operators ${\bf T}$, ${\bf A}$ satisfy
\begin{align}
\label{4_eq:thm_summary_T_A_eq1}
{\bf A}^3 = id, ~
{\bf T}_{23} {\bf T}_{12} = {\bf T}_{12} {\bf T}_{13} {\bf T}_{23}, ~
{\bf A}_2 {\bf T}_{21} {\bf A}_1 = {\bf A}_1 {\bf T}_{12} {\bf A}_2, ~
{\bf T}_{12} {\bf A}_1 {\bf T}_{21} = {\bf A}_1 {\bf A}_2 P_{(12)}.
\end{align}
\end{theorem}

\subsection{Family of operators ${\bf A}^{(m)}$, $m\in \mathbb{R}$}

For any real number $m$, put
\begin{align}
\label{4_eq:A_m_definition}
{\bf A}^{(m)} := \zeta^{m^2-1} {\bf A} e^{2\pi a(m-1)p},
\end{align}
where $\zeta = e^{-\pi i a^2/3} = e^{-2i\chi} e^{-\pi i/6}$. In particular, ${\bf A} = {\bf A}^{(1)}$. Then we have an interesting observation:

\begin{proposition}
\label{4_prop:A_m_relations}
The operators ${\bf T}$, ${\bf A}^{(m)}$ satisfy
\begin{align}
\label{4_eq:prop_A_m_relations_eq1}
\begin{array}{ll}
({\bf A}^{(m)})^3 = id, &
{\bf T}_{23} {\bf T}_{12} = {\bf T}_{12} {\bf T}_{13} {\bf T}_{23}, \\
{\bf A}_2^{(m)} {\bf T}_{21} {\bf A}_1^{(m)} = {\bf A}_1^{(m)} {\bf T}_{12} {\bf A}_2^{(m)}, &
{\bf T}_{12} {\bf A}_1^{(m)} {\bf T}_{21} = \zeta^{1-m^2} {\bf A}_1^{(m)} {\bf A}_2^{(m)} P_{(12)}.
\end{array}
\end{align}
\end{proposition}

\begin{proof}
Let's view $p_j,x_j,{\bf A}_j, {\bf A}_j^{(m)}, {\bf T}_{12}$ and ${\bf T}_{21}$ (for $j=1,2$) as acting on $L^2(\mathbb{R}, dx_1) \otimes L^2(\mathbb{R}, dx_2)$. For the first relation, we view $p,x,{\bf A}$ as acting on $L^2(\mathbb{R}, dx)$. In this proof, the result \eqref{4_eq:A_conjugation_action} (from Proposition \ref{4_prop:A_conjugation_action}) is constantly used:
\begin{align*}
{\bf A} x {\bf A}^{-1} = x+3p-ia, \quad
{\bf A} p {\bf A}^{-1} = -x-2p.
\end{align*}
Observe that
\begin{align*}
({\bf A}^{(m)})^3
& = \zeta^{3(m^2-1)} {\bf A} e^{2\pi a (m-1)p} {\bf A} e^{2\pi a (m-1)p} {\bf A} e^{2\pi a (m-1)p} \\
& = \zeta^{3(m^2-1)} {\bf A} e^{2\pi a (m-1)p} {\bf A} e^{2\pi a (m-1)p} e^{2\pi a (m-1)(-x-2p)} {\bf A} \quad (\mbox{by }\eqref{4_eq:A_conjugation_action}) \\
& = \zeta^{3(m^2-1)} e^{\pi i (m-1)^2a^2} {\bf A} e^{2\pi a (m-1)p} {\bf A} e^{2\pi a (m-1)(-x-p)} {\bf A} \quad(\mbox{by }\mbox{BCH}) \\
& = \zeta^{3(m^2-1)} e^{\pi i (m-1)^2a^2} {\bf A} e^{2\pi a (m-1)p} e^{2\pi a (m-1)(-p+ia)} {\bf A} {\bf A} \quad (\mbox{by }\eqref{4_eq:A_conjugation_action}) \\
& = \zeta^{3(m^2-1)} e^{\pi i ((m-1)^2+2 (m-1)) a^2} {\bf A} {\bf A} {\bf A}
= \zeta^{3(m^2-1)} e^{\pi i (m^2-1) a^2} {\bf A}^3
= id,
\end{align*}
because $\zeta^{3(m^2-1)} = e^{-\pi i (m^2-1)a^2}$ and ${\bf A}^3  = id$ (by Proposition \ref{4_prop:A_orderthree}); by BCH we mean the Baker-Campbell-Hausdorff formula.

\vs

In this proof, we also use the conjugation action of ${\bf T}_{12}$ on the heisenberg generators $p_j,x_j$, as studied in Proposition \ref{5_prop:T_action_on_heisenberg}; among them, we use
\begin{align}
\label{4_eq:T_conjugation_action_1}
{\bf T}_{12} e^{2\pi a (m-1) (-p_1)} e^{2\pi a (m-1) (-x_2-2p_2)} {\bf T}_{12}^{-1} & = e^{2\pi a (m-1) (-p_1)} e^{2\pi a (m-1)(-x_2-2p_2)}, \\
\label{4_eq:T_conjugation_action_2}
{\bf T}_{12} e^{2\pi a (m-1) (-p_1-2p_1)} {\bf T}_{12}^{-1} & = e^{2\pi a (m-1) (-p_1-2p_1)} e^{2\pi a (m-1)(-x_2-2p_2)}.
\end{align}
From \eqref{4_eq:A_m_definition} and \eqref{4_eq:A_conjugation_action}, we have ${\bf A}^{(m)} = \zeta^{m^2-1} e^{2\pi a (m-1)(-x-2p)} {\bf A}$. Observe
\begin{align*}
{\bf A}_2^{(m)} {\bf T}_{21} {\bf A}_1^{(m)} e^{-2\pi a (m-1) p_1}
& = \zeta^{m^2-1} {\bf A}_2^{(m)} {\bf T}_{21} {\bf A}_1
= \zeta^{2(m^2-1)} e^{2\pi a (m-1)(-x_2-2p_2)} {\bf A}_2 {\bf T}_{21} {\bf A}_1, \\
{\bf A}_1^{(m)} {\bf T}_{12} {\bf A}_2^{(m)} e^{-2\pi a (m-1) p_1}
& = \zeta^{m^2-1} {\bf A}_1^{(m)} {\bf T}_{12} e^{-2\pi a (m-1) p_1} e^{2\pi a (m-1) (-x_2-2p_2)} {\bf A}_2 \\
& = \zeta^{m^2-1} {\bf A}_1^{(m)} e^{-2\pi a (m-1) p_1} e^{2\pi a (m-1) (-x_2-2p_2)} {\bf T}_{12} {\bf A}_2 \quad (\mbox{by } \eqref{4_eq:T_conjugation_action_1}) \\
& = \zeta^{2(m^2-1)} {\bf A}_1 e^{2\pi a (m-1)(-x_2-2p_2)} {\bf T}_{12} {\bf A}_2 \\
& = \zeta^{2(m^2-1)} e^{2\pi a (m-1)(-x_2-2p_2)} {\bf A}_1 {\bf T}_{12} {\bf A}_2.
\end{align*}
Using ${\bf A}_2 {\bf T}_{21} {\bf A}_1 = {\bf A}_1 {\bf T}_{12} {\bf A}_2$, we can now deduce ${\bf A}_2^{(m)} {\bf T}_{21} {\bf A}_1^{(m)} = {\bf A}_1^{(m)} {\bf T}_{12} {\bf A}_2^{(m)}$. Note also that
\begin{align*}
{\bf T}_{12} {\bf A}_1^{(m)} {\bf T}_{21}
& = \zeta^{m^2-1} {\bf T}_{12} e^{2\pi a (m-1) (-x_1-2p_1)} {\bf A}_1 {\bf T}_{21} \\
& = \zeta^{m^2-1} e^{2\pi a (m-1) (-x_1-2p_1)} e^{2\pi a (m-1) (-x_2-2p_2)} {\bf T}_{12} {\bf A}_1 {\bf T}_{21}, \quad (\mbox{by } \eqref{4_eq:T_conjugation_action_1}) \\
{\bf A}_1^{(m)} {\bf A}_2^{(m)} P_{(12)}
& =\zeta^{2(m^2-1)} e^{2\pi a (m-1)(-x_1-2p_1)} {\bf A}_1
 e^{2\pi a (m-1)(-x_2-2p_2)} {\bf A}_2 P_{(12)} \\
& =\zeta^{2(m^2-1)} e^{2\pi a (m-1)(-x_1-2p_1)}  e^{2\pi a (m-1) (-x_2-2p_2)} {\bf A}_1 {\bf A}_2 P_{(12)}.
\end{align*}
Using ${\bf T}_{12} {\bf A}_1 {\bf T}_{21} = {\bf A}_1 {\bf A}_2 P_{(12)}$, we can now deduce ${\bf T}_{12} {\bf A}_1^{(m)} {\bf T}_{21} = \zeta^{1-m^2} {\bf A}_1^{(m)} {\bf A}_2^{(m)} P_{(12)}$.
\end{proof}

\begin{remark}
In view of \eqref{4_eq:prop_A_m_relations_eq1}, the operator ${\bf A}^{(-1)}$ is also special as our original operator ${\bf A}^{(1)}$,  because $m=-1$ makes $\zeta^{1-m^2}=1$ in the relation \eqref{4_eq:prop_A_m_relations_eq1}. It is an interesting question to find a conceptual way of constructing ${\bf A}^{(-1)}$.
\end{remark}

\section{Relation to the Kashaev representation}

\subsection{Kashaev's construction}

Kashaev defined the following (unitary) operators (\cite{Kash00}):
\begin{align}
\label{5_eq:Kashaev_operators_def}
\widetilde{\bf T}_{12} \equiv e^{2\pi i p_1 x_2} \Psi_b(x_1 + p_2 - x_2)^{-1}, \quad
\widetilde{\bf A} \equiv e^{-\pi i/3} e^{3\pi i x^2} e^{\pi i (p+x)^2},
\end{align}
where $\Psi_b$ is defined as $\Psi_b(z) = G(-z) e^{\frac{\pi i}{2} z^2 + i\chi}$, the operator $\widetilde{\bf T}_{12}$ acts on $L^2(\mathbb{R}^2, dx_1dx_2)$, and the operator $\widetilde{\bf A}$ on $L^2(\mathbb{R},dx)$. The operator $\widetilde{\bf A}$ can be realized as integral transformation (e.g. as in \cite{Kash99}):
\begin{align}
\label{5_eq:Kashaev_A_def}
\widetilde{\bf A} : L^2(\mathbb{R}, d\alpha) \to L^2(\mathbb{R}, d\beta), \quad
f(\alpha) \mapsto
\int_\mathbb{R} e^{2\pi i \alpha\beta} e^{\pi i \beta^2 - \pi i/12} f(\alpha)d\alpha.
\end{align}

\begin{theorem}\label{5_thm:Kashaev_operators} (Kashaev \cite{Kash99})
The operators $\widetilde{\bf T}, \widetilde{\bf A}$ satisfy
\begin{align}
\label{5_eq:Kashaev_rel}
\widetilde{\bf A}^3 = id, ~
\widetilde{\bf T}_{12} \widetilde{\bf T}_{13} \widetilde{\bf T}_{23} = \widetilde{\bf T}_{23} \widetilde{\bf T}_{12}, ~
\widetilde{\bf A}_1 \widetilde{\bf T}_{12} \widetilde{\bf A}_2 = \widetilde{\bf A}_2 \widetilde{\bf T}_{21} \widetilde{\bf A}_1, ~
\widetilde{\bf T}_{12} \widetilde{\bf A}_1 \widetilde{\bf T}_{21} = \zeta \widetilde{\bf A}_1 \widetilde{\bf A}_2 P_{(12)},
\end{align}
where $\zeta = e^{-\pi i a^2/3}$ is same as ours.
\end{theorem}

\begin{definition}\label{5_def:Kashaev_group}
For an index set $I$, the Kashaev group $G_I$ associated to $I$ is the group with generators $a_j, t_{jk}, p_{jk}$ ($j,k\in I$), with the relations for all $j,k,\ell \in I$
\begin{align*}
a_j^3 & = e, \\
a_j t_{jk} a_k & = a_k t_{kj} a_j, \\
t_{jk} a_j t_{kj} & = a_j a_k p_{jk}, \\
t_{k\ell} t_{jk} & = t_{jk} t_{j\ell} t_{k\ell},
\end{align*}
and $p_{j,k}$ satisfies the relations of an involution of $j,k \in I$ in the permutation group of the set $I$. We also define the central extension $\widehat{G}_I$ of $G_I$ with the corresponding generators $\widehat{a}_{j}$, $\widehat{t}_{jk}$, $\widehat{p}_{jk}$ and the central element $z$ satisfying the same relations as $G_I$ except $t_{jk} a_j t_{kj} = a_j a_k p_{jk}$, which is replaced by
\begin{align*}
\widehat{t}_{ij} \widehat{a}_i \widehat{t}_{ji}
= z \widehat{a}_i \widehat{a}_j \widehat{p}_{ij}.
\end{align*}
\end{definition}

\vs 
From \eqref{5_eq:Kashaev_rel} and Proposition \ref{4_prop:A_m_relations}, $({\bf T}, {\bf A}^{(m)})$ (for each $m\in\mathbb{R}$) is a representation of $\widehat{G}_I$ with $z$ represented by the identity times $\zeta^{1-m^2}$, while by Theorem \ref{5_thm:Kashaev_operators} $(\widetilde{\bf T}, \widetilde{\bf A})$ is a representation of $\widehat{G}_I$ with $z$ represented by the identity times $\zeta$.

\vs

We will now construct a family $(\widetilde{\bf T}, \widetilde{\bf A}^{(m)})$ of representations of the central extension $\widehat{G}_I$ of the Kashaev group, having Kashaev's pair of operators $(\widetilde{\bf T}, \widetilde{\bf A})$ as its member, and prove the (unitary) equivalence of the family with another one, namely $({\bf T}, {\bf A}^{(m)})$.

\subsection{Conjugation action on Heisenberg generators}

As a preliminary, we study how the operators $\widetilde{\bf T}$ and $\widetilde{\bf A}$ act on the Heisenberg generators $p_j,x_j$ by conjugation:

\begin{proposition}
\label{5_prop:Kashaev_operators_conjugation_action}
For any $\ell \in \mathbb{R}$, one has
\begin{align*}
& \widetilde{\bf T}_{12} e^{\ell (p_1+p_2)} \widetilde{\bf T}_{12}^{-1}
= e^{\ell p_2}, \quad
\widetilde{\bf T}_{12} e^{\ell x_1} \widetilde{\bf T}_{12}^{-1}
=
e^{\ell (x_1+x_2)}, \quad
\widetilde{\bf T}_{12} e^{\ell (p_1+x_2)} \widetilde{\bf T}_{12}^{-1}
= e^{\ell (p_1+x_2)}, \\
& \widetilde{\bf T}_{12} e^{2\pi b^{\pm 1} p_1} \widetilde{\bf T}_{12}^{-1}
= e^{2\pi b^{\pm 1} p_1} (1 + e^{\pi i b^{\pm 2}} e^{2\pi b^{\pm 1} (x_1 - p_1 + p_2)}), \quad
\widetilde{\bf A} x \widetilde{\bf A}^{-1} = p-x, \quad
\widetilde{\bf A} p \widetilde{\bf A}^{-1} = -x.
\end{align*}
\end{proposition}

\vs

\begin{proof}
The last two relations about $\widetilde{\bf A}$ is already mentioned by Kashaev (\cite{Kash00}). We use the following: for a function $f$ (that has a necessary analytic continuation) and for $\ell\in\mathbb{R}$, one has
\begin{align}
\label{5_eq:commutation_r_s}
e^{-\ell p} f(x) e^{\ell p} = f(x +\frac{i\ell}{2\pi}).
\end{align}

\vs

The first three relations will be obtained as corollary of Proposition \ref{5_prop:T_action_on_heisenberg} and Theorem \ref{5_thm:equivalence_representations}, hence we only prove the fourth one (though it's equally easy to get the first three here). Using \eqref{5_eq:commutation_r_s}, one can get
\begin{align*}
e^{2\pi b^{\pm 1} p_1} \Psi_b(x_1+p_2-x_2) e^{-2\pi b^{\pm 1} p_1}
& = \Psi_b(x_1+p_2-x_2 - ib^{\pm1}), \\
e^{2\pi i p_1 x_2} e^{ 2\pi b^{\pm 1} p_1} e^{-2\pi i p_1 x_2}
& = e^{2\pi b^{\pm 1} p_1}, \\
e^{2\pi i p_1 x_2} e^{2\pi b^{\pm 1}(x_1+p_2-x_2)} e^{-2\pi i p_1 x_2}
& = e^{2\pi b^{\pm 1} (x_1-p_1+p_2)}.
\end{align*}
Thus, using the above obtained formulas and functional equations of $\Psi_b$:
\begin{align*}
\Psi_b(w)^{-1} \Psi_b (w-ib^{\pm 1}) = (1 + e^{-\pi i b^{\pm 2}} e^{2\pi b^{\pm 1} w}),
\end{align*}
we obtain
\begin{align*}
\widetilde{\bf T}_{12} e^{2\pi b^{\pm 1} p_1} \widetilde{\bf T}_{12}^{-1}
& =
e^{2\pi i p_1 x_2} \Psi_b(x_1 + p_2 - x_2)^{-1} 
e^{2\pi b^{\pm 1} p_1}
\Psi_b(x_1 + p_2 - x_2) e^{-2\pi i p_1 x_2} \\
& = e^{2\pi i p_1 x_2} \Psi_b(x_1 + p_2 - x_2)^{-1} 
\Psi_b(x_1 + p_2 - x_2 - i b^{\pm 1}) e^{2\pi b^{\pm 1} p_1} e^{-2\pi i p_1 x_2} \\
& = e^{2\pi i p_1 x_2} (1 + e^{-\pi i b^{\pm 2}} e^{2\pi b (x_1 + p_2 - x_2)}) e^{2\pi b^{\pm 1} p_1} e^{-2\pi i p_1 x_2} \\
& = (1 + e^{-\pi i b^{\pm 2}} e^{2\pi b^{\pm 1} (x_1 - p_1 + p_2)}) e^{2\pi b^{\pm 1} p_1} = e^{2\pi b^{\pm 1} p_1} (1 + e^{\pi i b^{\pm 2}} e^{2\pi b^{\pm 1} (x_1 - p_1 + p_2)}).
\end{align*}
\end{proof}

\vs

Kashaev's operator $\widetilde{\bf T}$ serves as a quantum version of the mutation operator corresponding to a certain change of triangulation of a punctured surface (see \cite{Kash98} and \cite{GL}); Proposition \ref{5_prop:Kashaev_operators_conjugation_action} shows the action of $\widetilde{\bf T}$ on the generators of the quantized algebra. We will see in the following proposition that our ${\bf T}_{12}$ acts (by conjugation) on the same generators with $p$ replaced by $-p$ and $x$ by $-x-2p$ exactly in the same way as Kashaev's $\widetilde{\bf T}_{12}$. This suggests that our ${\bf T}$ can also be viewed as a quantum mutation operator.

\begin{proposition}\label{5_prop:T_action_on_heisenberg}
If ${\bf T}_{12}$ is understood as a mapping from $L^2(\mathbb{R}^2, d\beta d\alpha)$ to $L^2(\mathbb{R}^2, dAdB)$, then for $\ell \in \mathbb{R}$ one has
\begin{align*}
{\bf T}_{12} e^{\ell (-p_\beta-p_\alpha)} 
& = e^{\ell (-p_B)} {\bf T}_{12}, \\
{\bf T}_{12} e^{\ell (-\beta - 2p_\beta)}
& = e^{\ell (-A-2p_A)} e^{\ell (-B-2p_B)} {\bf T}_{12}, \\
{\bf T}_{12} e^{\ell (-p_\beta)} e^{\ell (-\alpha-2p_\alpha)}
& = e^{\ell (-p_A)} e^{\ell (-B-2p_B)} {\bf T}_{12}, \\
 {\bf T}_{12} e^{2\pi b^{\pm 1} (- p_\beta)}
& = e^{2\pi b^{\pm 1} (-p_A)} (1+e^{\pi i b^{\pm 2}} e^{2\pi b^{\pm 1} ((-A-2p_A)-(-p_A)+(-p_B))}) {\bf T}_{12}, \\
{\bf A}^{(0)} (-x-2p) ({\bf A}^{(0)})^{-1} & = (-p) - (-x-2p), \\
{\bf A}^{(0)} (-p) ({\bf A}^{(0)})^{-1} & = -(-x-2p).
\end{align*}
\end{proposition}

\begin{proof}
The last two relations about ${\bf A}^{(0)}$ come immediately from \eqref{4_eq:A_0_conjugation}, and the fourth relation about ${\bf T}_{12}$ can be obtained from Proposition \ref{5_prop:Kashaev_operators_conjugation_action} and
Theorem \ref{5_thm:equivalence_representations}. Recall that ${\bf T}_{12}$ is realized as an integral transformation  with distribution kernel $T(\beta,\alpha,A,B)$ (see \eqref{4_eq:T_formula} for its definition). For brevity, $T(\beta,\alpha,A,B)$ will be denoted by $T$.

\vs

Note that $\mathcal{G}(\alpha+\beta+A-2B) = \mathcal{G}(A + (\alpha - B) + (\beta-B))$, and that
\begin{align*}
\alpha A + \beta B - \alpha\beta - AB + \alpha^2/2 - B^2/2
= A(\alpha-B) + (\alpha - B)((\alpha-\beta) + (B - \beta))/2.
\end{align*}
Hence $e^{\ell (p_\beta+p_\alpha+p_B)} T = T$ (above shows $T$ depends only on $A, (\alpha-B), (\alpha-\beta), (B-\beta)$), so for all $\ell\in \mathbb{R}$ we have
\begin{align}
\label{5_eq:T_res1}
e^{\ell (p_\beta+p_\alpha)} T = e^{-\ell p_B} T.
\end{align}

\vs

Note that $\mathcal{G}(\alpha+\beta+A-2B) = \mathcal{G} ( \alpha + (\beta-B) + (A-B))$, and that
\begin{align*}
& \alpha A + \beta B - \alpha\beta - AB + \alpha^2/2 - B^2/2
= \alpha (A - \beta) + B(\beta - A) + \alpha^2/2 - B^2/2,
\end{align*}
hence it's easy to see that
$
e^{2\pi  m(p_\beta + p_A+p_B)}T
= e^{\pi m(\beta-A-B)} e^{\pi i m^2/2} T,
$ for any $m\in \mathbb{R}$; using BCH formula, we get 
\begin{align}
\label{5_eq:T_res2}
e^{\pi m (-\beta + 2p_\beta)} T
= e^{\pi m(-A-2p_A)} e^{\pi m(-B-2p_B)} T.
\end{align}

\vs

Note that $\mathcal{G}(\alpha+\beta+A-2B) = \mathcal{G}( \alpha + (\beta + A) - 2B)$, and that
\begin{align*}
& \alpha A + \beta B - \alpha\beta - AB + \alpha^2/2 - B^2/2
= (\alpha-B) A + \beta ( B-\alpha) + \alpha^2/2 - B^2/2,
\end{align*}
hence it's easy to see for any $m\in\mathbb{R}$ that
\begin{align}
\label{5_eq:intermediate1}
e^{2\pi m(p_A - p_\beta)} T
= e^{\pi (m(\alpha-B) - m(B-\alpha))} T
= e^{2\pi m (\alpha-B)} T,
\end{align}
thus $e^{\ell (p_A - p_\beta)} T = e^{\ell(\alpha-B)}T$, for any $\ell \in \mathbb{R}$. Since $e^{2\ell(p_\beta+p_\alpha+p_B)}T = T$ (by \eqref{5_eq:T_res1}), we can now deduce from \eqref{5_eq:intermediate1} that
$$
e^{\ell (+ p_\beta + 2p_\alpha + p_A +2p_B)} T
= e^{\ell(p_A-p_\beta)} e^{2\ell(p_\beta+p_\alpha+p_B)} T
= e^{\ell (\alpha-B)}T,
$$
which (together with BCH formula) yields
\begin{align}
\label{5_eq:T_res3}
e^{\ell p_\beta} e^{\ell(-\alpha+2p_\alpha)} T
= e^{\ell (-p_A)} e^{\ell (-B-2p_B)}T.
\end{align}
By taking transposes of \eqref{5_eq:T_res1}, \eqref{5_eq:T_res2}, and \eqref{5_eq:T_res3}, we get the first three lines of the assertion of this proposition. (The fourth relation can also be proved similarly, with the help of the functional relations of $G$ \eqref{2_eq:G_defining_relations}.)
\end{proof}

\subsection{Family of projective representations of the Kashaev group}

For any real number $m$, put
\begin{align}
\label{5_eq:A_tilde_m_definition}
\widetilde{{\bf A}}^{(m)} := \zeta^{m^2} \widetilde{{\bf A}} e^{-2\pi a mp}.
\end{align}
In particular, Kashaev's operator is $\widetilde{{\bf A}} = \widetilde{{\bf A}}^{(0)}$. Using similar argument as in Proposition \ref{4_prop:A_m_relations} (which now relies on Proposition \ref{5_prop:Kashaev_operators_conjugation_action}), we obtain the following result:
\begin{proposition}
\label{5_prop:A_tilde_m_relations}
The operators $\widetilde{{\bf T}}$, $\widetilde{{\bf A}}^{(m)}$ satisfy
\begin{align}
\label{5_eq:prop_A_tilde_m_relations_eq1}
\begin{array}{ll}
(\widetilde{{\bf A}}^{(m)})^3 = id, &
\widetilde{{\bf T}}_{23} \widetilde{{\bf T}}_{12} = \widetilde{{\bf T}}_{12} \widetilde{{\bf T}}_{13} \widetilde{{\bf T}}_{23}, \\
\widetilde{{\bf A}}_2^{(m)} \widetilde{{\bf T}}_{21} \widetilde{{\bf A}}_1^{(m)} = \widetilde{{\bf A}}_1^{(m)} \widetilde{{\bf T}}_{12} \widetilde{{\bf A}}_2^{(m)}, &
\widetilde{{\bf T}}_{12} \widetilde{{\bf A}}_1^{(m)} \widetilde{{\bf T}}_{21} = \zeta^{1-m^2} \widetilde{{\bf A}}_1^{(m)} \widetilde{{\bf A}}_2^{(m)} P_{(12)}.
\end{array}
\end{align}
$\qed$
\end{proposition}
Thus we have two families of representations of the central extension $\widehat{G}_I$ of the Kashaev group: $({\bf T}, {\bf A}^{(m)})$ and $(\widetilde{{\bf T}}, \widetilde{{\bf A}}^{(m)})$. In the next subsection we prove that these two families are equivalent.

\subsection{Equivalence of the two families $({\bf T}, {\bf A}^{(m)})$ and $(\widetilde{{\bf T}}, \widetilde{{\bf A}}^{(m)})$}

Let $U:L^2(\mathbb{R},dx) \to L^2(\mathbb{R},dy)$ and $U^{-1}:L^2(\mathbb{R},dy) \to L^2(\mathbb{R},dx)$ be the unitary transformations given by
\begin{align}
U: f(x) \mapsto 
\frac{1}{\sqrt{2}} \int_\mathbb{R} e^{-\pi i (x+y)^2/2} f(x) dx, \quad
U^{-1}: \varphi(y) \mapsto
\frac{1}{\sqrt{2}} \int_\mathbb{R} e^{\pi i (x+y)^2/2} \varphi(y) dy.
\end{align}
It is easy to show that these two are indeed inverses to each other, and that
\begin{align}
\label{5_eq:U_action_on_single_heisenberg}
U p U^{-1} = -p, \quad
U x U^{-1} = -x-2p,
\end{align}
if $U$ and $U^{-1}$ are understood as acting on $L^2(\mathbb{R}, dx)$. Using $U$, we can consider the unitary transformation ${\bf U}$ from $\otimes_{j\in I} \mathcal{H}_j$ (here $I$ is an index set, as in Definition \ref{5_def:Kashaev_group}) to itself 
such that
\begin{align}
\label{5_eq:U_action_on_heisenberg}
{\bf U} p_j {\bf U}^{-1} = -p_j, \quad
{\bf U} x_j {\bf U}^{-1} = -x_j-2p_j
\end{align}
for every $j\in I$. We are now ready to prove the following theorem:

\begin{theorem}
\label{5_thm:equivalence_representations}
For any real number $m$, we have ${\bf U}^{-1} {\bf T} {\bf U} = \widetilde{\bf T}$ and ${\bf U}^{-1} {\bf A}^{(m)} {\bf U} = \widetilde{\bf A}^{(m)}$, i.e. the representations $({\bf T}, {\bf A}^{(m)})$ and $(\widetilde{{\bf T}}, \widetilde{{\bf A}}^{(m)})$ of the central extension $\widehat{G}_I$ of the Kashaev group are equivalent via the unitary transformation ${\bf U}$.
\end{theorem}

\begin{proof}
It suffices to prove
\begin{align}
\label{5_eq:to_prove_equivalence}
{\bf U}^{-1} {\bf T} {\bf U} = \widetilde{\bf T},\quad
\mbox{and}\quad {\bf U}^{-1} {\bf A}^{(0)} {\bf U} = \widetilde{\bf A}^{(0)},
\end{align}
because from \eqref{4_eq:A_m_definition}, \eqref{5_eq:A_tilde_m_definition}, \eqref{5_eq:U_action_on_single_heisenberg}, and \eqref{5_eq:to_prove_equivalence} we can get
\begin{align*}
U^{-1} {\bf A}^{(m)} U
= U^{-1} (\zeta^{m^2} {\bf A}^{(0)} e^{2\pi a mp}) U
= \zeta^{m^2} \widetilde{{\bf A}}^{(0)} e^{-2\pi a mp}
= \widetilde{{\bf A}}^{(m)}
\end{align*}
for any real number $m$.

\vs

Recall \eqref{5_eq:Kashaev_A_def} for the definition of $\widetilde{\bf A}$.
We now compute
$$
\xymatrix{
U\widetilde{\bf A}U^{-1} : L^2(\mathbb{R},d\alpha)
\ar[r]^-{U^{-1}} &
L^2(\mathbb{R},dx) \ar[r]^-{\widetilde{\bf A}} &
L^2(\mathbb{R},dy) \ar[r]^-{U} & L^2(\mathbb{R},d\beta)
}
$$
as follows:
\begin{align*}
(U \widetilde{\bf A} U^{-1} f) (\beta)
& = \int \frac{1}{2} e^{-\pi i (y+\beta)^2/2} e^{2\pi i xy + \pi i y^2 - \pi i /12} e^{\pi i (x+\alpha)^2/2} f(\alpha) d\alpha dx dy \\
& = \frac{\sqrt{2} e^{\pi i/4}}{2} \int e^{\pi i (-\beta^2/2-3y^2/2 - \beta y - 2y\alpha) - \pi i /12} f(\alpha) d\alpha dy 
\quad (\mbox{by } \eqref{4_eq:finite_integral} \mbox{ for $x$}) \\
& = \frac{e^{-\pi i/12}}{\sqrt{3}} \int e^{\pi i (2\alpha^2/3 + 2\alpha\beta/3-\beta^2/3)} f(\alpha)d\alpha 
\quad (\mbox{by } \eqref{4_eq:finite_integral} \mbox{ for $y$})  \\
& = \zeta^{-1} ({\bf A} e^{-2\pi a p_\alpha} f)(\beta) = ({\bf A}^{(0)} f)(\beta). \quad (\mbox{by } \eqref{4_eq:A_formula_clean}, \eqref{4_eq:A_m_definition})
\end{align*}
This proves ${\bf U}^{-1} {\bf A}^{(0)} {\bf U} = \widetilde{{\bf A}} = \widetilde{{\bf A}}^{(0)}$.

\vs

From \eqref{5_eq:U_action_on_heisenberg}, we first note that
\begin{align}
\label{5_eq:transformed_Kashaev_T}
{\bf U} \widetilde{\bf T}_{12} {\bf U}^{-1}
= e^{2\pi i p_1(x_2+2p_2)} \Psi_b(-x_1-2p_1+x_2+p_2)^{-1}.
\end{align}
Our aim is to realize \eqref{5_eq:transformed_Kashaev_T} as an integral transformation and to compute its distribution kernel. To do this, we introduce some transformations: $U' : L^2(\mathbb{R}, dx) \to L^2(\mathbb{R}, dy)$, $({U'})^{-1} : L^2(\mathbb{R}, dy) \to L^2(\mathbb{R}, dx)$, given by
\begin{align*}
U' : f(x) \mapsto \int_\mathbb{R} e^{\pi i x^2 - 2\pi i xy} f(x)dx, \quad
({U'})^{-1} : \varphi(y) \mapsto \int_\mathbb{R} e^{- \pi i x^2 + 2\pi i xy} \varphi(y) dy.
\end{align*}
Then it's easy to show that indeed above defined $U'$ and $({U'})^{-1}$ are inverses to each other, and
\begin{align}
\label{5_eq:U_prime_action_on_single_heisenberg}
U' (p+x) (U')^{-1} = x, \quad
U' x (U')^{-1} = -p.
\end{align}
The idea is to put the identity operators $({U'})^{-1} U'$ and $UU^{-1}$ appropriately, and use \eqref{5_eq:U_action_on_single_heisenberg} and \eqref{5_eq:U_prime_action_on_single_heisenberg}. 
First, let's compute $I(x_1,x_2) := \Psi_b(-x_1-2p_1+x_2+p_2)^{-1} f (x_1,x_2)$, for $f\in W\otimes W$:
\begin{align*}
& I(x_1,x_2) = \Psi_b (-x_1-2p_1+x_2+p_2)^{-1} f (x_1,x_2) \\
& = \int_\mathbb{R} e^{-\pi i x_2^2 + 2\pi i x_2 y_2} \Psi_b (-x_1-2p_1 +y_2)^{-1}  (\int_\mathbb{R} e^{\pi i x_2'^2 - 2\pi i x_2' y_2}  f(x_1,x_2') dx_2')  dy_2 \\
& = \int_\mathbb{R} e^{-\pi i x_2^2 + 2\pi i x_2 y_2} \\
& \quad \cdot \left[ \int_\mathbb{R} \frac{e^{-\pi i (y_1+x_1)^2/2}}{\sqrt{2}} \Psi_b (y_1 +y_2)^{-1}  \left( \int_\mathbb{R} \frac{e^{\pi i (y_1+x_1')^2/2}}{\sqrt{2}}  (\int_\mathbb{R} e^{\pi i x_2'^2 - 2\pi i x_2' y_2}  f(x_1',x_2') dx_2') dx_1' \right) dy_1 \right]  dy_2.
\end{align*}
Recall $\Psi_b(z) = G(-z) e^{\frac{\pi i}{2} z^2 + i\chi}$; thus $\Psi_b(z)^{-1} = G(z) e^{-\frac{\pi i}{2}z^2-i\chi}$. So
\begin{align*}
I(x_1,x_2)
& = \frac{1}{2} e^{-i\chi} \int_{\mathbb{R}^4} G (y_1+y_2) \exp(*C_1) f(x_1',x_2') dy_1 dy_2 dx_1' dx_2', \\
(*C_1) / (\pi i)
& = - x_2^2 + 2 x_2 y_2
- (y_1+x_1)^2/2
- (y_1+y_2)^2/2
+ (y_1+x_1')^2/2
+ x_2'^2 - 2 x_2' y_2  \\
& =
- Y^2/2
+ y_1 (x_1'  - x_1- 2x_2 +2x_2') 
- 2 Y (x_2' - x_2)
+(- x_1^2/2 - x_2^2
+ {x_1'}^2/2 + x_2'^2)
\end{align*}
for $y_1+y_2=Y$ (change of variables: use $Y$ instead of $y_2$). Hence, integration w.r.t. $y_1$ yields the factor $2\delta(x_1'  - x_1- 2x_2 +2x_2')$.
Now, use Fourier transform formula \eqref{2_eq:G_Four_-_0}:
$$
\int_\mathbb{R} G(Y) e^{-\frac{\pi i}{2}Y^2} e^{-2\pi i Yw} dY
= e^{i\chi} G(w-ia) e^{\frac{\pi i}{2} w^2} e^{-\pi a w},
$$
for $w = x_2' - x_2$; integration w.r.t. $Y$ yields
\begin{align*}
I(x_1,x_2)
& = \int_{\mathbb{R}^2} \delta(x_1'  - x_1- 2x_2 +2x_2') G (x_2'-x_2-ia) \exp(*C_3) f(x_1',x_2') dx_1'dx_2'^{\cap x_2}, \\
(*C_3) / (\pi i)
& =
(- x_1^2/2 - x_2^2
+ {x_1'}^2/2 + x_2'^2)
+ ({x_2'}^2/2 + x_2^2/2 - x_2' x_2)
+ ia (x_2'-x_2).
\end{align*}
Now, integration w.r.t. $x_1'$ has the effect of replacing $x_1'$ by $x_1+2x_2-2x_2'$, because of the $\delta(x_1'-x_1-2x_2+2x_2')$ factor:
\begin{align}
\label{5_eq:proof_equivalence_H_toput}
I(x_1,x_2)
& = \int_\mathbb{R} G(x_2'-x_2-ia) \exp(*C_4) f(x_1+2x_2-2x_2', x_2') dx_2'^{\cap x_2}, \\
\nonumber
(*C_4) / (\pi i)
& =
3x_2^2/2 + 7{x_2'}^2/2
+ 2x_1 x_2 - 2x_1 x_2' - 5x_2 x_2'
+ ia (x_2'-x_2).
\end{align} 
By putting the identity operator $UU^{-1}$ appropriately and using \eqref{5_eq:U_action_on_single_heisenberg}, we get
\begin{align}
\nonumber
{\bf U} \widetilde{\bf T}_{12} {\bf U}^{-1} f(x_1,x_2)
& = e^{2\pi i p_1 (x_2+2p_2)} I(x_1,x_2) \\
\nonumber
& = \int_\mathbb{R} \frac{e^{-\pi i (w_2+x_2)^2/2}}{\sqrt{2}} \left[ e^{2\pi i p_1(-w_2)}  \int_\mathbb{R} \frac{e^{\pi i (w_2+z_2)^2/2}}{\sqrt{2}}  I(x_1,z_2) dz_2 \right] dw_2 \\
\label{5_eq:proof_equivalence_H_puthere}
& = \frac{1}{2} \int_\mathbb{R} e^{-\pi i (w_2+x_2)^2/2} \left[  \int_\mathbb{R} e^{\pi i (w_2+z_2)^2/2}  I(x_1-w_2,z_2) dz_2 \right] dw_2.
\end{align}
Now, by putting \eqref{5_eq:proof_equivalence_H_toput} into \eqref{5_eq:proof_equivalence_H_puthere}, we get
\begin{align*}
{\bf U} \widetilde{\bf T}_{12} {\bf U}^{-1} & f(x_1,x_2)
= \frac{1}{2}
\int_{\mathbb{R}^3} G(x_2'-z_2-ia) f(x_1-w_2+2z_2-2x_2', x_2') \exp(*C_5) dx_2'^{\cap z_2} dz_2 dw_2, \\
(*C_5) / (\pi i)
& = + 2z_2^2 + 7{x_2'}^2/2
+ 2x_1z_2  - 2(x_1-w_2) x_2' - 5z_2 x_2'
+ ia (x_2'-z_2)
 - x_2^2/2 - w_2x_2
 -w_2 z_2.
\end{align*}
Do changes of variables $z_2 \mapsto z = x_2' - z_2$ and $w_2 \mapsto W = x_1-w_2-2z$:
\begin{align*}
{\bf U} \widetilde{\bf T}_{12} {\bf U}^{-1} & f(x_1,x_2)
= \frac{1}{2} \int_{\mathbb{R}^3} G (z-ia) f(W, x_2') \exp(*C_6) dx_2' dW dz^{\cap 0}, \\
(*C_6) / (\pi i)
& =  (W x_2 + x_2' x_1
- W x_2'  - x_1 x_2
+ {x_2'}^2/2 - x_2^2/2)
 + ia z
 + z(  - W - x_2' - x_1 + 2 x_2).
 \end{align*}
It is easy to check by inspection that ${\bf U} \widetilde{\bf T}_{12} {\bf U}^{-1} f(x_1,x_2) = ({\bf T}_{12} f)(x_1,x_2)$ (see \eqref{4_eq:basic_T}, \eqref{4_eq:T_formula}, and \eqref{4_eq:curly_G_formula} for definition of ${\bf T}_{12}$). Thus, ${\bf U}^{-1} {\bf T} {\bf U} = \widetilde{\bf T}$.
\end{proof}

\begin{remark}
Recall ${\bf A} = {\bf A}^{(1)}$, $\widetilde{{\bf A}} = \widetilde{{\bf A}}^{(0)}$, thus ${\bf U}^{-1} {\bf A} {\bf U} \neq \widetilde{{\bf A}}$. Recall also that ${\bf A}^{(-1)}$ was special because $m=-1$ makes $\zeta^{1-m^2}=1$ in the relation \eqref{4_eq:prop_A_m_relations_eq1}. Thus, $({\bf T}, {\bf A}^{(1)})$ and $({\bf T}, {\bf A}^{(-1)})$ are genuine representations (as opposed to projective ones) of the Kashaev group. The operator ${\bf A}^{(0)}$, which corresponds to the the Kashaev's operator, lies ``in between'' ${\bf A}^{(1)}$ and ${\bf A}^{(-1)}$.
\end{remark}

\vs

One might wonder where the modified Heisenberg pair $(-p,-x-2p)$ (acting on the multiplicity modules $M$) come from.
Consider the following elements of the fraction field of $\mathcal{B} \otimes \mathcal{B}$:
$$
Z_1 = 1\otimes X + Y^{-1} \otimes Y, \quad
Z_2 = (X^{-1} \otimes 1 + X^{-1} Y \otimes Y^{-1} X)^{-1},
$$
and also $\widetilde{Z}_1$ and $\widetilde{Z}_2$, obtained by replacing $X,Y$ by $\widetilde{X}, \widetilde{Y}$ in definition of $Z_1$ and $Z_2$, respectively.
It's easy to check that all these four elements commute with $\Delta X$ and $\Delta Y$, and
$$
Z_1 Z_2 = q^2 Z_2 Z_1, \quad
\widetilde{Z}_1 \widetilde{Z}_2 = \widetilde{q}^2 \widetilde{Z}_2 \widetilde{Z}_1, \quad
Z_1^{1/b^2} = \widetilde{Z}_1, \quad Z_2^{1/b^2} = \widetilde{Z}_2,
$$
where the last two relations are as operators (represented via $\pi^{\otimes 2}$).

\vs

By first letting these four elements act on $\mathcal{H} \otimes \mathcal{H}$ by $\pi \otimes \pi$ and then transferring these actions to $M\otimes \mathcal{H}$ via the isomorphism $\mathcal{H} \otimes \mathcal{H} \stackrel{\sim}{\longrightarrow} M \otimes \mathcal{H}$, we find out that these elements act only on $M$ (not on $\mathcal{H}$), and the actions of $Z_1$, $Z_2$, $\widetilde{Z}_1$, $\widetilde{Z}_2$ are $e^{2\pi b p_\alpha}$, $e^{2\pi b (-\alpha-2p_\alpha)}$, $e^{2\pi b^{-1} p_\alpha}$, $e^{2\pi b^{-1} (-\alpha-2p_\alpha)}$, respectively.

\section{Finite-dimensional and universal quantum Teichm\"{u}ller spaces}

\subsection{Quantum Teichm\"{u}ller spaces of Riemann surfaces}

In general one can define a quantization of the Teichm\"{u}ller space of genus $g$ surfaces with $s$ punctures and $r$ boundary components with $(\delta_1, \delta_2, \ldots, \delta_r)$ distinguished points \cite{Penner4}. We'll consider first a simple example of a disk with $n$ distinguished points on the boundary, which we'll view as an $n$-gon. 

\vs

Let the edges of the $n$-gon be enumerated counterclockwise from $0$ to $n-1$, where $i$-th edge corresponds to the representation $\mathcal{H}_i \cong \mathcal{H}$. The special role of the $0$-edge is accounted in the decoration of the $n$-gon with $n-2$ dots near all vertices except the endpoints of the $0$-edge. It is well known that various triangulations of $n$-gon are in one-to-one correspondence with various arrangements of brackets (i.e. parentheses) in the product $\mathcal{H}_1 \otimes \mathcal{H}_2 \otimes \cdots \otimes \mathcal{H}_{n-1}$, where the role of dots is as in Figure \ref{4_fig:one_triangle}. Note that a choice of triangulation also uniquely determines the placement of each dot near every vertex in a particular triangle of the triangulation. For example, the decorated triangulation of a $6$-gon as in Figure \ref{6_fig:six_gon}
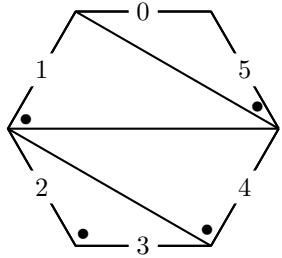
\begin{figure}[htbp!]
\centering
\begin{pspicture}[showgrid=false](0,0.2)(4,3.4)
\rput[bl](0,0){
\PstHexagon[unit=1.8,PolyName=P]
\pcline(P1)(P2)\ncput*{5}
\pcline(P2)(P3)\ncput*{0}
\pcline(P3)(P4)\ncput*{1}
\pcline(P4)(P5)\ncput*{2}
\pcline(P5)(P6)\ncput*{3}
\pcline(P6)(P1)\ncput*{4}
\pcline(P1)(P3)
\rput[l]{132}(P1){\hspace{3,2mm}$\bullet$}
\rput[l]{25}(P4){\hspace{1,8mm}$\bullet$}
\rput[l]{55}(P5){\hspace{1,0mm}$\bullet$}
\rput[l]{100}(P6){\hspace{1,3mm}$\bullet$}
\pcline(P1)(P4)
\pcline(P4)(P6)
}
\end{pspicture}
\caption{An example of a decorated triangulation of a $6$-gon}
\label{6_fig:six_gon}
\end{figure}
corresponds to the following arrangement of brackets
$$
Hom(\mathcal{H}_0, (\mathcal{H}_1 \otimes ((\mathcal{H}_2 \otimes \mathcal{H}_3) \otimes \mathcal{H}_4)) \otimes \mathcal{H}_5).
$$

\vs

By forgetting the brackets, we are now able to identify the quantum Teichm\"{u}ller space of an $n$-gon directly with the space of intertwining operators
\begin{align}
\label{6_eq:space_of_intertwiners}
Hom(\mathcal{H}_0, \mathcal{H}_1 \otimes \mathcal{H}_2 \otimes \cdots \otimes \mathcal{H}_{n-1}).
\end{align}
Clearly, we could consider a dual construction of the quantum Teichm\"{u}ller space of an $n$-gon identifying it with the space of intertwining operators
\begin{align}
\label{6_eq:dual_space_of_intertwiners}
Hom(\mathcal{H}_{n-1} \otimes \mathcal{H}_{n-2} \otimes \cdots \otimes \mathcal{H}_1, \mathcal{H}_0).
\end{align}
The isomorphism of the two quantizations results from isomorphisms $\mathcal{H}_i \cong \mathcal{H}_i'$ for all $i$, and the isomorphisms of $Hom$'s with their duals.

\vs

Combining two pictures together we obtain a quantization of an $n$-gon with a distinguished oriented diagonal rather than an edge. In this case the quantum Teichm\"{u}ller space becomes
\begin{align}
\label{6_eq:combined_space_of_intertwiners}
Hom(\mathcal{H}_{s_m} \otimes \cdots \otimes \mathcal{H}_{s_1}, \mathcal{H}_{r_1} \otimes \cdots \otimes \mathcal{H}_{r_k}),
\end{align}
where $k+m=n$. As in Figure \ref{6_fig:n_gon}, the distinguished diagonal
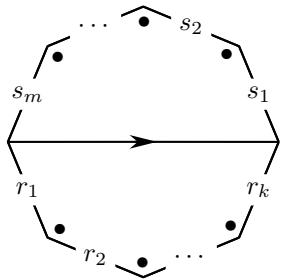
\begin{figure}[htbp!]
\centering
\begin{pspicture}[showgrid=false](0,0)(4,3.7)
\rput[bl](0,0){
\PstPolygon[PolyNbSides=8,unit=1.8,PolyName=P]
\pcline(P1)(P2)\ncput*{$s_1$}
\pcline(P2)(P3)\ncput*{$s_2$}
\pcline(P3)(P4)\ncput*{$\cdots$}
\pcline(P4)(P5)\ncput*{$s_m$}
\pcline(P5)(P6)\ncput*{$r_1$}
\pcline(P6)(P7)\ncput*{$r_2$}
\pcline(P7)(P8)\ncput*{$\cdots$}
\pcline(P8)(P1)\ncput*{$r_k$}
\rput[l]{210}(P2){\hspace{1,1mm}$\bullet$}
\rput[l]{270}(P3){\hspace{1,1mm}$\bullet$}
\rput[l]{310}(P4){\hspace{1,1mm}$\bullet$}
\rput[l]{30}(P6){\hspace{1,1mm}$\bullet$}
\rput[l]{90}(P7){\hspace{1,1mm}$\bullet$}
\rput[l]{120}(P8){\hspace{1,1mm}$\bullet$}
}
\psline[linewidth=0.8pt, 
arrowsize=2pt 4, 
arrowlength=2, 
arrowinset=0.3] 
{->}(P5)(2.1,1.8)
\psline{-}(1.8,1.8)(P1)
\end{pspicture}
\caption{Quantization of an $n$-gon with a distinguished oriented diagonal}
\label{6_fig:n_gon}
\end{figure}
divides the edges into two ordered sets and the orientation of the diagonal defines the order of edges within each of the two sets. Note that the decoration of the $n$-gon with the diagonal has $n-2$ dots near all vertices except the endpoints of this diagonal.

\vs

Constructions of quantum Teichm\"{u}ller spaces for a more general bordered surface with punctures from the representation $\mathcal{H}$ of the quantum plane are similar \cite{Kash98} \cite{GL} \cite{FC}. We will now consider another important example: construction of the quantum universal Teichm\"{u}ller space, which can be viewed as a certain limit of the quantum Teichm\"{u}ller spaces of $n$-gons when $n$ tends to infinity.

\subsection{Quantum universal Teichm\"{u}ller space}

Penner introduced in \cite{Penner2} the ``largest'' version of the universal Teichm\"{u}ller space $Homeo_+(S^1)/M\ddot{o}b(S^1)$ modeled on the group of orientation preserving homeomorphisms of the unit circle factored by the M\"{o}bius group; he also constructed a parametrization of this space using the Farey tessellation (which is a triangulation of the hyperbolic plane; see Figure \ref{6_fig:Farey_tessellation}) with the distinguished oriented edge that connects the $0$-vertex to the $\infty$-vertex.
\begin{figure}[htbp!]
\centering
\begin{pspicture}[showgrid=false,linewidth=0.5pt](-2.6,-2.6)(2.6,2.6)
\psarc(0,0){2.4}{0}{360}
\psarc(-2.4,2.4){2.4}{-90}{0}
\psarc(-2.4,-2.4){2.4}{0}{90}
\psarc(2.4,-2.4){2.4}{90}{180}
\psarc(2.4,2.4){2.4}{180}{-90}
\psarc[arcsep=0.5pt](2.4,1.2){1.2}{142.5}{-90}
\psarc[arcsep=0.5pt](0.8,2.4){0.8}{180}{-40}
\psarc[arcsep=0.5pt](2.4,0.8){0.8}{127}{-90}
\psarc[arcsep=0.5pt](1.714,1.714){0.343}{140}{-51}
\psarc[arcsep=0.5pt](0.48,2.4){0.48}{180}{-22}
\psarc[arcsep=0.5pt](1.2,2.1){0.3}{160}{-34}
\psarc[arcsep=0.5pt](2.4,0.6){0.6}{117}{-90}
\psarc[arcsep=0.5pt](2.03,1.29){0.185}{127}{-60}
\rput(-2.6,0){$\frac{0}{1}$}
\rput(2.6,0){$\frac{1}{0}$}
\rput(0,2.7){-$\frac{1}{1}$}
\rput(1.5,2.2){-$\frac{2}{1}$}
\rput(2.09,1.58){-$\frac{3}{1}$}
\rput(0.87,2.5){-$\frac{3}{2}$}
\rput(2.35,1.15){-$\frac{4}{1}$}
\psarc[arcsep=0.5pt](-2.4,1.2){1.2}{-90}{37.5}
\psarc[arcsep=0.5pt](-0.8,2.4){0.8}{-140}{0}
\psarc[arcsep=0.5pt](-2.4,0.8){0.8}{-90}{53}
\psarc[arcsep=0.5pt](-1.714,1.714){0.343}{-129}{40}
\psarc[arcsep=0.5pt](-0.48,2.4){0.48}{-158}{0}
\psarc[arcsep=0.5pt](-1.2,2.1){0.3}{-146}{20}
\psarc[arcsep=0.5pt](-2.4,0.6){0.6}{-90}{63}
\psarc[arcsep=0.5pt](-2.03,1.29){0.185}{-120}{53}
\rput(-1.55,2.2){-$\frac{1}{2}$}
\rput(-2.03,1.69){-$\frac{1}{3}$}
\rput(-0.89,2.5){-$\frac{2}{3}$}
\rput(-2.33,1.26){-$\frac{1}{4}$}
\psarc[arcsep=0.5pt](-2.4,-1.2){1.2}{-37.5}{90}
\psarc[arcsep=0.5pt](-0.8,-2.4){0.8}{0}{140}
\psarc[arcsep=0.5pt](-2.4,-0.8){0.8}{-53}{90}
\psarc[arcsep=0.5pt](-1.714,-1.714){0.343}{-40}{129}
\psarc[arcsep=0.5pt](-0.48,-2.4){0.48}{0}{158}
\psarc[arcsep=0.5pt](-1.2,-2.1){0.3}{-20}{146}
\psarc[arcsep=0.5pt](-2.4,-0.6){0.6}{-63}{90}
\psarc[arcsep=0.5pt](-2.03,-1.29){0.185}{-53}{120}
\rput(-1.57,-2.13){$\frac{1}{2}$}
\rput(-2.02,-1.68){$\frac{1}{3}$}
\rput(-0.92,-2.5){$\frac{2}{3}$}
\rput(-2.30,-1.18){$\frac{1}{4}$}
\psarc[arcsep=0.5pt](2.4,-1.2){1.2}{90}{-142.5}
\psarc[arcsep=0.5pt](0.8,-2.4){0.8}{40}{-180}
\psarc[arcsep=0.5pt](2.4,-0.8){0.8}{90}{-127}
\psarc[arcsep=0.5pt](1.714,-1.714){0.343}{51}{-140}
\psarc[arcsep=0.5pt](0.48,-2.4){0.48}{22}{-180}
\psarc[arcsep=0.5pt](1.2,-2.1){0.3}{34}{-160}
\psarc[arcsep=0.5pt](2.4,-0.6){0.6}{90}{-117}
\psarc[arcsep=0.5pt](2.03,-1.29){0.185}{60}{-127}
\rput(0,-2.7){$\frac{1}{1}$}
\rput(1.5,-2.2){$\frac{2}{1}$}
\rput(1.98,-1.71){$\frac{3}{1}$}
\rput(0.92,-2.5){$\frac{3}{2}$}
\rput(2.30,-1.18){$\frac{4}{1}$}
\psline[linewidth=0.5pt, 
arrowsize=2pt 4, 
arrowlength=2, 
arrowinset=0.3] 
{->}(-2.4,0)(0.2,0) 
\psline{-}(0,0)(2.4,0)
\end{pspicture}
\caption{The Farey tessellation with distinguished oriented edge}
\label{6_fig:Farey_tessellation}
\end{figure}
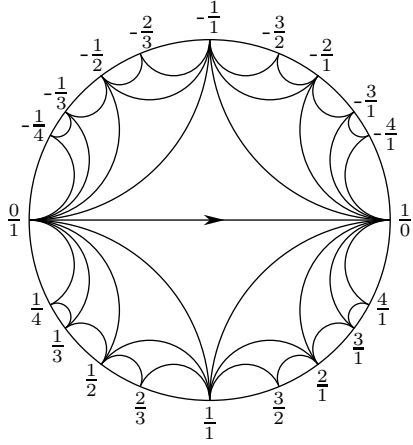

\vs

The vertices of the Farey tessellation are marked by the rational numbers including $\infty$: $\widehat{\mathbb{Q}} = \mathbb{Q} \cup \{\infty\}$. The ideal triangles of the Farey tessellation are in a natural one-to-one correspondence with nonzero rational numbers: $\mathbb{Q}^\times = \mathbb{Q}\setminus \{0\}$. The correspondence is obtained by assigning to a triangle its ``intermediate'' vertex in the upper or lower semicircle.

\vs

To construct the quantum universal Teichm\"{u}ller space we consider the Farey tessellation as an inductive limit of $n$-gons, all containing the oriented edge from $0$ to $\infty$. Suppose that a fixed $n$-gon has vertices at the points $0, \infty$, and $r_2, r_3, \ldots, r_k \in \mathbb{Q}_{>0}$ and $s_2,s_3,\ldots,s_m\in \mathbb{Q}_{<0}$, $k+m=n$. We set $r_1=0, s_1=\infty$, and we'll also mark the edges by the same rational numbers as their initial point with respect to the counterclockwise direction. Then the corresponding quantum Teichm\"{u}ller space will be precisely \eqref{6_eq:combined_space_of_intertwiners}.

\vs

Thus formally quantum universal Teichm\"{u}ller space can be identified with
\begin{align}
\label{6_eq:quantum_universal_Teichmuller_space1}
Hom\left(\bigotimes_{s\in - \mathbb{Q}^{-1}_{\ge 0}} \mathcal{H}_s, ~\bigotimes_{r\in \mathbb{Q}_{\ge 0}} \mathcal{H}_r\right),
\end{align}
where the product is taken in the increasing order of indices in $\mathbb{Q}_{\ge 0}$ and decreasing order of indices in $-\mathbb{Q}^{-1}_{\ge 0} \equiv \mathbb{Q}_{<0} \cup \{\infty\}$. By dualizing all $\mathcal{H}$'s in the first product of \eqref{6_eq:quantum_universal_Teichmuller_space1} we can obtain even more symmetric form of the quantum universal Teichm\"{u}ller space, namely
\begin{align}
\label{6_eq:quantum_universal_Teichmuller_space2}
Inv\left( \bigotimes_{r\in \widehat{\mathbb{Q}}} \mathcal{H}_r\right),
\end{align}
where the factors are ordered in the counterclockwise direction of the circle with the assumed cyclic symmetry.

\vs

The most natural way to make sense of the infinite tensor product in \eqref{6_eq:quantum_universal_Teichmuller_space2} is to realize it as the space of functions on a real Hilbert space
\begin{align}
H = \otimes_{r\in \widehat{\mathbb{Q}}} \mathbb{R} e_r, \quad (e_r, e_s) = \delta_{r,s},
\end{align}
with respect to a Gaussian measure. For any finite set $r_1,\ldots,r_n \in \widehat{\mathbb{Q}}$ we have a finite dimensional orthogonal projection in $H$
\begin{align}
P_{r_1\ldots r_n} : H \to H_{r_1\ldots r_n},
\end{align}
where $H_{r_1\ldots r_n} = \mathbb{R} e_{r_1} \oplus \cdots \oplus \mathbb{R} e_{r_n}$. Then for fixed $t>0$ one can define a Gaussian premeasure $\widetilde{\mu}^t$ on the cylinder sets in $H$ by
\begin{align}
\widetilde{\mu}^t(x\in H, P_{r_1\ldots r_n} x \in F) = (2\pi t)^{-n/2} \int_F e^{-\frac{||x||^2}{2t}} d^nx,
\end{align}
where $F$ is a Borel subset of $H_{r_1 \ldots r_n}$. To construct an actual Gaussian measure $\mu^t$ one needs to complete $H$ to a Banach space $B$ with respect to a measurable norm $| \cdot |$, see \cite{Kuo}, and extend $\widetilde{\mu}^t$ to cylinder sets in $B$. Then $\mu^t$ is $\sigma$-additive in the $\sigma$-field generated by cylinder sets. This completes our construction of the quantum universal Teichm\"{u}ller space. Note that it is invariant under the action of the modular group, which acts naturally on the basis of $H$.

\vs

One can look for an intrinsic geometric realization of the space $H$ related to the Farey tessellation that determines the coordinates of the classical Teichm\"{u}ller space. Penner in \cite{Penner3} proposed wavelet bases for the space of all complex-valued functions defined on the unit circle $S^1$. These wavelet bases are naturally indexed by the edges of the Farey tessellation, which are in one-to-one correspondence with $\widehat{\mathbb{Q}}\setminus\{0,1,\infty\}$, and three more real coordinates that parametrize the Lie algebra of the M\"{o}bius group. This brings us very close back to more traditional quantizations of two ``smaller'' versions of the universal Teichm\"{u}ller spaces $Diff_+(S^1)/M\ddot{o}b(S^1)$ \cite{KiY} and $QS_+(S^1)/M\ddot{o}b(S^1)$ \cite{NS}, modeled on the groups of orientation preserving diffeomorphisms and quasi-symmetric homeomorphisms of the unit circle.

\vs

Our construction, however, suggests one more (at this moment only heuristic) candidate for the quantum universal Teichm\"{u}ller space. It comes out from another remarkable appearance of the rational numbers as a natural index set in the classification of the projective modules for the quantum torus \cite{Connes}, which can be viewed as a unitary counterpart of the quantum plane. The projective modules $\mathcal{E}_r, r\in \widehat{\mathbb{Q}}$, can be extended to bimodules over another quantum torus with
\begin{align}
\label{6_eq:modular_q}
\widetilde{q}^2 = e^{2\pi i \frac{a\tau + b}{c\tau +d}}, \quad r= \frac{d}{c},
\end{align}
where $\left(\begin{smallmatrix} a & b\\ c& d\end{smallmatrix} \right) \in PSL_2(\mathbb{Z})$ is a representative of a coset in $\Gamma_\infty \backslash PSL_2(\mathbb{Z}) \cong \widehat{\mathbb{Q}}$, and $\Gamma_\infty$ is the subgroup of upper triangular matrices. These bimodules give rise to a Morita equivalence of quantum tori for all $\widetilde{q}$ in \eqref{6_eq:modular_q} \cite{Ri}, \cite{Connes}. For $r=0$ we obtain the modular double as the one studied in our paper, for $r=\infty$, $\mathcal{E}_\infty$ is the free module, i.e. the quantum torus algebra itself. The coincidence mentioned above leads to another version of the quantum universal Teichm\"{u}ller space in the unitary case, namely
\begin{align}
\label{6_eq:quantum_universal_Teichmuller_space3}
Inv\left( \bigotimes_{r\in \widehat{Q}} \mathcal{E}_r\right),
\end{align}
where the tensor product can be defined as in the case of the quantum plane though the unitarity is no longer preserved. Since $\mathcal{E}_\infty$ is the quantum torus algebra itself \eqref{6_eq:quantum_universal_Teichmuller_space3} simplifies to just the tensor product of all projective, but not free, modules
\begin{align}
\label{6_eq:quantum_universal_Teichmuller_space4}
\otimes_{r\in \mathbb{Q}} \mathcal{E}_r.
\end{align}
Returning back to the quantum plane, we conjecture that the Morita equivalence for $\widetilde{q}$ as in \eqref{6_eq:modular_q} still holds, in an appropriate sense, and the tensor product of the corresponding bimodules as in \eqref{6_eq:quantum_universal_Teichmuller_space4} gives rise to another quantization of the universal Teichm\"{u}ller space.

\end{document}